\documentclass[11pt,reqno]{amsart}
\usepackage[utf8]{inputenc} 
\usepackage[T1]{fontenc}
\usepackage{hyperref}
\usepackage{amssymb,amsfonts,amsthm,amscd,stmaryrd,dsfont,esint,upgreek,constants,todonotes,mathdesign,mathrsfs,mathtools,xparse,enumitem}
\usepackage{./BOONDOX-calo}
\usepackage[mathcal]{euscript}
\usepackage{graphicx,color}
\usepackage[a4paper,left=2.5cm,right=2.5cm,top=3cm,bottom=3cm]{geometry}
\numberwithin{equation}{section}
\newcommand{\N}{\mathbb N}
\newcommand{\Z}{\mathbb Z}
\newcommand{\R}{\mathbb R}
\def\E{\mathbb E}
\def\P{\mathbb P}
\DeclareMathOperator*{\esssup}{ess-sup}

\newconstantfamily{C}{symbol=C}
\newconstantfamily{c}{symbol=c}
\newconstantfamily{O}{symbol=\Omega}

\def\XXint#1#2#3{{\setbox0=\hbox{$#1{#2#3}{\int}$}
\vcenter{\hbox{$#2#3$}}\kern-.5\wd0}}

\newcommand{\T}{\mathbb{T}}

\numberwithin{equation}{section}
\newtheorem{thm}{Theorem}[section]
\newtheorem{lem}[thm]{Lemma}
\newtheorem{cor}[thm]{Corollary}
\newtheorem{prop}[thm]{Proposition}

\theoremstyle{definition}
\newtheorem{defn}[thm]{Definition}
\newtheorem{rmk}[thm]{Remark}

\def\smallnegint{\mathop{\int\mkern-13mu
        \raise.5ex\hbox{${\scriptscriptstyle\diagup}$}}\nolimits}
\def\ds{\displaystyle}

\def\ep{\varepsilon}

\def\div{\operatorname{div}}

\def\ssetminus{\,\raise.4ex\hbox{$\scriptstyle\setminus$}\,}
\def \lg{\langle}
\def \rg{\rangle}
\newcommand{\be}{\begin{equation}}
\newcommand{\ee}{\end{equation}}

\def\inte{\int_{\T^d}}
\def\dive{{\rm div}}

\newcommand{\normU}[2]{\left\Vert#1\right\Vert_{#2}}
\newcommand{\normUU}[2]{\Vert#1\Vert_{#2}}

\renewcommand{\bar}{\overline}
\renewcommand{\tilde}{\widetilde}
\renewcommand{\hat}{\widehat}

\setlist[enumerate]{leftmargin=0pt}

\def\shif{\mathbcal}         % for shifting with the white noise

        \def\uT{u_t^T}
        \def\uTW{\shif{u}_t^T}
        \def\vT{v_t^T}

        \def\MTW{\shif{M}_t^T}
        \def\mT{m_t^T}
        \def\mTW{\shif{m}_t^T}
        \def\uS{\bar u_t}
        
        \def\mS{\bar m_t}
        
        \def\vS{\bar v_t}

        \def\uD{u_t^\delta}
        \def\uDW{\shif{u}_t^\delta}
        \def\vD{v_t^\delta}
        \def\vDW{\shif{v}_t^\delta}
        
        \def\MDW{\shif{M}_t^\delta}
        \def\mD{m_t^\delta}
        \def\mDW{\shif{m}_t^\delta}
        \def\zTW{\shif{z}_t^T}
        \def\NTW{\shif{N}_t^T}
        \def\rTW{\boldsymbol{\varrho}_t^T}
        \def\zDW{\shif{z}_t^\delta}
        \def\zDWT{\shif{z}_t^{\delta,T}}
        \def\NDW{\shif{N}_t^\delta}
        \def\NDWT{\shif{N}_t^{\delta,T}}
        \def\rDW{\boldsymbol{\varrho}_t^\delta}
        \def\rDWT{\boldsymbol{\varrho}_t^{\delta,T}}

\def\cwn{\sigma} % cwn for common white noise

    \DeclarePairedDelimiter{\normC}{\lVert}{\rVert_{C^0}}
    \DeclarePairedDelimiter{\normH}{\lVert}{\rVert_{C^\beta}}
    
    \DeclarePairedDelimiter{\normD}{\lVert}{\rVert_{C^1}}
    
    \DeclarePairedDelimiter{\normDDH}{\lVert}{\rVert_{C^{2+\beta}}}
    \DeclarePairedDelimiter{\normDDHd}{\lVert}{\rVert_{\left(C^{2+\beta}\right)'}}
    \DeclarePairedDelimiter{\normtwo}{\lVert}{\rVert_{L^2_{\omega,x}}}
    \DeclarePairedDelimiter{\normtwox}{\lVert}{\rVert_{L^2_{x}}}

\newcommand{\EP}[1]{\mathbb{E}\left[ #1 \right]}
\newcommand{\EPt}[1]{\mathbb{E}\left[ #1 \middle| \mathcal{F}_t \right]}

\allowdisplaybreaks

\begin{document}
\title{On the long-time behavior of Mean Field Game systems with a common noise}

\author{Pierre Cardaliaguet}
\address{Universit\'e Paris-Dauphine \\ Place du Mar\'echal de Lattre de Tassigny \\ 75016 Paris, France}
\email{cardaliaguet@ceremade.dauphine.fr}

\author{Rapha\"{e}l Maillet}
\address{
Capital Fund Management (CFM) \\ 23-25 rue de l'Universit\'e\\ 75007 Paris \and 
Universit\'e Paris-Dauphine \\ Place du Mar\'echal de Lattre de Tassigny \\ 75016 Paris, France}
\email{maillet@ceremade.dauphine.fr}

\author{Wenbin Yan}
\address{Universit\'e Paris-Dauphine \\ Place du Mar\'echal de Lattre de Tassigny \\ 75016 Paris, France}
\email{wenbin.yan@dauphine.edu}

\dedicatory{Version: \today}

%\begin{abstract} \end{abstract}

\maketitle      %% SHORT TITLE FOR RUNNING HEAD

\begin{abstract}
In this paper, we study the long-time behavior of mean field game (MFG) systems influenced by a common noise. While classical  results establish the convergence of deterministic MFG towards stationary solutions under suitable monotonicity conditions, the introduction of a common stochastic perturbation significantly complicates the analysis. We consider a standard MFG model with infinitely many players whose dynamics are subject to both idiosyncratic and common noise. The central goal is to characterize the asymptotic properties as the horizon goes to infinity. By employing quantitative methods that replace classical compactness arguments unavailable in the stochastic context, we prove that solutions exhibit exponential convergence toward a stationary regime. Specifically, we identify a deterministic ergodic constant and demonstrate the existence of stationary random processes capturing the limiting behavior. Further, we establish almost sure long-time results thanks to a detailed analysis of the ergodic master equation, which is the long-time limit of the master equation. Our results extend known deterministic convergence phenomena to the stochastic setting, relying on novel backward stochastic PDE estimates. 

\end{abstract}

\tableofcontents

%
%
%
%\begin{table}[h]
%    \centering
%    \begin{tabular}{|c|c|c|c|c|c|c|c|}
%    $\uT$ & $\uTW$ & $\mT$ & $\mTW$ & $\vT$ & $\vTW$ & $\MT$ & $\MTW$ \\
%    $\uD$ & $\uDW$ & $\mD$ & $\mDW$ & $\vD$ & $\vDW$ & $\MD$ & $\MDW$ \\
%    $\zT$ & $\zTW$ & $\rT$ & $\rTW$ & $\tT$ & $\tTW$ & $\NT$ & $\NTW$ \\
%    $\zD$ & $\zDW$ & $\rD$ & $\rDW$ & $\tD$ & $\tDW$ & $\ND$ & $\NDW$ \\
%    $\uS$ & $\uSW$ & $\mS$ & $\mSW$ & $\vS$ & $\vSW$ & $\MS$ & $\MSW$ \\
%    $\uDS$ & $\uDSW$ & $\mDS$ & $\mDSW$ & $\vDS$ & $\vDSW$ & $\MDS$ & $\MDSW$ \\
%    \end{tabular}
%    \caption{Notations}
%    \label{tabnotations}
%\end{table}
%
%$\normC*{f}$, $\normH*{f}$, $\normHH*{f}$, $\normD*{f}$, $\normDH*{f}$, $\normDDH*{f}$, $\normU{f}{2}$, $\normtwo*{f}$, $\normtwox*{f}$
%
%Suggestion for notations: for any $\varphi\in C^0(\T^d)$, we may write $\hat{\varphi}$ for $\int_{\T^d} \varphi(x) dx$ and $\tilde{\varphi}$ for $\varphi(x) -\hat{\varphi}$, since we may use $\lg\cdot\rg$ for quadratic variation.\\
%
%$L^2_{\omega,x}$ 
%
% {\color{green} HERE WE NEED TO CLASSIFY THE PROBABILITY SPACE, and as we use \cite{CDLL}, this filtration has to be the Brownian one. Note that here for a Brownian motion with the past, the martingale representation theorem indeed still work. But we have to say a few words for that.}

%%%%%%%%%%%%%%%%%%%%%%%%%%%%%%%%%%%%
\section{Introduction}

Mean field games model  Nash equilibria in differential games with infinitely many small controllers. In the simplest case, they can be reduced to  the coupled  equations: 
\begin{equation}\label{eq.MFGdeterministic}
\left\{
\begin{array}{cl}
     (i)  & -\partial_t u^T- \Delta u^T +\frac12 |D u^T|^2 -f(x,m^T) =0, \\ [1ex]
     (ii) & \partial_t m^T- \Delta m^T -{\rm div}(m^T D  u^T)=0, \\ [1ex]
     (iii)& m^T(0)= \hat m_0, \qquad u^T(T,x)= g(x, m^T(T)).
\end{array}
\right.
\end{equation}
Here $T$ is the horizon, the quantity $m^T(t)$ is a probability measure and represents the distribution at time $t\in[0,T]$ of the controllers  ($\hat m_0$ being the initial distribution); $u^T$ is the value function of a representative small controller/player. The coupling functions $f$ and $g$, which depend on the current state $x$ and the current distribution $m(t)$, describe how players interact.  The MFG models---introduced by Lasry and Lions \cite{LL07mf} (see also \cite{HMC})---are largely motivated by economic applications (see for instance \cite{ABLLM, AHLLM} and the references therein). It is classical in the economic literature to focus on stationary solutions, i.e., solutions independent of time. In fact, it has been observed numerically that, under conditions discussed below, the solution of \eqref{eq.MFGdeterministic} stabilizes  very quickly  to a stationary state (see, e.g., \cite{AcLa20}). This phenomenon, also called the turnpike behavior, was first pointed out by Lions in his courses at Coll\`ege de France \cite{LLperso} and then sharpened in several set-ups in \cite{CLLP2, CiPo21, GMS}. The main result of \cite{CLLP2} states that---under suitable monotonicity condition on $f$ and $g$ (see assumption \eqref{hyp:mono} below)---there exists a constant $\bar \lambda$, depending on $f$ only,  such that, for any $t\geq 0$,  
$$
\lim_{T\to \infty} \frac{u^T(t, \cdot)}{T} = \bar \lambda, \quad \|Du^T(t, \cdot)-D\bar u(\cdot)\|_{L^2} +{\bf d}_1(m^T(t) ,\bar m) \leq C(e^{-\omega t}+e^{-\omega(T-t)} ), 
%\lim_{T\to \infty} u^T(0, x)-\bar\lambda T =\bar u(x) +c,
$$
where $\omega, C>0$ are constants and $(\bar u,\bar m,\bar \lambda)$ solves the ergodic MFG system 
\be\label{eq.MFGergodeterministic}
\left\{
\begin{array}{cl}
     (i)  & \bar \lambda- \Delta \bar u +\frac12 |D  \bar u|^2 -f(x,\bar m) =0, \\ [1ex]
     (ii) & - \Delta \bar m -{\rm div}(\bar m D   \bar u)=0. 
\end{array}
\right.
\ee
Here ${\bf d}_1$ is the Wasserstein distance on the set of probability measures. This {\it long-time average} result was further refined in \cite{CP19}, see also \cite{CiPo21}, into a {\it long-time behavior}, where the limit of $u^T(t, \cdot)- \bar \lambda (T-t)$ was characterized in terms of a solution of the ergodic master equation. Note that this later result is the counterpart, for the MFG system, of the weak KAM theory developed for  Hamilton-Jacobi equations in \cite{Fathi1, Fathi2} or \cite{BaSo00}. In \cite{CLLP2, CP19}  the long-time behavior/average of the MFG system \eqref{eq.MFGdeterministic} are studied on the flat torus $\T^d=\R^d/\Z^d$ and under suitable monotonicity conditions on the coupling functions. These results have known several variants as in \cite{BCGL, CCMW, CaMe21, CCDE24, CiMez24,  EJP25, MiMu24,Mu24,Por18} and have been used for numerical approximation in \cite{CaZe24}. See also  \cite{BaKo, CM20, CeCi21, CeCi24, GMP23, Ma19} for more delicate problems in which the monotonicity assumption on the coupling functions is dropped and \cite{KMR24, KMR25, PoRo22} for mean field games with specific structure. Let us underline that all these results hold for  ``deterministic'' MFG systems, in which there is only individual noises for the small players, responsible for the diffusion terms in \eqref{eq.MFGdeterministic} and \eqref{eq.MFGergodeterministic}.   \\

The goal of the present paper is to understand what becomes of this picture when the dynamics of the small players are in addition perturbed by a noise which is common to all the players. In this case, the MFG system \eqref{eq.MFGdeterministic} becomes a system of stochastic partial differential equations and the associated master equation involves second order derivatives in the space of measures. The MFG with a common noise and the associated master equation have attracted a lot of attention in the recent years and existence and uniqueness of solutions can be found in different set-ups \cite{Ahu16,  Be23, CCP22, CDLL,  CaSo22, CDbook, CaDeLa16, Dia25,GMMZ22, KoTr, MoZh24}. The limit of $N-$player games have been treated in this framework in \cite{CDLL, CDbook,Dje23,LaLF23}, while the intriguing question of the regularization effect of the common noise is discussed in \cite{CeDe22, De19, DM25, DeTc20, DeVa25}. However, to the best of our knowledge, the natural question of the long-time behavior of the resulting system has never been addressed before. We focus here in the more standard case where the common noise $W$ is a $d-$dimensional Brownian motion. We are thus interested in the MFG system, set in $(0,T)\times \T^d$:  
\begin{equation} \label{eq.MFGCNintro}
\left\{
\begin{array}{cl}
    (i) & d\uT = \left(-(1+\cwn)\Delta \uT +\frac12 |D \uT|^2 -f(x, \mT) -2 \cwn {\rm div}(\vT) \right)dt + \sqrt{2 \cwn} v^T_t\cdot dW_t, \\ [1ex]
    (ii) & d\mT = \left((1+ \cwn)\Delta \mT +{\rm div} (\mT D \uT) \right)dt -\sqrt{2 \cwn} D \mT\cdot dW_t, \\ [1ex]
    (iii) & m^T_{0} = \hat m_0, \qquad u^T_T(x) = g(x, m^T_T).
\end{array}
\right.
\end{equation}
 As above, $T$ is the finite horizon of the problem and $\hat m_0$ is the initial distribution of the players. The map $\uT$, which solves the backward stochastic Hamilton-Jacobi equation \eqref{eq.MFGCNintro}-(i) with the terminal condition $u^T_T = g(\cdot, m^T_T)$, is the value function of a typical small player: 
$$
\uT(x)= \inf_{\alpha} \E\left[ \int_t^T ( \frac12 |\alpha_s|^2 +f(X^\alpha_s, \mT)) ds+g(X^\alpha_T, m^T_T)\; \bigg|\; \mathcal F_t\right], 
$$
where $X^\alpha_s = x+\int_t^s \alpha_rdr + \sqrt{2} B_s+ \sqrt{2 \cwn}W_s$ and where the infimum is taken over all controls $\alpha\in L^2$. Here $B$ and $W$ are two independent Brownian motions, $B$ being interpreted as the individual noise of a typical player while $W$ is the common noise which affects all the players. The positive constant $\cwn$ indicates the influence of the common white noise $W$. The filtration $(\mathcal F_t)$ is the filtration generated by $W$. The random process $v$ in \eqref{eq.MFGCNintro}-(i) ensures the solution $u$ to be adapted to the filtration $(\mathcal F_t)$, in the line of the theory developed by Peng \cite{Pe92}. 
The random flow of probability measures $(m^T_t)$, adapted to the common noise $W$ and thus to the filtration $(\mathcal F_t)$, can be interpreted as the conditional distribution of players when they play the optimal drift $\alpha_t(x)=-D\uT(x)$, given the common noise $W$. This conditional flow thus solves the forward random Fokker-Planck equation \eqref{eq.MFGCNintro}(ii), with initial condition $\hat m_0$. \\

Under the assumption that the coupling functions $f$ and $g$ are monotone in the Lasry-Lions sense \cite{LL07mf} (see assumption \eqref{hyp:mono} below), the existence and the uniqueness of the solution to \eqref{eq.MFGCNintro}  has been first discussed  in \cite{CDLL} (see also \cite{CDbook} for a probabilistic approach of the problem). The construction of a solution of \eqref{eq.MFGCNintro} in \cite{CDLL} relies  on the existence of a solution to the so-called master equation: 
\begin{align}\label{eq.masterintro}
&- \partial_t U^T(t,x,m)-(1+ \cwn) \Delta U^T(t,x,m)+\frac12|D  U^T(t,x,m)|^2 - f(x,m) \notag \\
& -(1+ \cwn)\inte \dive_y(D_m U^T(t,x,m,y))m(dy) + \inte D_m U^T(t,x,m,y) \cdot D  U^T(t,y,m) m(dy) \\ 
&  - 2 \cwn \inte \dive_x(D_m U^T(t,x,m,y))m(dy) - \cwn \inte\inte {\rm Tr}(D^2_{mm} U^T(t,x,m,y,z))m(dy)m(dz) = 0 \notag,
\end{align}
This equation is stated on $[0,T]\times \T^d\times \mathcal P(\T^d)$ and has the terminal condition $U^T(T,x,m) = g(x,m).$ The  derivatives $D_mU$, $D^2_{mm}U$ are derivatives in the measure variable, as discussed in \cite{CDLL, CDbook}. The solution $U^T$ of the master equation is related with the solution $(u^T, m^T, v^T)$ of the MFG system through the identity 
\be\label{okjzneds22}
U^T(t,x,m^T_t)= u^T_{t}(x),
\ee
and actually $m^T$ solves the random McKean-Vlasov equation
$$
dm^T_t = ((1+ \cwn)\Delta m^T_t + {\rm div}(m^T_t D_xU^T(t,x,m^T_t)))dt -\sqrt{2 \cwn}Dm^T_t \cdot dW_t
$$
with initial condition $m^T_0=\hat m_0$. \\ 

The goal of this paper is to understand the behavior of the solution to the stochastic MFG system \eqref{eq.MFGCNintro} and to the solution to the master equation \eqref{eq.masterintro} as $T\to \infty$. Let us note that, compared to the deterministic MFG system \eqref{eq.MFGdeterministic}, we cannot expect here the existence of deterministic objects as the solution $(\bar \lambda, \bar u, \bar m)$ of \eqref{eq.MFGergodeterministic}: some randomness related with the common noise must remain at the limit. This makes the problem of course  more challenging than the deterministic one.\\

We now present our main results. Let us first explain what becomes of the ergodic MFG system \eqref{eq.MFGergodeterministic}. We show that there is a deterministic constant $\bar \lambda$ for which the random MFG system, set on the whole time interval $\R$, 
$$
\left\{
\begin{array}{cl}
    (i) & d\bar u_t = \left(-(1+\cwn)\Delta \bar u_t +\bar \lambda + \frac12 |D \bar u_t|^2 -f(x, \bar \mu_t) -2 \cwn {\rm div}(\bar v_t) \right)dt + \sqrt{2 \cwn} \bar v_t\cdot dW_t, \\ [1ex]
    (ii) & d\bar \mu_t = \left((1+ \cwn)\Delta \bar \mu_t +{\rm div} (\bar \mu_t D \bar u_t) \right)dt -\sqrt{2 \cwn} D \bar \mu_t\cdot dW_t, \\ [1ex]
\end{array}
\right.
$$
has a {\it stationary} solution $(\bar u, \bar \mu, \bar v)$ which ``attracts'' all the solution to \eqref{eq.MFGCNintro} in long-time. By stationary, we roughly mean that the law of solution $(\bar u_{\cdot +h}, \bar \mu_{\cdot +h}, \bar v_{\cdot +h})$ does not depend on $h\in \R$.  This solution attracts the solution of the time dependent problem in the following sense (see Theorem \ref{thm.main}): let $(u^T,m^T,v^T)$ solves \eqref{eq.MFGCNintro} for some initial condition $\hat m_0$. Then 
\be\label{okjzneds1}
\lim_{T\to \infty} \frac{u^T_t(x)}{T}=\bar \lambda \qquad \text{a.s.}\qquad  \forall t\geq 0,  
\ee
and
\be\label{okjznedsBIS1}
\E\left[\|D\uT -D\bar u_t\|_{L^2(\T^d)}\right] +\E\left[ {\bf d}_1(\mT, \bar \mu_t)\right] \leq C(e^{-\omega t}+e^{-\omega(T-t)}) \qquad \forall t\in [0,T], 
\ee
where $C, \omega>0$ depend on the running cost $f$ and some on some bounds on the terminal cost $g$, but not on the initial condition $\hat m_0$. The constant $\bar \lambda$ is often called the ergodic constant, and can be interpreted as the average payoff of the game in long run. Note that \eqref{okjzneds1} and \eqref{okjznedsBIS1} imply that the constant $\bar \lambda$ and the stationary processes $D\bar u$ and $\bar \mu$ are unique; the relationship between the ergodic constant $\bar \lambda$ and the process $(\bar u, \bar \mu, \bar v)$ is the equality: 
$$
\bar \lambda = - \E\left[  \inte (\frac12 |D\bar u_0(x)|^2 -f(x,\bar \mu_0))dx\right],
$$
which can be formally obtained by integrating in space and expectation the equation satisfied by $\bar u$ and use the stationarity of $\bar u$ and $\bar m$. Note that, even in the uncoupled setting, (i.e., when $f$ and $g$ do not depend on $m$) the result given above seems to be new for the Hamilton-Jacobi equation: we are only aware of the work \cite{GuMa09} which studies  ergodic linear-quadratic control problems with random non Markovian coefficient, with an approach through Riccati equations.  On the other hand, the convergence of the solution to the Fokker-Planck equation with a random perturbation in the uncoupled setting has been recently addressed in \cite{DMT24, Ma23}. \\

Our results also describe the behavior of $u^T_t-\bar \lambda t$: it involves the long-time behavior of the solution $U^T$ of  the master equation \eqref{eq.masterintro}. We prove  in Theorem \ref{thm.CVU} that $U^T(0, \cdot, \cdot)-\bar \lambda T$ converges, as $T\to\infty$, to a solution $\chi$ of the ergodic master equation \begin{align*}
&\bar \lambda-(1+ \cwn) \Delta \chi(x,m)+\frac12|D   \chi(x,m)|^2 - f(x,m) \notag \\
& -(1+ \cwn)\inte \dive_y(D_m  \chi(x,m,y))m(dy) + \inte D_m  \chi(x,m,y) \cdot D_x  (y,m) m(dy) \\ 
&  - 2 \cwn \inte \dive_x(D_m  \chi(x,m,y))m(dy) - \cwn \inte\inte {\rm Tr}(D^2_{mm}  \chi(x,m,y,z))m(dy)m(dz) = 0 \notag. 
\end{align*}
This equation is stated on $\T^d\times \mathcal P(\T^d)$. The solution $\chi$ (often called a {\it corrector}) is unique up to an additive constant. It is built in Theorem \ref{lem.existence_weak_solution_cell_Master_Equation}. Let us underline the relationship between the stationary process $(\bar u,\bar \mu, \bar v)$ introduced above  and the map $\chi$: we show in Theorem \ref{prop.fullcorrector} that $\bar \mu$ is a stationary solution to the McKean-Vlasov equation:
$$
d\bar \mu_t= ((1+\cwn) \Delta \bar \mu_t + \dive(\bar \mu_t D  \chi(x, \bar \mu_t))dt - \sqrt{2\cwn} D  \bar \mu_t\cdot dW_t, \qquad t\in \R, 
$$
while $\bar u$ is given (up to an additive constant) by
$$
\bar u_t(x)= \chi(x, \bar \mu_t)\qquad \forall (t,x)\in \R\times \T^d.
$$
The convergence of $U^T(0)-\bar\lambda T$ entails the convergence of $u^T_t-\bar \lambda t$, where $(u^T,m^T,v^T)$ solves \eqref{eq.MFGCNintro}: we prove in Corollary \ref{cor.uT-lambda}  the existence of a real constant $c$, independent of the initial condition $\hat m_0$, but depending on the terminal cost $g$, such that 
$$
\Big[ \uT(x) -\bar \lambda (T-t) \Big] - \Big[ \chi( x, \bar m_t) + c \Big] \xrightarrow[T\to\infty]{} 0 \qquad \text{a.s.},
$$
locally uniformly in $(t,x)$, where $\bar m_t$ solves the same McKean-Vlasov as $\bar \mu$, but with the initial condition $\hat m_0$: 
$$
d\mS = \left((1+\cwn)\Delta \mS +{\rm div}(\mS D_x \chi( x, \mS))\right)dt -\sqrt{2\cwn} D \mS \cdot dW_t,\qquad \bar m_0= \hat m_0.
$$
It is also possible to describe the behavior of $u^T_t$ as $T$ and $t$ are both large: for instance we show that, for any $\theta\in(0,1)$,
\[
\Big[ u^T_{\theta T}(x) - (1-\theta) \bar \lambda T \Big] - \Big[ \chi( x, \bar \mu_{\theta T}) + c \Big] \xrightarrow[T\to\infty]{} 0 \qquad \text{in $L^2_\omega$}. 
\]

Let us finally note that our results also apply to the discounted MFG system with a common noise: 
\be\label{eq.MFGCNdiscintro}
\left\{
\begin{array}{l}
  du^\delta_t= \left(-(1+\cwn)\Delta u^\delta_t +\delta u^\delta+\frac12 |Du^\delta_t|^2 -f(x,m^\delta_t) -2\cwn{\rm div}(v^\delta_t)\right)dt +\sqrt{2\cwn} v^\delta_t\cdot dW_t, \\ [1ex]
   dm^\delta_t=\left((1+\cwn)\Delta m^\delta_t +{\rm div}(m^\delta_tDu^\delta_t)\right)dt -\sqrt{2\cwn}Dm^\delta_t\cdot dW_t \\ [1ex]
    m^\delta_0= m_0,  \; u^\delta\;{\rm bounded}.
\end{array}
\right.
\ee
We show, in a kind of Tauberian theorem, that the small discount behavior of the solution $(u^\delta, m^\delta, v^\delta)$ can be described by $\bar \lambda$ and $(\bar u, \bar \mu, \bar v)$  (see Theorem \ref{thm.Cudelta}).
\bigskip

Let us now say  a few words about proofs. As in \cite{CLLP2,CP19}, the starting point is a uniform Lipschitz estimate for $u^T$ and uniform bounds below and above for $m^T$, where $(u^T, m^T, v^T)$ solves \eqref{eq.MFGCNintro}:
$$
\sup_{t\in [0,T]} \|Du^T_t\|_\infty <\infty, \qquad C^{-1}\leq m^T_t \leq C \; {\rm for }\; t\in [1, T], \qquad a.s.
$$
These uniform estimates combined with Lasry-Lions monotonicity argument imply the existence of constants $\omega,C>0$ such that  any two solutions $(u^1, m^1, v^1)$ and $(u^2, m^2, v^2)$ to \eqref{eq.MFGCNintro}-(i)-(ii) on a time interval $[t_0,T]$ satisfy:
$$
\E\left[\|m^1_t-m^2_t\|_{L^2_x}+ \|Du^1_t-Du^2_t\|_{L^2_x}\right] \leq C(e^{-\omega (t-t_0)}+e^{-\omega(T-t)}) \qquad \forall t\in [t_0+2, T-2].
$$ 
Using this on larger and larger time intervals gives the existence of a stationary process $(D\bar u, \bar \mu)$ such that  \eqref{okjznedsBIS1} holds. Note that, at this stage, we only build  $D\bar u$; the construction of the stationary process $\bar u$ is much more involved and only comes at the very end of the paper. We then derive \eqref{okjzneds1} by a comparison principle between the Hamilton-Jacobi equation satisfied by $u^T$ and the one satisfied by a (non-stationary) proxi of $\bar u$ (see Theorem \ref{thm.main}). We then discuss the behavior of the discounted MFG system \eqref{eq.MFGCNdiscintro} as the discount rate tends to $0$ (Theorem \ref{thm.Cudelta}). 

The remaining part of the paper is devoted to the construction of the corrector and the limit of the time dependent master equation. The key part is the analysis on long-time of a linearized random MFG system (Theorems \ref{thm.good_bound_for_linearized_system_finite} and \ref{thm.good_bound_for_linearized_system_discounted}). These results provide estimates on the various measure derivatives of the solution to the time dependent master equation and of the infinite horizon master equation, and eventually allow to build the solution $\chi$ to the ergodic master equation (Theorem \ref{lem.existence_weak_solution_cell_Master_Equation}). If the outline of these proofs follows that of \cite{CP19}, the analysis turns out to be  much more intricate. Indeed the manipulation of the random MFG system \eqref{eq.MFGCNintro} is  much more challenging than the deterministic one studied in \cite{CP19}. For instance, many results of \cite{CP19} rely on compactness argument, which are not available in the random setting: they have to be replaced by quantitative ones, as in the proofs of Theorem \ref{thm.main} and Theorem \ref{thm.good_bound_for_linearized_system_discounted}. In addition, the regularizing aspects of the deterministic parabolic equations rely in \cite{CP19} on Schauder estimates, which are not available in the random setting for lack of regularity in time. Finally, the existence of deterministic constants at the limit are often not obvious in the random setting and have to rely on probabilistic arguments (ergodic theorem or Kolmogorov 0-1 law, for instance). As a consequence, the complete construction of the stationary solution $(\bar u, \bar \mu, \bar v)$ is  given only at the end on the paper, after the construction of the solution of the ergodic master equation (Theorem \ref{prop.fullcorrector}). 
\bigskip

Let us finally conclude with possible extensions: we have chosen in this paper to work in the flat torus $\T^d$ and with a quadratic Hamiltonian. We largely do so to simplify the notation and several  estimates. We believe however that the results still hold without much changes on a compact manifold and with a general  uniformly convex Hamiltonian. The problem on the whole space and with possibly slightly non-monotone  couplings is more challenging and will be addressed in  future works. \\

{\it Organization of the paper:} After Section \ref{Sec:NotHyp}, dedicated to the notation and the assumptions, we address in Section \ref{sec.timedep} the long-time average of the solution to the time dependent MFG system \eqref{eq.MFGCNintro} and in Section \ref{sec.discount} the limit  of the solution to the discounted MFG system \eqref{eq.MFGCNdiscintro} as the discount rate vanishes. In the intermediate and technical Section \ref{sec.linear}, we obtain   estimates for linearized MFG systems on arbitrary time intervals: these estimates are instrumental to obtain uniform in $m$ estimates for solution of the time dependent master equation and of the discounted master equation. We derive from them the existence of a corrector (Section \ref{sec.ergo}) and  the long-time behavior of the master equation (Section \ref{sec.longtimemaster}). This long-time behavior finally allows to describe the long-time behavior of the solution of the MFG system. We collect in the appendix  decay estimates for solutions of forward and backward linear random parabolic systems and conclude with new exponential decay of Gronwall type used in the main results. 

%%%%%%%%%%%%%%%
\section{Notation and assumptions} \label{Sec:NotHyp}

Throughout the paper, we fix a probability space $(\Omega,\P,\mathcal{F})$ supporting a canonical two-sided Brownian motion $(W_t)_{t\in \R}$, defined as follows:
\begin{itemize}
    \item[(i)] $\Omega = C_0(\mathbb{R}, \mathbb{R}^d)$, the space of all continuous functions $\omega : \mathbb{R} \to \mathbb{R}^d$ such that $\omega(0) = 0 \in \mathbb{R}^d$. We endow $\Omega$ with the compact-open topology, which is given by the complete metric
    \[
    d(\omega, \hat{\omega}) := \sum_{n=1}^{\infty} \frac{1}{2^n} \frac{\| \omega - \hat{\omega} \|_n}{1 + \| \omega - \hat{\omega} \|_n}, \quad \| \omega - \hat{\omega} \|_n := \sup_{|t| \leq n} \| \omega(t) - \hat{\omega}(t) \|.
    \]
    \item[(ii)] $\mathcal{F} = \mathcal{B}(\Omega)$, the Borel $\sigma$-algebra on $(\Omega, d)$.
    \item[(iii)] The Wiener measure $\mathbb{P}$ on $(\Omega, \mathcal{F})$ such that the $d$ processes $(W^1_t), \dots, (W^d_t)$ defined by 
    \[
    (W^1_t(\omega), \dots, W^d_t(\omega))^\top := \omega(t) \quad \text{for} \quad \omega \in \Omega
    \]
    are independent one-dimensional Brownian motions, i.e. for all $x \in \mathbb{R}^d$ and $t\neq 0$
    \[
    \mathbb{P} \left( \left\{ \omega \in \Omega : \omega_1(t) \leq x_1, \dots, \omega_d(t) \leq x_d \right\} \right) = \frac{1}{(2\pi t)^{d/2}} \int_{-\infty}^{x_1} \dots \int_{-\infty}^{x_d} e^{-\|y\|^2 / 2 |t|} dy_1 \dots dy_d.
    \] 
    \item[(iv)] The sub-$\sigma$-algebra $\mathcal{F}_{s,t}$ is the completion of the $\sigma$-algebra generated by $\omega(u) - \omega(v)$ for $s \leq v \leq u \leq t$. We write $\mathcal{F}_t$ for the smallest $\sigma-$algebra containing all $\mathcal{F}_{s,t}$ for any $s\leq t$. Note that $\mathcal{F}_{t,t}$ is different from $\mathcal{F}_t$, where the first one, by Blumenthal's zero-one law, is trivial, and the second one containing all the history before $t$.  
    
    \item[(v)] The family of shifts $(\theta_t)_{t\in \R}$, where for each $t\in \R$, $\theta_t:\Omega\to \Omega$ is   defined by 
    $$
    \theta_t(\omega)(\cdot)= \omega(t+\cdot)-\omega(t),\qquad \forall~\omega\in\Omega, 
    $$ 
    is measure preserving and ergodic \cite{Ki12}.
\end{itemize}

In order to solve some of the backward SPDE in the later sections, we need a  martingale representation theorem. This result is generally considered  for one-sided Brownian motion \cite{G16}. We include its extension to our framework for the completeness of the proof. 

\begin{lem}\label{lem.representation}
Given the probability space as defined above, let $T\in \R$ and  $X$ be a random variable which is $\mathcal{F}_T$-measurable and $L^2$ integrable. Then there is a  unique $(\mathcal{F}_s)_{s\leq T}$-progressively measurable process $(h_s)$, such that, for any $t < T$, $\E\left[ \int_t^T |h_s|^2ds\right]<\infty$ and  we have
\be\label{eq.martingale_representation}
X = \E [X|\mathcal{F}_t] + \int_t^T h_s d W_s.
\ee
\end{lem}
\begin{proof}
Suppose without loss of generality that $\EP{X}=0$. Consider $X_\tau := \EP{ X|\mathcal{F}_{\tau,T}}$ where $ -\infty < \tau \leq T$. Using the fact that   $\mathcal{F}_{\tau,T}$ is independent from $\mathcal{F}_\tau$ and  \cite[Theorem 5.18]{G16}, $X_\tau$ can be expressed as
\[
X_\tau = \EP{X|\mathcal{F}_{\tau,t}} + \int_t^T h^\tau_s d W_s = \EP{X} + \int_\tau^T h^\tau_s d W_s = \int_\tau^T h^\tau_s d W_s,
\]
where $h^\tau_s$ is $\mathcal{F}_{\tau,s}$-progressively measurable and $L^2$ integrable such that
\[
\EP{X_\tau^2} = \EP{\int_\tau^T (h^\tau_s)^2 ds}. 
\]
Noting that $X_\tau$ is an $L^2$ closed backward martingale, we can hence apply Doob's inequality to derive that $X_\tau$ converges in $L^2$ to $X$ as $\tau \to -\infty$, which is equivalent to saying that $X_\tau$ is a Cauchy sequence for the $L^2$ norm. Now extend $h^\tau$ to be $0$ before $\tau$ such that it is defined on $(-\infty,T]$. Then, we find that for any $-\infty < \tau_1 < \tau_2 \leq T$,
\[
\EP{\big| X_{\tau_1} - X_{\tau_2} \big|^2} = \EP{\int_{-\infty}^T \big| h^{\tau_1}_s - h^{\tau_2}_s \big|^2 ds }.
\]
Since $X_\tau$ is a Cauchy sequence in $L^2$, the equality above shows that so is $h^\tau$ in $L^2$. By the completeness of $L^2$, we can hence find a progressively measurable process  $h$, which is the limit of $h^\tau$. Now, as $\EP{X|\mathcal{F}_{\tau,t}}$ is an $L^2$ closed martingale, it  converges to $\EPt{X}$. In conclusion, \eqref{eq.martingale_representation}  holds. The uniqueness follows from standard arguments.
\end{proof}

We recall that our state space is the $d-$dimensional torus ($d$ being a nonzero integer) $\T^d= \R^d/\Z^d$. We denote by $\mathcal P(\T^d)$ or $\mathcal P$ the set of Borel probability measures on $\T^d$, endowed with the Wasserstein distance ${\bf d}_1$ \cite{V08}. \\

For $n\in \N$ and $\beta\in [0,1)$, we denote by $C^{n+\beta}$ the space of functions $f:\T^d\to \R$ which have derivatives of order $n$ and if any derivative $D^\ell f$ with $\ell=(\ell_1, \dots, \ell_d)$, $|\ell|:=\ell_1+\dots \ell_d=n$, is $\beta-$H\"{o}lder continuous (just continuous if $\beta=0$). The standard norm on this space is denoted by $\|\cdot \|_{C^{n+\beta}}$ or $\|\cdot\|_{n+\beta}$ if there is no ambiguity. We denote by $(C^{n+\beta})'$ its dual space, with norm $\|\cdot\|_{(C^{n+\beta})'}$. If $f\in L^2(\Omega\times \T^d)$ (i.e. is measurable and square integrable), we denote by $\|f\|_{L^2_x}= (\int_{\T^d} |f(x)|^2)^{1/2}$ its (random) norm in space and by $\|f\|_{L^2_{\omega,x}}= (\E[\int_{\T^d} |f(x)|^2])^{1/2}$ its full norm. For  $f\in L^1(\T^d)$, we often write $\hat{f}$ for $\int_{\T^d} f(x) dx$ and $\tilde{f}$ for $\tilde f(x)= f(x) -\hat{f}$. 

For $n\in \N$, we denote by $H^n$ the Sobolev space $W^{n,2}(\T^d)$ and set, by convention, $H^0=L^2$. Given $t_0<T$, we let ${\bf H}^n_{t_0,T}$ (or ${\bf H}^n$ if there is no ambiguity) be the set  of $(\mathcal F_{t_0,t})$ predictable processes $u$ such that 
$$
\E\left[ \int_{t_0}^T \|u_t\|_{H^n}^2dt\right]<\infty. 
$$

\bigskip

{\bf Standing assumptions:} The coupling functions $f,g: \T^d\times \mathcal P\to \R$ are supposed to satisfy the following regularity   and monotonicity conditions: for some fixed $n\geq 3$ and $\alpha\in (0,1)$, 

\begin{itemize}
\item[{\bf (Hf)}] $f$ is of class $C^2$ with respect to the measure variable  with regularity:
$$
\begin{array}{l}
 \ds  \sup_{m\in \mathcal P} \left(\left\|f(\cdot, m)\right\|_{n+\alpha}+ \left\|\frac{\delta f(\cdot, m,\cdot)}{\delta m}\right\|_{(n+\alpha, n+\alpha)}\right)\\
 \ds \qquad  +\sup_{m\in \mathcal P}\left\|\frac{\delta^2 f(\cdot, m,\cdot,\cdot)}{\delta m^2}\right\|_{(n+\alpha,n+\alpha,n+\alpha)}
+ {\rm Lip}_n(\frac{\delta^2 f}{\delta m^2})\; <\; \infty.
\end{array}$$
(where we refer to \cite{CLLP2} for the norms and the notation $Lip_n$). 

\item[{\bf (Hg)}] in the same way, $g$  is of class $C^2$ with respect to the measure variable and with regularity: 
$$
\begin{array}{l}
\ds  \sup_{m\in \mathcal P} \left(\left\|g(\cdot, m)\right\|_{n+1+\alpha}+ \left\|\frac{\delta g(\cdot, m,\cdot)}{\delta m}\right\|_{(n+1+\alpha,n+1+\alpha)}\right)\\
 \ds \qquad  +\sup_{m\in \mathcal P}\left\|\frac{\delta^2 g(\cdot, m,\cdot,\cdot)}{\delta m^2}\right\|_{(n+1+\alpha,n+1+\alpha,n+1+\alpha)}
+ {\rm Lip}_{n+1}(\frac{\delta^2 g}{\delta m^2})\; <\; \infty.
\end{array}$$
\\ 

\item[{\bf (Hm)}] The maps $f$ and $g$ are assumed to be monotone:  for any $m\in\mathcal P$ and for any centered Radon measure $\mu$, 
\be\label{hyp:mono}
\inte\inte \frac{\delta f}{\delta m}(x,m,y)\mu(x)\mu(y)dxdy\geq 0, \qquad \inte\inte \frac{\delta g}{\delta m}(x,m,y)\mu(x)\mu(y)dxdy\geq 0.
\ee
\end{itemize}

Unless otherwise specified, these conditions {\bf (Hf), (Hg), (Hm)} are supposed to hold throughout the paper.   
Note that the requirement \( n \geq 3 \) is not  optimal. We impose it in the same spirit as the regularity condition in \cite[Theorem 2.4.5]{CDLL}, to ensure the existence of a smooth solution to the second master equation, allowing us to  discuss its long-time behavior. \\

By convention, we typically use $\beta$ or $\beta'$ to denote an arbitrary real number in the interval $(0,1)$, and $\alpha'$ to denote a number in $(0,\alpha)$, where $\alpha \in (0,1)$ is fixed as in assumptions \textbf{(Hf)} and \textbf{(Hg)}. 
%The symbols $\beta$ and $\beta'$ are employed when the highest level of regularity is not required (e.g., when $C^{3+\beta}$ regularity for $\uT$ suffices instead of $C^{4+\alpha}$), whereas $\alpha$ or $\alpha'$ are reserved for situations in which the strongest regularity assumptions are essential. 

%%%%%%%%%%%%%%%%%%%%%%%%%%%%%%%%%
\section{The time dependent problem} \label{sec.timedep}

We concentrate here on the long-time average of the time dependent MFG system with a common noise: 
\be \label{eq.MFGCN}
\left\{
\begin{array}{cl}
    (i) & d\uT = \left(-(1+\cwn)\Delta \uT +\frac12 |D \uT|^2 -f(x, \mT) -2\cwn {\rm div}(\vT) \right)dt + \sqrt{2 \cwn} v^T_t\cdot dW_t, \\ [1ex]
    (ii) & d\mT = \left((1+\cwn)\Delta \mT +{\rm div} (\mT D \uT) \right)dt -\sqrt{2 \cwn} D \mT\cdot dw_t, \\ [1ex]
%    (iii) & m^T_{t_0} = m_0, \qquad u^T_T(x) = g(x, m^T_T).
\end{array}
\right.
\ee
often supplemented with the initial condition (for $m^T$) and the terminal condition (for $u^T$):
\be\label{eq.MFGCNbc}
\begin{array}{cl}
   (iii) & m^T_{t_0} = \hat{m_0}, \qquad u^T_T(x) = g(x, m^T_T).
\end{array}
\ee
This system is stated on $[t_0,T]\times \T^d$, where $t_0,T\in \R$, $t_0<T$. Throughout this part we assume that $\hat{m_0}\in \mathcal P$ is arbitrary. Our aim is to understand the limit of $\mT$ and of $\uT/T$ as $T\to\infty$. To do so, we first introduce the definition of solution for the MFG system and for the associated Hamilton-Jacobi equations.

\subsection{Solutions of Hamilton-Jacobi equations and of the MFG system} \label{subsec.HJMFG}

We first discuss  various notions of solutions for the Hamilton-Jacobi (HJ) equation
\be\label{eq.HJ}
\left\{
\begin{array}{cl}
    (i) & du_t = \left(-(1+\cwn)\Delta u_t +\frac\theta2 |D u_t|^2 -f_t(x) -2\cwn {\rm div}(v_t) \right)dt + \sqrt{2 \cwn} v_t\cdot dW_t, \\ [1ex]
    (ii) &   u_T(x) = g(x).
 \end{array}
\right.
   \ee
where $\theta\in \{0,1\}$, $f:\Omega\times [t_0,T]\times \T^d\to \R$ is measurable and adapted to the filtration $(\mathcal F_{t_0,t})$ and $g:\Omega\times\T^d\to \R$ is $\mathcal F_{t_0,T}\otimes \mathcal B(\T^d)-$ measurable. The analysis of backward stochastic HJ equation goes back to the pioneering work by Hu-Peng \cite{HP91} (for linear equations) and Peng \cite{Pe92} (for backward random HJ equations) and is further discussed for instance in \cite{DuTa12, DuMe13, DuChen14, HMY02, MaYo97}. We follow here  \cite{DuChen14}. Assuming that 
$$
{\rm ess-sup}_{(\omega,t,x)} |f_t(x)| + |g(x)| < \infty, 
$$
Theorem 2.1 and Theorem 2.2 of \cite{DuChen14} state (in a much more general framework) that there exists a unique {\it weak solution} $(u,v)\in {\bf H}^1\times [{\bf H}^0]^d$ to \eqref{eq.HJ}, in the sense that, for any test function $\phi\in C^\infty(\T^d)$ and a.s., for any $t\in [t_0,T]$, 
$$ 
\int_{\T^d} g \phi -\int_{\T^d} u_t \phi = \int_t^T \int_{\T^d} (-(1+\sigma) u_s \Delta \phi + (\frac{\theta}{2} |Du_s|^2-f_s)\phi +2\cwn v_s\cdot D\phi)ds
+ \sqrt{2\cwn}  \int_t^T \int_{\T^d} \phi v_s\cdot dW_s. 
$$
The regularity of the solution can be improved under stronger assumptions on the data: 

\begin{prop}\label{prop.HJ} Fix $n\geq 2$ and $\beta\in (0,1)$. Suppose that 
\be\label{hyp.gfHJH}
{\rm ess-sup}_{(\omega,t)}( \|f_t\|_{C^{n+\alpha}}+ \|g\|_{C^{n+1+\alpha}})< \infty. 
\ee
Then there exists a unique {\rm strong solution} to \eqref{eq.HJ}, in the sense that  $u$ has path in $C^0([t_0,T], C^{n+1})$, with 
\(
\sup_{t \in [t_0,T]} \| u_t\|_{C^{n+1+\alpha}} \in L^\infty(\Omega),
\)
$v\in [{\bf H}^{n+1}]^d$ and equality
$$
u_t(x) = g(x) - \int_t^T (-(1+\sigma) \Delta u_s(x) +\frac{\theta}{2} |Du_s(x)|^2-f_s +2\cwn {\rm div} (v_s(x)))ds
- \sqrt{2\cwn}  \int_t^Tv_s(x)\cdot dW_s 
$$
holds a.s. for any $t\in [t_0,T]$ and for a.e. $x\in \T^d$.

Moreover, the following comparison holds: if $(f^1,g^1)$ and $(f^2,g^2)$ satisfy \eqref{hyp.gfHJH} with $f^1\geq f^2$ and $g^1\leq g^2$ a.s., and if $(u^1,v^1)$ and $(u^2,v^2)$ are the strong solution to \eqref{eq.HJ} with $(f,g)$ replaced by $(f^1,g^1)$ and by $(f^2,g^2)$ respectively, then $u^1\leq u^2$ a.s.
\end{prop} 

The uniqueness of the strong solution is of course a direct consequence of the uniqueness of the weak solution. To prove the existence part of the result, let us introduce the shifted Hamilton-Jacobi equation: set 
\be\label{def.shiffg}
\shif f_t(x)= f(x+\sqrt{2\cwn}(W_t-W_{t_0})), \qquad {\rm and}\qquad 
\shif g(x)= g(x+\sqrt{2\cwn}(W_T-W_{t_0}))
\ee
and consider the shifted Hamilton-Jacobi equation 
\be\label{eq.HJshif}
\left\{
\begin{array}{cl}
(i)& d\shif u_t= \left(-\Delta \shif u_t +\frac\theta2 |D  \shif u_t|^2 -\shif f_t(x) \right)dt +d\shif M_t \\ [1ex]
(ii) & \shif u_T(x)=\shif g(x)
\end{array}\right. 
\ee
Here the unknown is $(\shif u, \shif M)$ where $(\shif M_t)$ is a martingale with $\shif M_{t_0}=0$. 
Under assumption \eqref{hyp.gfHJH}, \cite[Theorem 4.3.1]{CDLL} states that equation \eqref{eq.HJshif} has a unique classical solution, in the sense that $(\shif u,\shif M)$ is an $(\mathcal F_{t_0,t})_{t\in [t_0,T]}$-adapted process with paths in
\(
C^0([t_0,T], C^{n+1}(\T^d)  \times [C^{n-1}(\T^d)]^d),
\)
such that 
\(
\sup_{t \in [t_0,T]} \|\shif u_t\|_{C^{n+1+\alpha}}+\|\shif M_t\|_{[C^{n-1+\alpha}]^d}  \in L^\infty(\Omega),
\)
and $(\shif M_t(x))$ is an $(\mathcal F_{t_0,t})$ martingale for any $x\in \T^d$. In addition, equalities (i) and (ii) hold a.s. in a classical sense for any $(t,x)$. The  link between \eqref{eq.HJ} and \eqref{eq.HJshif} is the fact that, if $(u,v)$ solves \eqref{eq.HJ} and if we set 
$$
\shif u_t(x)=  u_t(x+\sqrt{2\cwn}(W_t-W_{t_0})), 
$$
$$
\shif v_t(x)= D \shif u_t(x)+ v_t(x+\sqrt{2\cwn}(W_t-W_{t_0}))\qquad {\rm and}\qquad \shif M_t = \sqrt{2\cwn} \shif v_t\cdot W_t,
$$
then, by formally applying It\^{o}-Wentzell formula, $(\shif u, \shif M)$ solves \eqref{eq.HJshif}. However, in general, the regularity of $(u,v)$ is not strong enough to justify this computation (unless $n$ is large, see for instance \cite{DM25}). Nevertheless we have:

\begin{lem}\label{lem.HJshif} Assume that \eqref{hyp.gfHJH} holds. Let $(\shif u, \shif M)$ be the classical solution to \eqref{eq.HJshif} and let 
\be\label{equalityushifu}
u_t(x)= \shif u_t(x-\sqrt{2\cwn}(W_t-W_{t_0})).
\ee
Then there exists $v\in {\bf H}^{n+1}$ such that $(u,v)$ is the unique strong solution to \eqref{eq.HJ}. 
\end{lem} 

\begin{proof} We first do the proof in the case $\theta=0$. In view of \cite[Theorem 2.3]{DuMen10}, there exists a unique strong solution $(u,v)$ to the linear equation \eqref{eq.HJ} with $\theta=0$, where ``strong solution'' in \cite{DuMen10} means  that $u\in {\bf H}^{2}$, $v\in {\bf H}^{1}$, $u\in C^0([t_0,T],L^2(\T^d))$ and equality 
$$
u_t(x) = g(x) - \int_t^T (-(1+\sigma) \Delta u_s -f_s +2\cwn {\rm div} (v_s))ds
- \sqrt{2\cwn}  \int_t^Tv_s\cdot dW_s 
$$
holds a.s. for any $t\in [t_0,T]$ and for a.e. $x\in \T^d$. In addition, \cite[Theorem 2.3]{DuMen10} actually proves that  $u\in {\bf H}^{n+2}$, $v\in [{\bf H}^{n+1}]^d$, $u\in C^0([t_0,T],H^n)$. Let us show that  \eqref{equalityushifu} holds. As explained above, the difficulty is that, unless $n$ is large, we cannot apply the Ito-Wentzell formula directly.  

To overcome this issue, let us introduce a smooth mollifier $\eta:\R^d\to \R_+$ with a compact support and such that $\int_{\R^d} \eta =1$, and set $\eta^\ep (x)= \ep^{-d}\eta (x/\ep)$. Define $\shif f^\ep= \eta^\ep\ast \shif f_t$, $\shif g^\ep= \eta^\ep \ast \shif g$. By the uniqueness claim of \cite[Theorem 4.3.1]{CDLL} the unique solution to \eqref{eq.HJshif} with $\theta=0$ and $\shif f^\ep$ and $\shif g^\ep$ in place of $\shif f$ and $\shif g$ is given by $\shif u^\ep=\eta^\ep \ast \shif u$ and $\shif M^\ep= \eta^\ep\ast \shif M$. Note now that, for any $k\in \N$,  $k\geq 2$, $(\shif u^\ep, \shif M^\ep)$ has paths in
\(
C^0([t_0,T], C^{k}(\T^d)  \times [C^{k-2}(\T^d)]^d),
\)
such that 
\(
\sup_{t \in [t_0,T]} \|\shif u^\ep_t\|_{C^{k+\alpha}}+\|\shif M^\ep_t\|_{[C^{k-2+\alpha}]^d}  \in L^\infty(\Omega).
\)
 Following \cite[Lemma 3.2]{DM25}, there exists a process $\shif k^\ep\in \bigcap_{k} {\bf H}^k$ such that, for any multiindex $a\in \N^d$,  
$$
D^a \shif M^\ep_{t}(x)= \int_{t_0}^t D^a\shif k^\ep_{ s}(x) dW_s
$$
and, for any $k\in \N$,  
$$
\E\left[\int_{t_0}^T \|\shif k^\ep_{ s}\|_{C^k}^2ds\right]<\infty. 
$$
Let us set 
$$
u^\ep_{t}(x)=  \shif u^\ep_{t}(x-\sqrt{2\cwn}(W_t-W_{t_0})), \qquad v^\ep_{t}(x)= ({\bf k}^\ep_{t}-D\shif u^\ep_{t})(x-\sqrt{2\cwn}(W_t-W_{t_0})). 
$$
By It\^{o}-Wentzell formula (see \cite[Theorem 3.3.1]{K90}) the pair $(u^\ep, v^\ep)$ is a strong solution to the linear equation \eqref{eq.HJshif} with $\theta=0$ and $f^\ep= \eta^\ep \ast f$, $g^\ep= \eta^\ep\ast g$ in place of $f$ and $g$. In view of \cite[Theorem  2.3]{DuMen10}, the difference $(u-u^\ep, v-v^\ep)$ (which is a strong solution to an equation of the form \eqref{eq.HJ} with $\theta=0$ and $f-f^\ep$ and $g-g^\ep$ in place of $f$ and $g$) satisfies, for any $k\in \N$, 
\begin{align*}
& \E\left[ \int_{t_0}^T (\|(u_t-u^\ep_{t})(\cdot)\|^2_{H^{k+2}(\T^d)} +\|(v_t-v^\ep_{t})(\cdot)\|^2_{[H^{k+1}(\T^d)}]^d)dt \right]  \\
& \qquad \qquad \qquad \leq C
\E\left[ \int_{t_0}^T \|(f_t-f^\ep_{t})(\cdot)\|^2_{H^{k}(\T^d)}dt +\|(g-g^\ep)(\cdot)\|^2_{H^{k+1}(\T^d)} \right] .
\end{align*}
In view of our assumption \eqref{hyp.gfHJH}, the right-hand side tends to $0$ for $k=n$ as $\ep$ tends to $0$. Hence $u^\ep$ tends to $u$ in ${\bf H}^{n+2}$. On the other hand, by construction, the pointwise limit of $u^\ep_{t}(x)$ is nothing but $\shif u_t(x-\sqrt{2\cwn}(W_t-W_{t_0}))$. Thus equality \eqref{equalityushifu} holds and, from the regularity of $\shif u$, we infer that 
$$
\sup_t\| u_t(\cdot)\|_{C^{n+1+\alpha}}= 
\sup_t \| \shif u_t(\cdot)\|_{C^{n+1+\alpha}} \in L^\infty(\Omega). 
$$

We now assume that $\theta=1$ and let $(\shif u, \shif M)$ be the unique classical solution to the HJ equation \eqref{eq.HJshif}. We note that $(\shif u, \shif M)$ can also be seen as  the unique classical solution to the linear equation \eqref{eq.HJshif} with $\theta=0$ and $\shif f$ replaced by $\tilde{\shif f}_t= \frac12 |D\shif u_t|^2 -\shif f_t$. As 
$$
{\rm ess-sup}_{(\omega, t)} \| \tilde{\shif f}_t\|_{C^{n+\alpha}} < \infty 
$$
by the regularity of $\shif u$, the first part of the proof implies that, if we define $u$ by \eqref{equalityushifu}, then there exists $v\in {\bf H}^{n+1}$ such that equality \eqref{equalityushifu} holds and $(u,v)$ is a strong solution (in the sense of \cite{DuMen10}) of  \eqref{eq.HJshif} with $\theta=0$ and where $ f$ is replaced by $\tilde f_t= \frac12 |Du_t|^2 - f_t$. But then $(u,v)$ solves \eqref{eq.HJ} in a strong sense. 
\end{proof}

\begin{proof}[Proof of Proposition \ref{prop.HJ}] The uniqueness of the solution to \eqref{eq.HJ} has been discussed above, while the existence is given by \cite[Theorem 4.3.1]{CDLL} combined with Lemma \ref{lem.HJshif}. The comparison is a  consequence of the corresponding comparison theorem in \cite{CSS22} for the classical solution to the shifted HJ equation \eqref{eq.HJshif}, again combined with Lemma \ref{lem.HJshif}.
\end{proof}

Now we discuss the notion of solution for the MFG system \eqref{eq.MFGCN}. In the paper, we  mostly work with {\it classical solutions} (or simply {\it solutions}):  

\begin{defn}[Classical solution to \eqref{eq.MFGCN}] The triple ($u,m,v)$ is a classical solution to \eqref{eq.MFGCN} if it is an $(\mathcal F_{t_0,t})_{t\in [t_0,T]}$-adapted process with paths in
\(
C^0([t_0,T], C^{n+1}(\T^d) \times \mathcal{P} \times [C^{n-1}(\T^d)]^d),
\)
such that
\(
\sup_{t \in [t_0,T]} \left( \|\uT(t)\|_{n+1+\alpha'} + \|\vT(t)\|_{n-1+\alpha'} \right) \in L^\infty(\Omega).
\)
In addition, equation \eqref{eq.MFGCN}-(i) holds in the classical sense with probability one, while \eqref{eq.MFGCN}-(ii) is understood in the sense of distributions, again with probability one. 
\end{defn}
Under our standing assumptions,  it is proved in \cite[Corollary 2.4.7]{CDLL} that the MFG system \eqref{eq.MFGCN}-\eqref{eq.MFGCNbc} admits a unique classical solution.
We note that in \cite{CDLL}, the Hamiltonian is  assumed to be globally Lipschitz continuous. It is known that this condition can actually be relaxed \cite{CSS22, CDbook, DM25} thanks to a uniform Lipschitz estimate for $u$ (see also Lemma \ref{lem.regularity}).\\

We will often study System \eqref{eq.MFGCN}-\eqref{eq.MFGCNbc} with a random initial measure $\hat{m_0}$. If $\hat{m_0}$ is $\mathcal F_{t_0}-$measurable, (respectively $\mathcal F_{t_1,t_0}$ measurable  for some $t_1<t_0$) then a simple adaptation of \cite{CDLL} shows that there exists a unique strong solution adapted to the filtration  $(\mathcal F_{t})_{t\leq T}$ (respectively to the filtration $(\mathcal F_{t_1,t})_{t\in [t_1,T]}$). Note also that, if $\hat{m_0}$ is deterministic, then a solution for  $(\mathcal F_{t_0,t})_{t\in [t_0,T]}$  is automatically a solution for the larger filtrations $(\mathcal F_{t})_{t\leq T}$ or $(\mathcal F_{t_1,t})_{t\in [t_1,T]}$.\\

The proof of the existence of a classical solution to the MFG system \eqref{eq.MFGCN}-\eqref{eq.MFGCNbc} in  \cite{CDLL} relies on two main ingredients. The first one is the analysis of a shifted system that we describe next as it plays a key role in our analysis. The second one is the construction of a classical solution to the master equation, that we will discuss after Section~\ref{sec.ergo}. \\

To describe the shifted MFG system, let us fix $t_0,T\in \R$ with $t_0<T$ and  introduce the random maps (for $x\in \T^d$ and $m\in \mathcal P(\T^d)$): 
$$
\shif f_t(x,m)= f(x+\sqrt{2\cwn}(W_t-W_{t_0}),  (id+\sqrt{2\cwn}(W_t-W_{t_0}))\sharp m)
$$
and
$$
\shif g_T(x,m)= g(x+\sqrt{2\cwn}(W_T-W_{t_0}),  (id+\sqrt{2\cwn}(W_T-W_{t_0}))\sharp m). 
$$
Note that $\shif f$ is adapted to $(\mathcal F_{t_0,t})_{t\geq t_0}$ and that $g_T$ is $\mathcal F_{t_0,T}-$measurable. The shifted MFG system is the backward-forward system 
\be\label{eq.MFGCNtilde}
\left\{
\begin{array}{cl}
(i)& d\shif u_t= \left(-\Delta \shif u_t +\frac12 |D  \shif u_t|^2 -\shif f_t(x,\shif m_t) \right)dt +d\shif M_t \\ [1ex]
(ii) & d\shif m_t=\left(\Delta \shif m_t +{\rm div}(\shif m_t D  \shif u_t)\right)dt \\ [1ex]
%(iii) & \shif m_{t_0}= \bar m_0, \qquad \shif u_T(x)=\shif g_T(x,\shif m_T)
\end{array}\right. 
\ee
supplemented with the initial and terminal conditions
\be\label{eq.MFGCNtildebc}
\begin{array}{cl}
(iii) & \shif m_{t_0}= \hat{m_0}, \qquad \shif u_T(x)=\shif g_T(x,\shif m_T)
\end{array}
\ee
In the system above, the unknown is the triple $(\shif u, \shif m, \shif M)$, where for any $x\in \T^d$, $\shif M_t(x)$ is a martingale (with $\shif M_{t_0}=0$). Note that the second equation is a deterministic Fokker-Planck equation with random coefficient. This will play a central role in our analysis below. 

By a solution to \eqref{eq.MFGCNtilde}, we mean the following: 
\begin{defn}[Solution to the shifted system \eqref{eq.MFGCNtilde}] By a solution $(\shif u, \shif m, \shif M)$ to \eqref{eq.MFGCNtilde} we mean that $(\shif u, \shif m, \shif M)$ is an $(\mathcal F_{t_0,t})_{t\in [t_0,T]}-$adapted process with path in $C^0([t_0,T], C^{n+1}(\T^d)\times \mathcal P(\T^d)\times [C^{n-1}(\T^d)]^d)$, such that
\(
\sup_{t \in [t_0,T]} \left( \|\shif u_t\|_{C^{n+1+\alpha}} + \|\shif M_t\|_{[C^{n-1+\alpha}]^d} \right) \in L^\infty(\Omega)
\) 
and such that, for each $x\in \T^d$, $t\to \shif M_t(x)$ is a martingale vanishing at $t_0$. The first equation in \eqref{eq.MFGCNtilde} holds in a classical sense with probability 1 for any $(t,x)$, while, still with probability 1, the second one  is understood in the sense of distributions. 
\end{defn}
According to \cite[Section 5.5]{CDLL}, there exists a unique solution $(\shif u, \shif m, \shif M)$ to the shifted MFG system  \eqref{eq.MFGCNtilde}-\eqref{eq.MFGCNtildebc}.
As for system \eqref{eq.MFGCN}-\eqref{eq.MFGCNbc}, it will also be convenient to work with a random initial measure and for the larger filtration  $(\mathcal F_{t})_{t\leq T}$. \\

The relationship between the MFG system \eqref{eq.MFGCN} and the shifted MFG system \eqref{eq.MFGCNtilde} is the following: if $(u, m, v)$  is a classical solution of \eqref{eq.MFGCN} on the time interval $[t_0,T]$ and if we set 
\be\label{shifumVsum}
\shif u_t(x)= u_t(x+\sqrt{2\cwn}(W_t-W_{t_0})), \; \shif m_t = (\operatorname{Id}-\sqrt{2\cwn}(W_t-W_{t_0}))\sharp m_t,
\ee
and
\be\label{shifumVsumbis}
\shif v_t(x)= D \shif u_t(x)+ v_t(x+\sqrt{2\cwn}(W_t-W_{t_0}))\qquad {\rm and}\qquad \shif M_t = \sqrt{2\cwn} \shif v_t\cdot W_t,
\ee
then by It\^{o}-Wentzel formula, $(\shif u, \shif m, \shif M)$ solves the shifted MFG system \eqref{eq.MFGCNtilde} (see the discussion in \cite[Section 5.5]{CDLL}). 

In \cite{CDLL}, we build a solution to the MFG system from the shifted one. The difficult part for this is the construction of the $v$ component of the solution to the MFG system, which requires a representation for the martingale part $\shif M$ of the solution to the shifted MFG system. This is delicate question in general: on this topic, see \cite{DM25} for instance. In \cite{CDLL} we build the $v$ component from the  solution to the master equation, which requires the specific structure of the MFG system \eqref{eq.MFGCN}: namely, the maps $f$ and $g$ have to be deterministic and  smooth in measure. 

We will need below to build a solution to the  MFG system \eqref{eq.MFGCN} when the terminal condition $g$ is smooth in space but just Lipschitz in measure, or when the  terminal condition is random. To fix the ideas, we now discuss this latter case. More precisely, let us assume that $\bar u_T:\Omega\times \T^d$ is $\mathcal F_T-$measurable and such that
\be\label{Hyp.baruT}
{\rm ess-sup} \|\bar u_T\|_{C^{n+1+\alpha}} < \infty. 
\ee
Note that we can no longer build a classical solution to \eqref{eq.MFGCN} with a terminal condition $\bar u_T$ as in \cite{CDLL} from the shifted system, as $\bar u_T$ is random. For this reason we introduce a notion of weak solution to the MFG system as follows: 

\begin{defn}[Weak solution to   \eqref{eq.MFGCN}] \label{def.weaksol} The triple ($u,m,v)$ is a weak solution to \eqref{eq.MFGCN}  if  $v\in [{\bf H}^{n+1}]^d$ and  ($u,m)$ is an $(\mathcal F_{t_0,t})_{t\in [t_0,T]}$-adapted process with paths in
\(
C^0([t_0,T], C^{n+1}(\T^d) \times \mathcal{P} ),
\)
such that 
\(
\sup_{t \in [t_0,T]}  \|\uT(t)\|_{n+1+\alpha}  \in L^\infty(\Omega).
\)
In addition, we assume that for any $t\in [t_0,T]$ and with probability one and for a.e. $x\in \T^d$, one has 
\be\label{zkhebsrndfc}
u_t(x)= \bar u_T(x)+ \int_t^T  \left(-(1+\cwn)\Delta u_s(x) +\frac12 |D u_s(x)|^2 -f(x, m_s) -2\cwn {\rm div}(v_s(x)) \right)ds + \sqrt{2 \cwn}  \int_t^Tv_s\cdot dW_s,
\ee
while \eqref{eq.MFGCN}-(ii) is understood in the sense of distributions, again with probability one. 
\end{defn}

The main difference between classical and weak solution is, on the one hand, the regularity of $v$ and, on the other hand, the sense in which the HJ equation is satisfied: in the weak setting, $(u,v)$ is supposed to be a strong solution (in the sense of Proposition \ref{prop.HJ}) to the HJ equation \eqref{eq.MFGCN}-(i). 

\begin{prop}\label{prop.shifumVsum} Under our standing assumption and if $\bar u_T$ satisfies \eqref{Hyp.baruT},  there exists a unique weak solution $(u,m,v)$ to \eqref{eq.MFGCN} with initial condition $m_{t_0}=\hat{m_0}$ and terminal condition $u_T=\bar u_T$. In addition, $v\in {\bf H}^n$ and if we set 
$$
\shif{\bar u}_T (x)=  \bar u_T(x+\sqrt{2\cwn}(W_T-W_{t_0}))
$$
and let  $(\shif u, \shif m, \shif M)$ be the unique classical solution  to \eqref{eq.MFGCNtilde} with initial condition $\shif m_{t_0}=\hat{m_0}$ and terminal condition $\shif u_T= \shif{\bar u}_T$, then equalities \eqref{shifumVsum} hold. 
\end{prop}

\begin{rmk}\label{rem.prop.shifumVsumextended} A symmetric result holds when the terminal condition is of the form $u_T=g(\cdot, m_T)$, where $g$ is bounded in $C^3$ is space (uniformly in $x$), monotone and Lipschitz in measure. 
\end{rmk} 

Let us note that \cite[Section 5.5]{CDLL} implies the existence of a  unique classical solution  $(\shif u, \shif m, \shif M)$ to \eqref{eq.MFGCNtilde} with initial condition $\shif m_{t_0}=\hat{m_0}$ and terminal condition $\shif u_T= \shif{\bar u}_T$. Indeed, in contrast with the construction of the solution of the MFG system, the construction of the solution to the shifted MFG system in   \cite[Section 5.5]{CDLL}  does not use the special structure of the data.

\begin{proof} To build the solution, let  $(\shif u, \shif m, \shif M)$ be the unique classical solution  to \eqref{eq.MFGCNtilde} with initial condition $\shif m_{t_0}=\hat{m_0}$ and terminal condition $\shif u_T= \shif{\bar u}_T$ and let us define $(u,m)$ by \eqref{shifumVsum}. By Lemma \ref{lem.HJshif}, there exists  $v\in {\bf H}^{n+1}$ such that  equality \eqref{zkhebsrndfc} holds.  It remains to check that \eqref{eq.MFGCN}-(ii) holds in the sense of distribution. The proof of of this fact is standard and we only check it. Let $\phi\in C^\infty_c((t_0,T)\times \T^d)$  and let us set $\boldsymbol{\phi}_t(x)= \phi_t(x+\sqrt{2\cwn}(W_t-W_{t_0}))$. Fix $n\in \N\backslash \{0\}$ and, for  $k\in \{0, n-1\}$, let $t_k= t_0+ (T-t_0)k/n$. As $\shif m$ is a weak solution, we have a.s: 
$$
\int_{\T^d} \boldsymbol{\phi}_{t_k}(\shif m_{t_{k+1}}-\shif m_{t_k}) = \int_{t_k}^{t_{k+1}} \int_{\T^d} (\Delta \boldsymbol{\phi}_{t_k}- D\shif u_s\cdot D\boldsymbol{\phi}_{t_k})\shif m_sds. 
$$
On the other hand, by It\^{o}'s formula, 
$$
\int_{\T^d} (\boldsymbol{\phi}_{t_{k+1}}-\boldsymbol{\phi}_{t_k})\shif m_{t_{k+1}}  = \int_{t_k}^{t_{k+1}}\int_{\T^d} \sigma \Delta \boldsymbol{\phi}_s \shif m_{t_{k+1}}ds + 
\sqrt{2\cwn} \int_{t_k}^{t_{k+1}}\int_{\T^d} D\boldsymbol{\phi}_s \shif m_{t_{k+1}}\cdot dW_s. 
$$
Summing over $k$ and changing variable in space, we find 
$$
0= \sum_{k=1}^{n-1}\int_{t_k}^{t_{k+1}} \int_{\T^d} ((1+\sigma) \Delta \phi_{t_k}- D u_s\cdot D{\phi}_{t_k}) m_sds
+ \sqrt{2\cwn}  \sum_{k=1}^{n-1}\int_{t_k}^{t_{k+1}}\int_{\T^d} D{\phi}_s m_{t_{k+1}}\cdot dW_s. 
$$
We let then $n\to\infty$ to obtain the result by continuity of the processes. This shows that equation (ii) holds a.s. in the sense of distributions. \\

The uniqueness of the weak solution is a direct consequence of its characterization through the solution of \eqref{eq.MFGCNtilde} that we check now. Let us  fix a  weak solution $(u, m, v)$ to \eqref{eq.MFGCN} with initial condition $m_{t_0}=\hat{m_0}$ and terminal condition $\bar u_T$. We define $(\shif u, \shif m)$ by \eqref{shifumVsum}. Let us show the existence of a process $\shif M$ such that $(\shif u, \shif m, \shif M)$ is a (and in fact the unique) solution to \eqref{eq.MFGCNtilde} with initial condition $\shif m_{t_0}=\hat{m_0}$ and terminal condition $\shif u_T=\bar{\shif u}_T$.  We have seen above that there exists a unique classical solution  $(\shif w, \shif M)$ to the shifted HJ equation 
$$
d\shif w_t =  \left(-\Delta \shif w_t +\frac12 |D \shif w_t|^2 -f_t(x, \shif  m_t)  \right)dt + d\shif M_t, \qquad \shif w_T=\bar{\shif u}_T. 
$$
By Lemma \ref{lem.HJshif}, we know that there exists $z\in {\bf H}^{n+1}$ such that, if we set
$$
w_t(x)= \shif w_t(x- \sqrt{2\cwn}(W_T-W_{t_0}),
$$
then $(w,z)$  is a strong solution to the HJ equation \eqref{eq.MFGCN}-(i). But $(u,v)$ is another solution to this HJ equation and  \cite[Theorem 2.2]{DuChen14} says that this solution is unique: thus $(u,v)=(w,z)$. This proves that $\shif w=\shif u$ and thus the process $(\shif u, \shif m, \shif M)$ solves \eqref{eq.MFGCNtilde} initial condition $\shif m_{t_0}=\hat{m_0}$ and with terminal condition $\shif u_T=\bar{\shif u}_T$.
\end{proof}

\subsection{Long-time average of the time dependent problem} 

In this part we show the existence of a stationary (weak) solution to \eqref{eq.MFGCN}-(i)-(ii) defined on the whole time interval $\R$ and the convergence in a suitable sense of the solution to \eqref{eq.MFGCN} to this stationary solution as $T\to \infty$. \\

Let us start with the existence of a  stationary solution: 

\begin{lem}[The stationary weak solution] \label{lem.Wmu} 
There exists a stationary process $(\bar{\mathcal V}, \bar \mu)_{t\in \R}$, adapted to the filtration $(\mathcal F_t)$ and taking values in $\mathcal C^0(\R, [C^{2}(\T^d)]^d\times \mathcal P)$, with
$$
{\rm essup} \sup_t (\|\bar{\mathcal V}_t\|_{[C^{2+\beta}(\T^d)]^d}+ \|\bar \mu_t\|_{C^{\beta}(\T^d)})<\infty,
$$ 
where $\beta$ is arbitrary in $[0,1)$, such that $\bar{\mathcal V}_t$ is a gradient for any $t$ and a.s.. Moreover, for any $t_0<T$, the unique weak solution $(w, m, z)$ (in the sense of Definition \ref{def.weaksol}) to \eqref{eq.MFGCN}-(i)-(ii) on the time interval $[t_0,T]$ with initial condition $m_{t_0}= \bar \mu_{t_0}$ and terminal condition $w_T= \bar w_T$, satisfies $Dw_t= \bar{\mathcal V}_t$ and $m_t= \bar \mu_t$ for any $t\in [0,T]$. In addition, there is a constant $C>0$ such that 
$C^{-1}\leq \bar \mu\leq C$ a.s. 
\end{lem}

Recall that the weak solution $(w, m, z)$  to \eqref{eq.MFGCN}-(i)-(ii) on the time interval $[t_0,T]$ with initial condition $m_{t_0}= \bar \mu_{t_0}$ and terminal condition $w_T= \bar w_T$ exists and is unique by Proposition \ref{prop.shifumVsum}. 

We will see that the pair $(\bar{\mathcal V}, \bar \mu)$ is actually unique, as it attracts all the solutions of the MFG system. At this stage it is not clear that the process $w$ can be built in a stationary way: we will see in Theorem \ref{prop.fullcorrector}  below (with the help of the master equation), that this is actually possible. \\

The long-time average of the solution is described next: 
\begin{thm}\label{thm.main}  Fix $C>0$ and $t_0\in \R$. Let $(\bar{\mathcal V}, \bar \mu)$  be the stationary solution as defined by Lemma \ref{lem.Wmu} and, for any $T>t_0$, let $(u^T, m^T, v^T)$ be a classical solution to \eqref{eq.MFGCN} on the time interval $[t_0,T]$ with initial condition $m^T_{t_0}=\hat{m_0}$ and satisfying 
\be\label{cond.termcond}
{\rm ess-sup} \|u^T_T\|_{C^{2}} \leq C. 
\ee
Then, for any $\beta\in (0,1)$ there exists $C_\beta$, depending on  $C$, $f$, $g$, $d$ and $\beta$ only, such that
 \be\label{ineq.expo}
 \E\left[\|(m^T-\bar \mu)_t\|_{C^\beta}+ \|(Du^T-\bar{\mathcal V})_t\|_{L^2_x}\right] \leq C_\beta(e^{-\omega (t-t_0)}+e^{-\omega(T-t)}) \qquad \forall t\in [t_0+2, T-2]
 \ee
 and, for any fixed $t\geq t_0$, 
$$
\lim_{T\to\infty}  \frac{u^T_{t}(x)}{T} = \bar \lambda \qquad {\rm in}\; L^1(\Omega), 
$$
the limit being uniform in $x\in \T^d$, where 
\be\label{defbarlambda}
\bar \lambda = - \E\left[  \inte (\frac12 |\bar{\mathcal V}_0(x)|^2 -f(x,\bar \mu_0))dx\right].
\ee
\end{thm}

Of course, the result applies to the classical solution $(u^T, m^T, v^T)$ of \eqref{eq.MFGCN}-\eqref{eq.MFGCNbc}.

\begin{rmk}
    A consequence of Corollary \ref{cor.uT-lambda} below on the long-time behavior of the master equation is the sharper convergence: 
    \[
        \lim_{T \to \infty} \frac{u^T_t(x)}{T} = \bar{\lambda} \qquad \text{a.s.,}
    \]
which  holds for any fixed \( t \geq t_0 \) and \( x \in \T^d \).
\end{rmk}

The rest of the section is devoted to the proof of Lemma \ref{lem.Wmu} and Theorem \ref{thm.main}. We start with a regularity result. 

\begin{lem}\label{lem.regularity} Fix $C>0$. There exists $C_0>0$, depending on  $f$ and on $C$ only, such that, for any $t_0,T\in \R$ with $t_0<T$ and for any classical solution $(u,m,v)$ to \eqref{eq.MFGCN} such that $\|(D^2u_T)_+\|_\infty\leq C$, a.s., 
 \be\label{lipestiCT}
\|(D^2u_t)_+\|_\infty+ \|Du_t\|_\infty  \leq C_0, \; \qquad \forall t\in [t_0,T], 
\ee
and
\be \label{lipestiCT2}
C^{-1}_0\leq m_t\leq C_0, \qquad  \forall  t\in [t_0+1,T]. 
\ee
\end{lem}

\begin{proof}[Proof of Lemma \ref{lem.regularity}] The proof relies on the shifted system \eqref{eq.MFGCNtilde} introduced above. Let $(\shif u, \shif m, \shif M)$ be defined by \eqref{shifumVsum}-\eqref{shifumVsumbis}. Proposition \ref{prop.shifumVsum} states  that $(\shif u, \shif m, \shif M)$ solves \eqref{eq.MFGCNtilde}. Let us first prove a semi-concavity estimate for $\shif u$: the argument is the same as in \cite{CLLP2}.  Fix $h\in \R^d$ with $|h|=1$ and note that $\shif z_t(x)= \lg D^2\shif u_t(x) h, h\rg$ solves 
\begin{align*}
d\shif z_t(x) & =  \Bigl(-\Delta \shif z_t(x) +D\shif u_t(x)\cdot D\shif z_t(x) + |D^2\shif u_t(x)h |^2  -\lg D^2\shif f_t(x,\shif m_t) h, h\rg  \Bigr)dt +d\shif N_t (x)
\end{align*}
where $\shif N_t(x)= \lg D^2\shif M_t(x)h,h\rg$, with $\shif z_T(x)  \leq \|(D^2\shif u_T)_+\|_\infty= \|(D^2 u_T)_+\|_\infty$. Note that $|D^2\shif u_t(x)h |^2\geq (\shif z_t(x))^2$. Thus 
\begin{align*}
d\shif z_t(x) & \geq   \Bigl(-\Delta \shif z_t(x) +D\shif u_t(x)\cdot D\shif z_t(x) + \shif z_t(x)  -C \Bigr)dt +d\shif N_t (x)
\end{align*} 
where $C$ depends on $\|D^2_{xx}\shif f\|_\infty= \|D^2_{xx}f\|_\infty$. By comparison (see \cite{CSS22}), this implies that $\shif z\leq C\vee  \|(D^2 u_T)_+\|_\infty$ a.s. So we have proved that  $x\to  \shif u_t(x)$ is uniformly semiconcave in $\T^d$, which shows that $D \shif u_t(x)$ is uniformly bounded. Recalling equality  \eqref{shifumVsum}, this in turn implies that $x\to   u_t(x)$ is uniformly semiconcave in $\T^d$ and that $D  u_t(x)$ is uniformly bounded. \\

Following the argument of \cite[Lemma 7.3]{CLLP2} we can check that, for any $t\geq t_0$ and $ h\in (0,1]$, 
$$
\|\shif m_{t+h}\|_\infty \leq C h^{-d/2}\|\shif m_{t}\|_{L^1}, 
$$
which implies that $\shif m$ is bounded in $L^\infty$ for $t\geq t_0+1$. The bound below for $\shif m$ is then ensured for instance by \cite[Theorem 2.5.1]{BKR}: for any $h\in (0, 1]$, there exists $C_h>0$ such that
$$
\|(\shif m_{t+h})^{-1}\|_\infty\leq C_h\|\shif m_{t}\|_\infty .
$$
Using  \eqref{shifumVsum} to translate these estimates to $u$ and $m$  gives \eqref{lipestiCT} and  \eqref{lipestiCT2}.
\end{proof}

Next we state the central result, which shows that the solution of system \eqref{eq.MFGCN} enjoys a turnpike property.

\begin{lem} \label{lem.keyuniqueCT} Fix $C_0>0$. There exists $\omega>0$ depending in $C_0$ only and, for any $\beta\in (0,1)$, $C_\beta>0$, depending on $C_0$ and on $\beta$ only, 
 such that, if $(u_1,m_1,v_1)$ and $(u_2,m_2,v_2)$  are classical solutions to \eqref{eq.MFGCN} on a time interval $(t_0,T)$ (with $t_0+4<T$), satisfying the bounds \eqref{lipestiCT} and \eqref{lipestiCT2} on that time interval and for  the constant $C_0$, then
\be\label{zjq;eksdjnfCT}
\E\left[\|(m_1-m_2)_t\|_{C^\beta}+ \|D(u_1-u_2)_t\|_{L^2_x}\right] \leq C_\beta(e^{-\omega (t-t_0)}+e^{-\omega(T-t)}) \qquad \forall t\in [t_0+2, T-2].
\ee
\end{lem}  

Similar results have been obtained in the deterministic setting in \cite{CLLP2} and \cite{CP19}. In these papers, the proof relies on estimates on linear PDEs and on compactness arguments to ensure the exponential decay.  We  generalize the PDE estimates to the stochastic setting in the Appendix \ref{sec:AppendEsti}.  As for the compactness arguments, they are not available in the stochastic framework and we have to replace them by new quantitative estimates given in the Appendix \ref{sec:AppendExpo}. 

\begin{proof}  Let  $(u_1,m_1,v_1)$ and $(u_2,m_2,v_2)$  be as above. We define by time shift $(\shif u_1,\shif m_1, \shif M_1)$ and $(\shif u_2,\shif M_2, \shif v_2)$ by \eqref{shifumVsum}-\eqref{shifumVsumbis}. Recall that Proposition \ref{prop.shifumVsum} states that $(\shif u_1,\shif m_1, \shif M_1)$ and $(\shif u_2,\shif M_2, \shif v_2)$ solve \eqref{eq.MFGCNtilde}. Let us set $\tilde {\shif u}_{i,t}= {\shif u}_{i,t}-\inte {\shif u}_{i,t}$ (for $i=1,2$) and 
$$
\alpha(t) = \|\tilde{\shif u}_{1,t}-\tilde {\shif u}_{2,t}\|^2_{L^2_{\omega, x}}, \qquad
\beta(t) = \| D(\shif u_{1,t}-\shif u_{2,t})\|^2_{L^2_{\omega, x}}
$$
and
$$
\gamma(t) = \|\shif m_{1,t}-\shif m_{2,t}\|^2_{L^2_{\omega, x}}.
$$
Our aim is to show that the quantities $\alpha$, $\beta$, $\gamma$ satisfy inequalities \eqref{hypexpo1},  \eqref{hypexpo2},  \eqref{hypexpo3},  \eqref{hypexpo4} in  the appendix, and then apply Lemma  \ref{lem.expodecay} which states $\alpha$ and $\gamma$ have an exponential decay. \\

We start with the standard Lasry-Lions monotonicity argument  (see also \cite[Lemma 4.3.11]{CDLL}): we have, for any $t_0+1\leq t_1\leq t_2\leq T$,  
\begin{align}\label{LLmonotonicity}
& \E\left[ \int_{\T^d}(\shif u_{1,t_2}-\shif u_{2,t_2})(\shif m_{1,t_2}-\shif m_{2,t_2})-
\int_{\T^d}(\shif u_{1,t_1}-\shif u_{2,t_1})(\shif m_{1,t_1}-\shif m_{2,t_1}) \right]  \notag\\
& \qquad \leq - \frac12 \E\left[ \int_{t_1}^{t_2} \int_{\T^d} |D\shif u_1-D\shif u_2|^2 (\shif m_1+\shif m_2)\right] . 
\end{align} 
We  claim that 
 \begin{align*}
& \left| \E\left[ \int_{\T^d}(\shif u_{1,t}-\shif u_{2,t})(\shif m_{1,t}-\shif m_{2,t}) \right] \right|  \leq (\alpha(t)\gamma(t))^{1/2}.
 \end{align*}
Indeed, as $\shif m_{1,t}$ and $\shif m_{2,t}$ are both probability measures, 
\begin{align*}
&\left| \E\left[ \int_{\T^d}(\shif u_{1,t}-\shif u_{2,t})(\shif m_{1,t}-\shif m_{2,t}) \right]\right| \\
& = \left| \E\left[ \int_{\T^d}(\tilde{\shif u}_{1,t}-\tilde{\shif u}_{2,t})(\shif m_{1,t}-\shif m_{2,t}) \right]\right| \leq  (\alpha(t)\gamma(t))^{1/2}. 
\end{align*} 
Using the bound below of $m_1$ and $m_2$ in \eqref{lipestiCT2}, we obtain therefore from \eqref{LLmonotonicity} that, for any $t_0+1\leq t_1\leq t_2\leq T$, 
\be\label{kqejhsrndf}
\int_{t_1}^{t_2}\beta(t)dt \leq C_0  \left( (\alpha(t_1)\gamma(t_1))^{1/2}+  (\alpha(t_2)\gamma(t_2))^{1/2}\right).
\ee
Note also that $\mu:= \shif m_2-\shif m_1$ satisfies an equation of the form \eqref{eq.kolmo} with $V= - D\shif u_1$ and $B= \shif m_2D(\shif u_2-\shif u_1)$. Recalling  the bounds for $Du_1$ and $m_2$ in \eqref{lipestiCT}, Lemma \ref{lemCLLP2.1} says that there exists $\omega>0$ and $C>0$ such that, for any $t_0+1\leq t_1\leq t_2\leq T$,  
\begin{align*}
\|\mu(t_2)\|_{L^2}& \leq Ce^{-\omega (t_2-t_1)} \|\mu(t_1)\|_{L^2} +C\left[\int_{t_1}^{t_2} \|B(s)\|_{L^2}^2ds\right]^{1/2}.
\end{align*}
Thus 
\begin{align*}
\gamma(t_2)
& \leq Ce^{-\omega (t_2-t_1)}\gamma(t_1) + C\int_{t_1}^{t_2} \beta(s)ds.
\end{align*}
Next we note that $w= \shif u_{2}-\shif u_{1}$ satisfies a backward equation of the form \eqref{eq.backward_stochastic}, with $\delta =0$, $V= D\shif u_1$ and 
\be\label{lkazjenzrseBIS}
A_t(x)=  \frac12|D(\shif u_{2,t} - \shif u_{1,t})(x)|^2 + \shif f_t(x,\shif m_{1,t}) - \shif f_t(x,\shif m_{2,t}).
\ee
 By Lemma \ref{lemCLLP2.2}, we have therefore 
\begin{align*}
\alpha(t_1) \leq C e^{-\omega (t_2-t_1)} \alpha(t_1) +C\left(\int_{t_1}^{t_2} e^{-\omega(s-t_1)}\|A(s)\|_{L^2(\T^d\times \Omega)}ds\right)^2, 
\end{align*}
where, in view of the estimates in  \eqref{lipestiCT}, and using Cauchy-Schwarz inequality, 
$$
\left(\int_{t_1}^{t_2} e^{-\omega(s-t_1)}\|A(s)\|_{L^2(\T^d\times \Omega)}ds\right)^2 \leq C \int_{t_1}^{t_2} e^{-\omega(s-t_1)}\beta(s)ds + C \sup_{s\in [t_1,t_2]} \gamma(t). 
$$
Thus 
\begin{align*}
\alpha(t_1) \leq C e^{-\omega (t_2-t_1)} \alpha(t_1) +C \int_{t_1}^{t_2} e^{-\omega(s-t_1)}\beta(s)ds + C \sup_{s\in [t_1,t_2]} \gamma(t).
\end{align*}
Finally, by Poincar\'e's inequality, we have, for any $t\in [t_0+1,T]$, 
$$
\alpha(t) \leq C \beta(t).
$$
Thus $\alpha, \beta, \gamma$ satisfy conditions \eqref{hypexpo1},  \eqref{hypexpo2},  \eqref{hypexpo3},  \eqref{hypexpo4} in  the appendix. Lemma  \ref{lem.expodecay} then says that there exist some constants $\omega'>0$ and $C>0$ such that, for any $t\in [t_0+1,T]$, 
\begin{align}\label{auohzeinrjdfg}
\alpha(t)+\gamma(t)& \leq C(e^{-\omega' (t-t_0)}+e^{-\omega'(T-t)}) (\gamma(t_0+1)+\alpha(T))\notag\\ 
& \leq C(e^{-\omega' (t-t_0)}+e^{-\omega'(T-t)}) .
\end{align}
In order to estimate $D(\shif u_2-\shif u_1)$,  we come back to the equation satisfied by $w:= \shif u_{2}-\shif u_{1}$. Recalling inequality \eqref{zerldjkjnpimj2} in Lemma \ref{lemCLLP2.2}, we obtain therefore, for any $t\in [t_0+1, T-3/2]$,
\begin{align*}
  \|Dw_t\|^2_{L^2}  &  \leq C \left( \| \tilde w_{t+1/2} \|^2_{L^2} + \int_{t}^{t+1/2} (\| \tilde w_s\|^2_{L^2}+ \|A(s)\|^2_{L^2})ds \right) \\
& \leq  C\alpha(t+1/2)+ C\int_{t}^{t+1/2} (\alpha+ \beta+\gamma) ds\\
& \leq C (e^{-\omega t}+e^{-\omega(T-t)}),
  \end{align*}
  where we used  in the last line the Lipschitz estimate \eqref{lipestiCT} with \eqref{kqejhsrndf} and \eqref{auohzeinrjdfg}. So we have proved that, for any $t\in [t_0+1, T-3/2]$ (and changing $\omega$),  
  $$
 \E\left[\|(m_1-m_2)(t)\|_{L^2(\T^d)}+ \|D(u_1-u_2)(t)\|_{L^2(\T^d)}\right] \leq C(e^{-\omega (t-t_0)}+e^{-\omega(T-t)}).
 $$

We finally refine this inequality into a $C^\beta$ one for $m_2-m_1$. Recalling the equation satisfied by  $\mu:=\shif m_1-\shif m_2$, we can combine \cite[Theorem III.8.1]{LSU} and \cite[Theorem III.10.1]{LSU} to infer that, for $p\geq 2$ large enough, a.s.,  and for any $t\in [t_0+2, T-2]$, 
$$
\|\mu_t\|_{C^\beta} \leq C_\beta\left( \| \mu_s\|_{L^2((t-1,t)\times \T^d)}+ \|B\|_{L^p((t-1, t)\times \T^d)} \right) 
$$
where $B= \shif m_2(D\shif u_1-D\shif u_2)$. In view of \eqref{lipestiCT} and \eqref{auohzeinrjdfg}, we infer that 
$$
\E\left[ \|\mu_t\|_{C^\beta}\right] \leq C_\beta (e^{-\omega (t-t_0)}+e^{-\omega(T-t)}).
$$
This completes the proof of the lemma. 
\end{proof}

\begin{rmk} \label{rem.fulltime} We note for later use that the constraint $t\geq t_0+1$ in \eqref{auohzeinrjdfg} only comes from the necessity to have a bound below for $m_1+m_2$. If we know that $m_1+m_2$ is bounded below on the whole time interval $[t_0,T]$, then the same arguments show that \eqref{auohzeinrjdfg} holds for $t\in [t_0,T]$. 
\end{rmk}

We are now ready to prove the existence of a stationary solution. 

\begin{proof}[Proof of Lemma \ref{lem.Wmu}] Let us start with the construction of $(\bar{\mathcal V}, \bar \mu)$. Let $(w^T, \mu^T, z^T)$ be the solution to \eqref{eq.MFGCN} on the time interval $(-T,T)$, with initial condition $\mu_{-T}^T=\delta_0$  and terminal condition $w^T_T=0$. According to Lemma \ref{lem.regularity}, the bounds \eqref{lipestiCT} and \eqref{lipestiCT2} hold on $[-T+1, T]$. Thus, for any $t>0$, Lemma \ref{lem.keyuniqueCT} says that $Dw^T$ and $\mu^T$ are Cauchy sequences in $L^1(\Omega\times (-t,t), L^2(\T^d))$ and $L^1(\Omega\times (-t,t), C^\alpha)$ respectively. They converge therefore to processes $\bar{\mathcal V}$ and $\bar \mu$ as $T\to\infty$. Note that $\bar{\mathcal V}$ and $\bar \mu$ are progressively measurable with respect to the filtration $(\mathcal F_t)_{t\in \R}$ and that, as $\bar{\mathcal V}_t$ is the limit of the bounded gradients $Dw^T_t$,  $\bar{\mathcal V}_t$ is itself a.s. a gradient of a $W^{1,\infty}$ function. Moreover, passing to the limit in the equation for $\mu^T_t$ shows that  $\bar \mu$ satisfies, in the sense of distributions and with probability $1$,  
$$
d\bar \mu_t= \left((1+\cwn)\Delta \bar \mu_t +{\rm div}(\bar \mu_t\bar{\mathcal V}_t)\right)dt -\sqrt{2\cwn}D\bar \mu_t\cdot dW_t. 
$$
As $\bar{\mathcal V}$ is bounded, this proves that $\bar \mu$ has continuous path in $\mathcal P$ a.s.. Note also that, by \eqref{zjq;eksdjnfCT}, we have 
\be\label{zjq;eksdjnfCTbis}
\E\left[\|\bar \mu_t-\mu^T_t\|_{C^\beta}+ \|\bar{\mathcal V}_t-Dw^T_t\|_{L^2_x}\right] \leq C_\beta(e^{-\omega (t-t_0)}+e^{-\omega(T-t)}) \qquad \forall t\in [-T+2, T-2].
\ee

Let us now improve the regularity in space of $\bar{\mathcal V}$. We will discuss later its continuity in time. Let $(w^T, \mu^T, z^T)$ be as above and let us define  the shifted solution $(\shif w^T, \boldsymbol{\mu}^T, \shif M^T)$ from $(w^T, \mu^T, z^T)$. We recall that $(\shif w^T,  \boldsymbol{\mu}^T, \shif M^T)$ solves 
\eqref{eq.MFGCNtilde} on the time interval $[-T,T]$ with initial condition $ \boldsymbol{\mu}^T_{-T}=\delta_0$ and terminal condition $\shif w^T_T=0$. Set $\tilde{\shif w}^T_t= {\shif w}^T_t-\inte {\shif w}^T_t$. Then, according to \cite[Lemma 4.3.2]{CDLL}, $\tilde {\shif w}^T_t$ is given by 
$$
\tilde {\shif w}^T_t(x)= \E \left[a_t(x)\ |\ \mathcal F_{-T,t}\right], 
$$
where, for $s\in [t,t+1]$,  
$$
a_s(x)= \Gamma_{t+1-s}\ast \tilde{\shif w}^T_{t+1}(x) + \int_s^{t+1} \Gamma_{u-t}\ast \tilde h_u(x)du.
$$
Here $\Gamma$ is the heat kernel and 
$$
h_s(x) = \frac12 |D\shif w^T_s(x)|^2-\shif f_s(x, \boldsymbol{\mu}^T_s(x)), \qquad \tilde h_s(x)= h_s(x)-\inte h_s(x).
$$
Note that $a$ solves the backward equation (with random coefficient, but solved $\omega$ by $\omega$ and not adapted) 
$$
- \partial_t a_s(x) -\Delta a_s(x) +h_s(x)=0\; \text{in}\; (t,t+1)\times \T^d, \qquad a_{t+1}=\tilde {\shif w}^T_{t+1}. 
$$
According
%\footnote{\cite[Theorem 5.1.2]{Lun} is stated on the whole space, but the adaptation to the torus $\T^d$ is straightforward.} 
to \cite[Theorem 5.1.2]{Lun} $a_t$ can be bounded a.s. by 
$$
\| a_t \|_{C^{1+\beta}(\T^d)} \leq 
C_\beta (\|\tilde {\shif w}^T_{t+1} \|_{C^0(\T^d)}+ \|\tilde h\|_{C^0([t,t+1]\times \T^d)}),
$$
where $\beta\in(0,1)$ is arbitrary. By the Lipschitz estimate of Lemma \ref{lem.regularity} and the assumption on $f$, $\tilde{\shif w}^T_{t+1}$ and $\tilde h$ are uniformly  bounded. This proves that $\tilde {\shif w}^T_t$ is bounded in $C^{1+\beta}$ for $t\in [-T+2,T-2]$ (see Lemma \cite[Lemma 4.3.4]{CDLL}). This in turn implies that $h$---which is a continuous process---is bounded in  $C^{0, \beta'}([-T+1,T-1])$ for any $0<\beta'<\beta$. Then \cite[Theorem 5.1.4-(iv)]{Lun} implies that 
$$
\| a_t\|_{C^{2+\beta'}(\T^d)} \leq 
C_\beta (\|\tilde  {\shif w}^T_{t+1} \|_{C^0(\T^d)}+ \|\tilde h\|_{C^{0,\beta'}([t,t+1]\times \T^d)}),
$$
which in turn implies that $\tilde {\shif w}^T_t$ is uniformly bounded in $C^{2+\beta'}(\T^d)$ for $t\in [-T+2,T-2]$.  Applying the same technique once again, but to the equation satisfied by the components of $D \tilde{\shif w}^T$, we conclude that $\tilde {\shif w}^T_t$ is uniformly bounded in $C^{3+\beta}(\T^d)$ for $t\in [-T+3,T-3]$, for any $\beta \in (0,1)$. Coming back to $Dw^T_t$, we find that, almost surely, for any $t\in\R$, its limit $\bar{\mathcal V}_t$ as $T\to\infty$ is uniformly bounded in $[C^{2+\beta}(\T^d)]^d$ for any $\beta \in (0,1)$.\\

Let us now check that $\bar{\mathcal V}$ and $\bar \mu$ are stationary processes. Recall that the family of shifts $(\tau_h)_{h\in \R}$, where $\tau_h:\Omega\to \Omega$ (where $\Omega=C^0(\R, \R^d)$) is defined by $\tau_h\omega(\cdot)= \omega(\cdot+h)-\omega(h)$. Fix $h>0$ and $t\in \R$. For $T>0$ large enough, we consider the solution $(w^{T,h},\mu^{T,h},z^{T,h})$ to \eqref{eq.MFGCN} on the time interval $(-T-h,T-h)$, with initial condition $\delta_0$ at time $-T-h$ and terminal condition $w^{T,h}_{T-h}=0$. In view of the uniqueness of the solution, we have 
$(w^{T,h},\mu^{T,h},z^{T,h})(t,x,W)= (w^T,\mu^T,z^T)(t+h, x, \tau_hW)$, where $(w^T, \mu^T, z^T)$ is defined above. On the other hand, by Lemma \ref{lem.keyuniqueCT}, we have 
\begin{align*}
&
\E\left[\|(\mu^{T}_{t+h}(\cdot, \tau_h\cdot)-\mu^{T}_t(\cdot, \cdot))\|_{C^\alpha}+ \|D(w^{T}_{t+h}(\cdot, \tau_h\cdot)-w^{T}_t(\cdot, \cdot))\|_{L^2(\T^d)}\right]\\
& = \E\left[\|(\mu^{T,h}_t-\mu^{T}_t)\|_{C^\alpha}+ \|D(w^{T,h}_t-w^{T}_t)\|_{L^2(\T^d)}\right] \leq C(e^{-\omega (t+T-|h|)}+e^{-\omega(T-t+|h|)}). 
\end{align*}
Letting $T\to \infty$ we obtain therefore that $\bar \mu (t+h, x, \tau_hW)= \bar \mu (t, x, W)$ and $\bar{\mathcal V} (t+h, x, \tau_hW)= \bar{\mathcal V} (t, x, W)$. Thus   $\bar{\mathcal V}$ and $\bar \mu$ are stationary processes. \\

Given an horizon $T>0$, let $\bar w_T$ be the $\mathcal F_T-$measurable random variable such that $D\bar w_T=\bar{\mathcal V}_T$ and $\bar w_T(0)=0$. By the regularity of $\bar{\mathcal V}_T$, $\bar w_T$ is bounded in $C^{3+\beta}$ for any $\beta \in (0,1)$.  
By Proposition \ref{prop.HJ}, there exists a unique strong solution $(w,z)$ of the backward HJ equation,
\be\label{def.wt}
dw_t= \left(-(1+\cwn)\Delta w_t +\frac12 |Dw_t|^2 -f(x,\bar \mu_t) -2\cwn{\rm div}(z_t)\right)dt +\sqrt{2\cwn}z_t\cdot dW_t 
\ee
with terminal condition $w_T=\bar w_T$ such that 
\[
\esssup \sup_{t \in [t_0, T]} \normU{w_t}{C^{3+\beta}} < \infty.
\]
We are going to show that $Dw_t= \bar{\mathcal V}_t$, which will show that $(w, \bar \mu, z)$ is the unique weak solution to \eqref{eq.MFGCN} on $[t_0,T]$ with initial condition $\bar\mu_{t_0}$ and terminal condition $w_T= \bar w_T$ and will therefore complete the proof of the lemma.\\ 

To show that $Dw_t= \bar{\mathcal V}_t$, fix $T_1>|T|+|t_0|$ large. We  consider $(w^{T_1}, \mu^{T_1}, z^{T_1})$ the solution to \eqref{eq.MFGCN} on the time interval $(-T_1,T_1)$, with initial condition $\mu^{T_1}_{-T_1}=\delta_0$ and terminal condition $w^{T_1}_{T_1}=0$ at time $T_1$. Recalling Lemma \ref{lem.representation}, let  $\xi$ be the unique $(\mathcal F_t)-$adapted process such that 
$$
\E\left[w^{T_1}_T(0)\;|\; \mathcal F_t\right] = w^{T_1}_T(0)- \int_t^T \xi_s\cdot dW_s\qquad \forall t\in [-T,T].
$$
Note that $(w^{T_1}_t-\E[w^{T_1}_T(0)\;|\; \mathcal F_t], \mu^{T_1}, z^{T_1}-\xi)$ is a classical solution to \eqref{eq.MFGCN} on the time interval $(-T_1,T)$, with initial condition $\mu^{T_1}_{-T_1}=\delta_0$ at time $-T_1$ and terminal condition $w^{T_1}_{T}-w^{T_1}_{T}(0)$ at time $T$. 
Let us set  
$$
a_t(x)= w_t(x)-w^{T_1}_t(x)+\E\left[w^{T_1}_T(0)\;|\; \mathcal F_t\right], \qquad b_t(x)= z_t(x)-z^{T_1}_t(x) -\xi_t,
$$
By the comparison principle in Proposition \ref{prop.HJ} applied to the two strong solutions of the HJ equation, $(w,z)$ and $(w^{T_1}_t-\E[w^{T_1}_T(0)\;|\; \mathcal F_t],z^{T_1}-\xi)$, we have 
\begin{align*}
\E\left[\|a_t\|_\infty \right] & \leq \E\left[\|a_T\|_\infty\right]+  \int_t^T \E\left[\|f(\cdot,\bar \mu^T_s)-f(\cdot, \mu^{T_1}_s)\|_\infty\right] ds \\ 
& \leq \E\left[\|a_T\|_\infty\right]+ C \int_t^T \E\left[{\bf d}_1(\bar \mu_s, \mu^{T_1}_s)\right] ds.
\end{align*}
Thus by \eqref{zjq;eksdjnfCTbis}, the fact that $a_T(0)=0$, and Sobolev inequality, we get, for $t\geq t_0$ and $d<p<\infty$,  
\begin{align}\label{esti.|ut|}
\E\left[\|a_t\|_\infty \right] & \leq C \E\left[\|\bar{\mathcal V}_T- Dw^{T_1}_T\|_{L^p_x}\right]+ C \int_t^T \E\left[{\bf d}_1(\bar \mu_s, m^{T_1}_s)\right] ds \notag\\
& \leq C \E\left[\|\bar{\mathcal V}_T- Dw^{T_1}_T\|_{L^2_x}^{\frac 2p} \|\bar{\mathcal V}_T- Dw^{T_1}_T\|_{L^\infty_x}^{1-\frac 2p} \right]+ C \int_t^T \E\left[{\bf d}_1(\bar \mu_s, \mu^{T_1}_s)\right] ds \notag\\
& \leq C(e^{-\frac 2p \omega (T_1-T)}+e^{-\frac 2p \omega (T+T_1)}) +C \int_t^T (e^{-\omega (T_1-s)}+e^{-\omega (s+T_1)})ds\notag \\
& \leq C e^{-\frac 2p \omega(T_1-T)}. 
\end{align}
Note that $(a,b)$ solves  
$$
da_t=  \left(-(1+\cwn)\Delta a_t +\frac12 (|Dw_t|^2 -|Dw^{T_1}_t|^2)-f(x,\bar \mu_t)+f(x,\mu^{T_1}_t) -2\cwn{\rm div}(b_t)\right)dt +\sqrt{2\cwn}b_t\cdot dW_t
$$
with terminal condition $a_T(x)= \bar w_T(x)-w^{T_1}_T(x)+w^{T_1}_T(0)$. 
By shifting the system, and then using It\^{o}'s product formula for functionals (see \cite[Lemma 4.3.11]{CDLL}), in view of the equation satisfied by $\bar \mu$ and $(a,b)$, we have
$$
d\inte  a_t \bar \mu_t = \left( \inte \frac12 (|Dw_t|^2 -|Dw^{T_1}_t|^2)\bar \mu_t -(f(x,\bar \mu_t)-f(x,\mu^{T_1}_t))\bar \mu_t- Da_t\cdot \bar{\mathcal V}_t\bar \mu_t \right) dt + d M_t,
$$
where $M$ is a martingale. Hence 
\begin{align*}
&\E\left[ \int_{t_0}^T  \inte  \left(\frac12(|Dw_t|^2 -|Dw^{T_1}_t|^2)\bar \mu_t -D(w_t-w^{T_1}_t)\cdot \bar{\mathcal V}_t\bar \mu_t \right) dt\right]\\ 
 & \leq \E\left[ \int_{t_0}^T  \inte (f(x,\bar \mu_t)-f(x,\mu^{T_1}_t))\bar \mu_t dt\right] + \E\left[\inte a_T \bar \mu_T- a_0 \bar \mu_0\right] 
 \end{align*}
where, on the one hand,  
\begin{align*}
&\E\left[ \int_{t_0}^T  \inte \left(\frac12 (|Dw_t|^2 -|Dw^{T_1}_t|^2)\bar \mu_t -D(w_t-w^{T_1}_t)\cdot \bar{\mathcal V}_t\bar \mu_t \right) dt\right]\\ 
&= \E\left[ \int_{t_0}^T  \inte  (Dw_t -Dw^{T_1}_t)\cdot ( \frac12 (Dw_t+Dw^{T_1}_t) -\bar{\mathcal V}_t)\bar \mu_t   dt\right]\\ 
& \geq \E\left[ \int_{t_0}^T  \inte  \frac12 |Dw_t -Dw^{T_1}_t|^2\bar \mu_t   dt\right]
- \E\left[ \int_{t_0}^T  \inte  |Dw_t -Dw^{T_1}_t| |Dw^{T_1}_t -\bar{\mathcal V}_t| \bar \mu_t   dt\right] \\
& \geq \frac14 \E\left[ \int_{t_0}^T  \inte   |Dw_t -Dw^{T_1}_t|^2\bar \mu_t   dt\right]
- C \E\left[ \int_{t_0}^T  \inte  |Dw^{T_1}_t -\bar{\mathcal V}_t|^2 \bar \mu_t   dt\right].
 \end{align*}
On the other hand, by \eqref{esti.|ut|},  
\begin{align*}
\left| \E\left[\inte a_T \bar \mu_T- a_0 \bar \mu_0\right] \right| \leq C e^{-\frac 2p \omega(T_1-T)}. 
\end{align*}
Hence 
\begin{align*}
 & \E\left[ \int_{t_0}^T  \inte  |Dw_t -Dw^{T_1}_t|^2\bar \mu_t   dt\right]  \\ 
 & \leq C \E\left[ \int_{t_0}^T  \inte (|Dw^{T_1}_t -\bar{\mathcal V}_t|^2 + {\bf d}_1(\bar \mu_t, \mu^{T_1}_t))  \bar \mu_t   dt\right] +C e^{-\frac 2p \omega(T_1-T)}.
\end{align*}
Using \eqref{zjq;eksdjnfCTbis},  we obtain $Dw_t=\bar{\mathcal V}_t$ by letting  $T_1\to \infty$. Note finally that, by construction, $w$ is a continuous process, which shows that $\bar{\mathcal V}$ has a continuous representative and---in view of our previous estimate---has path in $[C^0(\R, C^{2})]^d$. 
\end{proof}

We now turn to the proof of the main result of the section: 

\begin{proof}[Proof of Theorem \ref{thm.main}] Let $(u^T,m^T,z^T)$ be a solution of the MFG system \eqref{eq.MFGCN} on $[t_0,T]$ with initial condition $\hat m_0$ and a terminal condition satisfying \eqref{cond.termcond}. Note that the bounds in Lemma \ref{lem.regularity} hold for this solution. Let also $(w, m, z)= (w, \bar \mu, z)$ be solution of the MFG system \eqref{eq.MFGCN}  defined in Lemma \ref{lem.Wmu}. Let us first check that \eqref{ineq.expo} holds. For this let  $(w^T, \mu^T, z^T)$ be the solution of the MFG system on the time interval $[-T,T]$ with initial condition $\mu^T_{-T}=\delta_0$ and $w^T_T=0$ as introduced in the proof of Lemma \ref{lem.Wmu}. Then, for any $|T'|>T+|t_0|$, we have
by inequality \eqref{zjq;eksdjnfCT} in Lemma \ref{lem.keyuniqueCT}, 
%
% and the construction of $\bar \mu$, $\bar{\mathcal V}$ in the proof of Lemma \ref{lem.Wmu}, because, if $T>0$ and $(u^T, m^T, v^T)$ is defined as in that proof on the time interval $[-T,T]$ and if $T'>T$, then 
$$
\E\left[\|(m^T-\mu^{T'})_t\|_{C^\beta}+ \|D(u^T-w^{T'})_t\|_{L^2_x}\right] \leq C_\beta(e^{-\omega (t-t_0)}+e^{-\omega(T-t)}) \qquad \forall t\in [t_0+2, T-2].
$$
Letting $T'\to \infty$ gives   \eqref{ineq.expo} as $\mu^{T'}$ converges to $\bar \mu$ and $D w^{T'}$ to $\bar{\mathcal V}$ (see the proof of Lemma \ref{lem.Wmu}). 

For $t\in [t_0,T]$, we set  $\hat w_{t} := \inte w_{t}$ and  $\hat  z_{t} :=\inte z_{t}$. By \eqref{eq.MFGCN}-(i) we have  
$$
d\hat  w_t = \left( \inte(\frac12 |Dw_t|^2 -f(x,\bar \mu_t)) \right)dt + \hat  z_t \cdot dW_t
$$
Therefore, as $Dw_{t}= \bar{\mathcal V}_{t}$,  
\begin{align*}
\frac1T \hat  w_{t_0}  & = \E\left[ \frac{1}{T} \inte \bar w_T- \frac{1}{T} \int_{t_0}^T \inte (\frac12 |Dw_s|^2 -f(x,\bar \mu_s)) ds \ |\ \mathcal F_{t_0}\right]\\
&  = \E\left[ \frac{1}{T} \inte \bar w_T- \frac{1}{T} \int_{t_0}^T \inte (\frac12 |\bar{\mathcal V}_s|^2 -f(x,\bar \mu_s)) ds \ |\ \mathcal F_{t_0}\right] 
\end{align*}
As $(\bar{\mathcal V}, \bar \mu)$ is stationary we obtain by the ergodic theorem: 
\be\label{kjhlaenrsdfgcv}
\lim_{T\to\infty}  \frac{1}{T} \int_{t_0}^T \inte (\frac12 |\bar{\mathcal V}_s|^2 -f(x,\bar \mu_s)) ds 
=
\E\left[  \inte (\frac12 |\bar{\mathcal V}_0|^2 -f(x,\bar \mu_0))\right] \;  \text{a.s. and in $L^1$.}
\ee
Thus, as $\bar w_T$ uniformly bounded, we get
$$
\lim_{T\to\infty} \frac1T \hat  w_{t_0}= - \E\left[  \inte (\frac12 |\bar{\mathcal V}_0|^2 -f(x,\bar \mu_0))\right]=:\bar \lambda \qquad  \text{a.s. and in $L^1$.}
$$
As $w_0$ is uniformly Lipschitz, we infer that, for any $x\in \T^d$, 
$$
\lim_{T\to\infty} \frac1T w_0(x)= - \E\left[  \inte (\frac12 |\bar{\mathcal V}_0|^2 -f(x,\bar \mu_0))\right], \qquad  \text{a.s. and in $L^1$,}
$$
the limit being uniform in $x$. \\

Let us now compare $w$ and $u^T$, both being a strong solution of the same HJ equation with different right-hand side and different terminal condition. By comparison principle (Proposition \ref{prop.HJ}), we have 
\begin{align*}
\E\left[ \|w_{t_0}-u^T_{t_0}\|_\infty\right] & \leq \E\left[ \|w_T-u^T_T\|_\infty+ \int_{t_0}^T \|f(\cdot, \bar \mu^T_t)-f(\cdot, m^T_t)\|_\infty dt\right]  \\ 
& \leq C+  C\int_{t_0}^T \E\left[ {\bf d}_1(\bar \mu^T_t, m^T_t)\right]  dt  \leq C+ C\int_1^{T-1} (e^{-\omega(T-t)}+e^{-\omega t})dt.
\end{align*}
This implies the convergence of $\frac{1}{T} u^T_{t_0}(x)$ to $\bar \lambda$ in $L^1$, uniformly in $x$. 
\end{proof} 

We finally discuss the long-time behavior of suitable solutions to the MFG system. This result will play a central role for the long-time behavior of the master equation. 

\begin{lem}\label{lem.keyCV} For $T<0$, let $(u^T, m^T, v^T)$ be the solution of the MFG system \eqref{eq.MFGCN} defined on $[T,0]$, with initial condition $m^T_T=\bar \mu_T$ (where $\bar \mu$ is the stationary process defined in Lemma \ref{lem.Wmu}) and terminal condition $g(\cdot, m^T_0)$. There exists a real constant $\bar c$ such that 
$$
\lim_{T\to -\infty} u^T_T(\cdot) - \bar u_T(\cdot) = \bar c\qquad \text{in}\; L^1(\Omega), 
$$
uniformly in $x$, where $(\bar u, \bar v)$ is the strong solution to the HJ equation
$$
d\bar u_t = \left(-(1+\cwn)\Delta \bar u_t +\frac12 |D \bar u_t|^2 -f(x, \bar \mu_t) -2\cwn {\rm div}(\bar v_t) \right)dt + \sqrt{2 \cwn} \bar v_t\cdot dW_t
$$
on the time interval $(-\infty, 0)$ with terminal condition at time $t=0$ given by $\bar u_0$, where $\bar u_0$ is the $\mathcal F_0-$random variable such that $D\bar u_0=\bar{\mathcal V}_0$, $\bar u_0(0)=0$. 
\end{lem} 

{\bf Remarks.} \\
1) Note carefully that here, in contrast with the convention we used so far, $T$ is the negative initial time while $0$ is the terminal time. 

2) The existence and the uniqueness of $(\bar u, \bar v)$ is part of the lemma. We will see in the proof that actually $(\bar u, \bar \mu, \bar v)$ is a weak solution to the MFG system \eqref{eq.MFGCN} on the time interval $(-\infty, 0)$ with terminal condition $\bar u_0$. As a consequence, we have $D\bar u= \bar{\mathcal V}$ a.s..  We also want to mention that, although we use the notation $(\bar u,\bar v)$, here $(\bar u,\bar v)$ is {\it not} a stationary process.

3) As we will see in the proof and in contrast with the previous results of this section, here the fact that the terminal condition of the MFG system is precisely $g(\cdot, m^T_0)$ is crucial.

\begin{proof}[Proof of Lemma \ref{lem.keyCV}] Let us first discuss the existence of $(\bar u, \bar v)$. Let $T<0$ and $(\bar u^T, \bar \mu^T, \bar v^T)$ be the unique weak solution to \eqref{eq.MFGCN} with initial condition $\bar \mu_T$ at time $T$ and terminal condition $\bar u_0$. We know by Lemma \ref{lem.Wmu} that $\bar \mu^T= \bar \mu$ and $D\bar u^T= \bar{\mathcal V}$ on $[T,0]$. By the uniqueness of the strong solution of the backward HJ equation (Proposition \ref{prop.HJ}), this implies that $(\bar u^T, \bar v^T)= (\bar u^{T'}, \bar v^{T'})$ on $[T,0]$ for any $T'\leq T$. This implies the existence and the uniqueness of $(\bar u, \bar v)$ defined on $(-\infty,0]$ such that $(\bar u, \bar v)= (\bar u^T, \bar v^T)$ on $[T,0]$ for any $T<0$. Note that $D\bar u= \bar{\mathcal V}$. \\

For $T<0$, let now $(u^T, m^T, v^T)$ be the solution to \eqref{eq.MFGCN} on $[T,0]$ with initial condition $\bar \mu_T$ and terminal condition $g(\cdot, m^T_0)$. Let us first compare $u^T$ and $u^{T'}$ for $T'<T<0$. For this we argue as in the proof of Lemma \ref{lem.keyuniqueCT}. Fix $\tau \in  (T\vee T'+1, -1)$. By Lasry-Lions monotonicity argument on the time interval $[\tau, 0]$  (see \cite[Lemma 4.3.11]{CDLL}), we have:  
\begin{align*}
\E\left[ \int_{\tau}^0 \inte (m^T+m^{T'})|Du^T-Du^{T'}|^2\right] & \leq C\E\left[ \inte (u^T_{\tau}-u^{T'}_{\tau})(m^T_\tau-m^{T'}_\tau)\right]  \\ 
& \leq C\E\left[ \|Du^T_\tau-Du^{T'}_\tau\|_{L^2_x}\|m^T_\tau-m^{T'}_\tau\|_{L^2_x}\right], 
\end{align*}
where the last line holds by Poincar\'e's inequality. On the other hand, we know from Lemma \ref{lem.keyuniqueCT} that, for any $t\in [T\vee T'+2, -2]$,   
\be\label{kuzqejsh;bdnfgvcb}
 \E\left[\|(m^T-m^{T'})_t\|_{C^\alpha}+ \|D(u^T-u^{T'})_t\|_{L^2(\T^d)}\right] \leq C(e^{-\omega (t-(T\vee T'))}+e^{\omega t}) .
\ee
Thus, using Lemma \ref{lem.regularity} to bound below $m^T$ and $m^{T'}$, 
$$
\E\left[ \int_{\tau}^0 \inte |Du^T-Du^{T'}|^2\right] \leq C  (e^{-2\omega (\tau-(T\vee T'))}+e^{2\omega \tau}) .
$$
Applying Lemma \ref{lemCLLP2.1} to the equation satisfied by the shifted version of $m^T-m^{T'}$, we have, for some $\omega>0$ depending on $\|Du^T\|_\infty$ and $\|Du^{T'}\|_\infty$ only (which are bounded independently of $T,T'$), and for any $t\in [\tau, 0]$,
\begin{align}
\E\left[\|m^T_t-m^{T'}_t\|^2_{L^2_x}\right] 
& \leq Ce^{-2\lambda t} \E\left[\|m^T_\tau-m^{T'}_\tau\|^2_{L^2_x}\right] +C\E\left[\int_\tau^t \|m^{T'}_s |Du^T_s-Du^{T'}_s|\|_{L^2_x}^2  ds\right] \notag\\ 
& \leq C  (e^{-2\omega (\tau-(T\vee T'))}+e^{2\omega \tau}), \label{kqjzesjdfhcg}
\end{align}
changing $\omega$ and $C$ if necessary. As $u^T$ and $u^{T'}$ solve the same HJ equation with different right-hand side, we infer by comparison principle (Proposition \ref{prop.HJ}) and \eqref{kqjzesjdfhcg} that (restricting $\omega$ if necessary)
\begin{align}\label{qskdlifnm}
\E\left[\|u^T_\tau-u^{T'}_\tau\|_{L^\infty_x}\right] &  \leq \E\left[ \|g(\cdot, m^T_0)-g(\cdot, m^{T'}_0)\|_\infty\right] + \E\left[ \int_\tau^0 \|f(\cdot, m^T_s)-f(\cdot, m^{T'}_s)\|_\infty ds \right]\notag  \\ 
& \leq   C  (e^{-\omega (\tau-(T\vee T'))}+e^{\omega \tau}).  
\end{align}

We now compare the $u^T$ to $\bar u$. Define the $\mathcal F_\tau-$measurable random variable $c^T_\tau$ by
$$
c^T_\tau = \inte (u^T_\tau- \bar u_\tau). 
$$
By  Poincar\'e's inequality and Theorem  \ref{thm.main}, we have 
\be\label{kzelsjkdfng}
\E\left[ \|u^T_\tau- \bar u_\tau- c^T_\tau \|^2_{L^2_x} \right] \leq C \E\left[ \|Du^T_\tau- D\bar u_\tau\|^2_{L^2_x} \right] 
= C \E\left[ \|Du^T_\tau- \bar{\mathcal V}_\tau\|^2_{L^2_x} \right] 
\leq C(e^{-2\omega (\tau-T)}+e^{2\omega \tau}).
\ee
This implies by  \eqref{qskdlifnm}  that
\be\label{kzelsjkdfng4}
\E\left[ |c^T_\tau-c^{T'}_\tau |\right] \leq  C(e^{-\omega (\tau-(T\vee T'))}+e^{\omega \tau}).
\ee
We now compare  $(u^T, m^T, v^T)$  and $(\bar u + \kappa^{T,\tau}, \bar \mu, \bar v+\xi)$, where 
$$
\kappa^{T,\tau}_t= \E\left[ c^T_\tau\ |\ \mathcal F_t\right]
$$ 
and (recalling Lemma \ref{lem.representation}) $\xi$ is the adapted process such that 
$$
 \E\left[ c^T_\tau\ |\ \mathcal F_t\right] =  c^T_\tau-\int_t^\tau \xi_sdW_s.
 $$
 Both are weak solutions to the MFG system \eqref{eq.MFGCN} on $[T, \tau]$, with the same initial condition  $\bar \mu_T$ at time $T$ and with close  terminal conditions at time $\tau$ thanks to \eqref{kzelsjkdfng}. As $\bar \mu$ is bounded below, Remark \ref{rem.fulltime} after the proof of Lemma \ref{lem.keyuniqueCT} shows that in this setting, for each $t\in [T, \tau]$,
 \begin{align*}
 \E\left[\|(m^T-\bar \mu)_t\|^2_{L^2_x}+ \|\widetilde{(u^T-\bar u)}_t\|^2_{L^2_x}\right] & \leq C  (e^{-\omega (t-T)}+e^{-\omega(\tau-t)}) 
 \E\left[\|(m^T-\bar \mu)_T\|^2_{L^2_x}+ \|\tilde u^T-\tilde{\bar u})_\tau\|^2_{L^2_x}\right] \\
 &  \leq C (e^{-2\omega (\tau-T)}+e^{2\omega \tau})  ,
 \end{align*}
 where, for the last inequality, we used the fact that $m^T_T=\bar \mu_T$ and \eqref{kzelsjkdfng}. 
Again, by Proposition  \ref{prop.HJ}, a comparison argument on the HJ equations satisfied by $(u^T, v^T)$ and $(\bar u^T+ \kappa^{T,\tau},\bar v+\xi)$  implies that 
\begin{align*}
\E\left[\| u^T_T - \bar u_T- \kappa^{T,\tau}_T\|_\infty\right]& \leq  \E\left[\| u^T_\tau - \bar u_\tau- \kappa^{T,\tau}_\tau\|_\infty\right] + \E\left[ \int_T^\tau \|f(\cdot, m^T_s)-f(\cdot, \bar \mu_s)\|_\infty ds\right] \\
& \leq C(e^{-\omega (\tau-T)}+e^{\omega \tau})  .
\end{align*}
Let us set $\bar c_T=c^T_T= \inte (u^T_T-\bar u_T)$. Then 
\be\label{zkjehqsdnfgc1}
\E \left[ |\bar c_T-\kappa^{T, \tau}_T|\right]  \leq \E\left[\| u^T_T - \bar u_T- \kappa^{T,\tau}_T\|_\infty\right]  \leq C(e^{-\omega (\tau-T)}+e^{\omega \tau})  . 
\ee
Our next goal is to prove that $(\bar c_T)$ is Cauchy in $L^1$ and that it converges to a constant $\bar c$ as $T\to-\infty$. For this we fix $\ep>0$ small. Recalling \eqref{kzelsjkdfng4}, we also fix $T_0$ and $\tau$ such that 
\be\label{lieqbdsf}
(e^{-\omega (\tau-T)}+e^{\omega \tau}) + \E\left[ |c^T_\tau-c^{T_0}_\tau|\right]\leq \ep \qquad \forall T\leq T_0.
\ee
Note that, as $s\to-\infty$, $(\kappa^{T_0, \tau}_s)$  converges in $L^2_\omega$ to $\E\left[c^{T_0}_\tau \ |\ \cap_{t} \mathcal F_t \right]$. This limit is a.s. constant by Kolmogorov 0-1 law and thus equal to $\E\left[c^{T_0}_\tau \right]$. Let us choose $s_0<0$ such that 
$$
\E\left[ | \kappa^{T_0,\tau}_{s}- \E[c^{T_0}_\tau]|\right]  \leq \ep \qquad \forall s\leq s_0.
$$
Combining this inequality with \eqref{lieqbdsf},  we have
\be\label{zkjehqsdnfgc2}
\E\left[ | \kappa^{T,\tau}_{s}- \E[c^T_\tau]|\right] \leq C\ep \qquad \forall T\leq T_0, \; s\leq s_0. 
\ee
Therefore, for $T' \leq T\leq T_0\wedge s_0$,  
\begin{align*}
\E\left[ |\bar c_T- \bar c_{T'}|\right] 
& \leq \E\left[ |\kappa^{T, \tau}_T-\kappa^{T',\tau}_{T'}|\right] + \E\left[| \kappa^{T,\tau}_T-\bar c_T|\right]+ 
\E\left[ | \kappa^{T',\tau}_{T'}-\bar c_{T'}|\right] \\ 
& \leq \E\left[ | \kappa^{T, \tau}_{T}-\kappa^{T,\tau}_{T'}|\right] + \E\left[ |\kappa^{T,\tau}_{T'}- \kappa^{T',\tau}_{T'}|\right]  + C(e^{-\omega (\tau-T)}+e^{\omega \tau}) \\ 
& \leq C\ep  + \E\left[ |c^T_\tau-c^{T'}_\tau|\right]  +C(e^{-\omega (\tau-T)}+e^{\omega \tau}) \leq C\ep, 
\end{align*}
for $T$ and $\tau$ small enough, where we used  \eqref{zkjehqsdnfgc1}  in the second inequality and, for the third one, the definition of  $\kappa^{T,\tau}_{T'}=\E\left[ c^{T}_\tau\; |\; {\mathcal F}_{T'}\right]$  and the fact that \eqref{lieqbdsf} and \eqref{zkjehqsdnfgc2} hold. So $(\bar c_T)$ is a Cauchy sequence in $L^1$. Note also that, by \eqref{zkjehqsdnfgc1}  and \eqref{zkjehqsdnfgc2} and for $-T$ large,
\begin{align*}
\E\left[ |\bar c_T- \E\left[ \bar c_{T}\right] |\right] \leq  \E\left[| \kappa^{T,\tau}_T-\bar c_T|\right] +
 \E\left[| \E\left[\kappa^{T,\tau}_T\right]-\E\left[\bar c_T\right]|\right]  \leq C\ep. 
\end{align*}
Therefore $(\bar c_T)$ converges to a constant. 
\end{proof}

%%%%%%%%%%%%%%%%%%%%%%%%%%%%%
\section{The discounted problem} \label{sec.discount}

This part is dedicated to the analysis of the discounted MFG systems with a small discount factor. The understanding of this behavior  is instrumental to build a solution to the ergodic master equation. The discounted MFG system reads: 
\be\label{eq.MFGCNdisc}
\left\{
\begin{array}{cl}
    (i) & du^\delta_t= \left(-(1+\cwn)\Delta u^\delta_t +\delta u^\delta+\frac12 |Du^\delta_t|^2 -f(x,m^\delta_t) -2\cwn{\rm div}(v^\delta_t)\right)dt +\sqrt{2\cwn} v^\delta_t\cdot dW_t \\ [1ex]
    (ii) & dm^\delta_t=\left((1+\cwn)\Delta m^\delta_t +{\rm div}(m^\delta_tDu^\delta_t)\right)dt -\sqrt{2\cwn}Dm^\delta_t\cdot dW_t \\ [1ex]
    (iii) & m^\delta_{t_0}= \hat m_0,  \qquad u^\delta\;\text{bounded in $\Omega\times [t_0,\infty)\times \T^d$.}
\end{array}
\right.
\ee
Here the unknown is $(u^\delta, m^\delta, v^\delta)$ and the system is stated in $[t_0, \infty)\times \T^d$. As for the time dependent problem, the coupling function $f$ satisfies our standing assumptions {\bf (Hf), (Hm)}. The discount rate $\delta$ is a positive constant. Finally, $\hat m_0\in \mathcal P(\T^d)$ is arbitrary. Under these standing assumptions, an adaptation of the construction of \cite{CDLL} shows that the MFG system \eqref{eq.MFGCNdisc} has a unique solution which is an $(\mathcal F_{t_0,t})_{t\in [t_0,\infty)}-$adapted process with paths in $C^0([t_0,\infty], C^{4}(\T^d)\times \mathcal P\times [C^{3}(\T^d)]^d)$, such that $\sup_t (\|\uD\|_{4+\alpha}+\|\vD\|_{3+\alpha})$ is in $L^\infty(\Omega)$. Equation  \eqref{eq.MFGCNdisc}-(i) holds in a classical sense with probability 1, while, still with probability 1, \eqref{eq.MFGCNdisc}-(ii) is understood in the sense of distributions. We will often study the above system for an $\mathcal F_{t_0}-$random measure $\hat m_0$, in which case the solution will be an $(\mathcal F_{t})_{t\in [t_0,\infty)}-$adapted process. \\

Our first result is the existence of a (unique) weak stationary solution for the discounted problem. 

\begin{lem} \label{lem.WmuDISC} There exists $\delta_0>0$ and, for any $\delta \in (0,\delta_0)$, there exists a pair $(\bar{\mathcal V}^\delta, \bar \mu^\delta)$ of stationary processes adapted to the filtration $(\mathcal F_t)$, with path in $\mathcal C^0(\R, [C^{2+\beta}(\T^d)]^d\times \mathcal P)$, with
\[
\esssup \sup_t (\normUU{\bar{\mathcal{V}}^\delta_t}{[C^{2+\beta}(\T^d)]^d} + \normUU{\bar{\mu}^\delta_t}{C^{\beta}(\T^d)})<\infty,
\]
where $\beta$ is arbitrary in $[0,1)$, such that, $\bar{\mathcal V}^\delta_t$ is a gradient for any $t$ a.s.. and such that, for any $T\in \R$, the unique  solution $(w^\delta, m^\delta, z^\delta)$ to \eqref{eq.MFGCNdisc} on the time interval $(T, \infty)$ with initial condition $m_{T}= \bar \mu^\delta_T$ satisfies $Dw^\delta_t= \bar{\mathcal V}^\delta_t$ and $m^\delta_t= \bar \mu^\delta_t$ for any $t\in [T,\infty)$. 
\end{lem}

We explain now how this stationary solution attracts all the solutions to \eqref{eq.MFGCNdisc} and allows to understand its long-time average: 

\begin{thm}\label{thm.Cudelta} 
Let $(\bar{\mathcal{V}}, \bar \mu)$ be as defined by Lemma \ref{lem.Wmu} and $(\bar{\mathcal{V}}^\delta, \bar{\mu}^\delta)$ be as defined by Lemma \ref{lem.WmuDISC}. There are constants $\delta_0, \omega>0$ and, for any $\beta\in(0,1)$, $C_\beta > 0$, such that, if $\delta \in (0, \delta_0)$ and if $(u^\delta, m^\delta, v^\delta)$ is the solution to \eqref{eq.MFGCNdisc}, then 
\be\label{ineqcontracdelta}
 \E\left[\normH{ m^\delta_{t}- \bar \mu^\delta_{t}}+ \|Du^\delta_{t}-\bar{\mathcal V}^\delta_{t}\|_{L^2(\T^d)}\right] \leq C_\beta e^{-\omega (t-t_0)} \qquad \forall t\in [t_0+1, \infty)
\ee
where 
\be\label{ineqcontracdelta4}
\E\left[ \inte |\bar{\mathcal V}^\delta_0-\bar{\mathcal V}_0|^2dx\right]+  \E\left[ \inte |\bar \mu^\delta_0-\bar \mu_0|^2 dx \right] \leq C\delta. 
\ee
In addition, for any $t\geq t_0$, 
\be\label{ineqcontracdelta3}
\E\left[ \sup_{x\in \T^d} |\delta u^\delta_{t}(x)-\bar \lambda |  \right] \leq C\delta
 \ee
 where $\bar \lambda$ is the constant defined in Theorem \ref{thm.main}. 
\end{thm}

The rest of the section is devoted to the proofs of Lemma \ref{lem.WmuDISC} and Theorem \ref{thm.Cudelta}. As in Section~\ref{sec.timedep}, to manipulate system \eqref{eq.MFGCNdisc}  it is often convenient to change variables and work with the shifted system 
\begin{align}\label{eq.MFGCNtildedelta}
& (i)\; d\uDW= \left(\delta \uDW -\Delta \uDW +\frac12 | D  \uDW|^2 -\shif f_t(x,\mDW) \right)dt +d\MDW \notag\\
& (ii)\; d\mDW=\left(\Delta \mDW +{\rm div}(\mDW D  \uDW)\right) dt \\
& (iii)\; \shif m^\delta_{t_0}= \hat m_0, \qquad \shif u^\delta\; {\rm bounded} \notag
\end{align}
where 
$$
\shif f_t(x,m)= f(x+\sqrt{2\cwn}(W_t-W_{t_0}),  (id+\sqrt{2\cwn}(W_t-W_{t_0}))\sharp m).
$$
System \eqref{eq.MFGCNtildedelta} is stated in $(t_0, \infty)\times \T^d$ and the notion of solution is the same as the notion of solution for the shifted  MFG system \eqref{eq.MFGCNtilde}. As explained in \cite[Section 5.5]{CDLL}, the relationship between this solution and the solution of \eqref{eq.MFGCNdisc} is given by:
\be\label{defshifdelta}
\begin{array}{c}
\uDW(x)= \uD(x+\sqrt{2\cwn}(W_t-W_{t_0})), \; \mDW = (id-\sqrt{2\cwn}(W_t-W_{t_0}))\sharp \mD, \\
\vDW(x)= \sqrt{2\cwn}D  \uDW(x)+ \vD(x+\sqrt{2\cwn}(W_t-W_{t_0})),\; d \shif M_t= \shif v_t\cdot dW_t. 
\end{array}
\ee

Let us start with  basic estimates on the solution to \eqref{eq.MFGCNdisc}. 

\begin{lem} \label{lem.boundsDisc} There exists a constant $C>0$, depending on $f$ only, such that if $(u^\delta, m^\delta, v^\delta)$ is the solution to \eqref{eq.MFGCNdisc}, then, a.s.,  
\be\label{esti1111}
\|\delta u^\delta_t\|_\infty+ \|D  u^\delta_t\|_\infty \leq C, \qquad  \forall t\geq t_0, 
\ee
and 
\be\label{esti1112}
C^{-1} \leq m^\delta_t \leq C,\qquad \forall  t\geq t_0+1. 
\ee
\end{lem}
\begin{proof} By comparison on the HJ equation (using $\pm \delta^{-1} \|f\|_\infty$ as super and sub-solutions respectively), one has $\| u^\delta\|_\infty \leq \|f\|_\infty/\delta$. Thus the first inequality in \eqref{esti1111} holds. For the proof of the other estimates, we can use the classical approximation of the solution  $(u^{\delta}, m^{\delta},v^{\delta})$ by the solution  $(u^{\delta,T}, m^{\delta,T},v^{\delta,T})$ of \eqref{eq.MFGCNdisc} defined on a time interval $[t_0,T]$ with $T>t_0$ large,  with initial condition $m^{\delta,T}_{t_0}=\hat m_0$ and terminal condition $u^{\delta,T}_T=0$. Then a small variant of  Lemma \ref{lem.regularity} shows that the estimates \eqref{lipestiCT} and \eqref{lipestiCT2} hold, independently of the horizon $T$ and of the discount factor $\delta$. We can then let $T\to \infty$ to obtain the lemma. 
\end{proof}

\begin{lem}\label{lem.uniqueDISC} There exists $\omega,\delta_0>0$ and, for any $\beta\in(0,1)$, $C_\beta >0$, depending on $f$ only, such that, if $\delta \in  (0, \delta_0)$ and if $(u^\delta_1, m^\delta_1, v^\delta_1)$ and $(u^\delta_2, m^\delta_2, v^\delta_2)$ solve \eqref{eq.MFGCNdisc}-(i)-(ii) on an interval $(t_0,\infty)$ and satisfy the estimates \eqref{esti1111} and \eqref{esti1112} on that interval, then 
$$
 \E\left[\normH{m^\delta_{1,t}- m^\delta_{2,t}}+ \|D(u^\delta_{1,t}-u^\delta_{2,t})\|_{L^2(\T^d)}\right] \leq C_\beta e^{-\omega (t-t_0)} \qquad \forall t\in [t_0+1, \infty).
$$
\end{lem}

\begin{proof} Throughout the proof, we work with the shifted variables $(\shif u^\delta_i, \shif m^\delta_i, \shif M^\delta_i)$ associated to  $(u^\delta_i, m^\delta_i, v^\delta_i)$ as in \eqref{defshifdelta}.  We set $\tilde{\shif u}^\delta_{i,t}(x)= \shif u^\delta_{i,t}(x)- \inte \shif u^\delta_{i,t}$ and
$$
\alpha(t) = \|\tilde{\shif u}^\delta_{1,t}-\tilde{\shif u}^\delta_{2,t}\|^2_{L^2(\T^d\times \Omega)}, \qquad
\beta(t) = \| D(\shif u^\delta_{1,t}-\shif u^\delta_{2,t})\|^2_{L^2(\T^d\times \Omega)}
$$
and
$$
\gamma(t) = \|\shif m^\delta_{1,t}-\shif m^\delta_{2,t}\|^2_{L^2(\T^d\times \Omega)}.
$$
By the standard Lasry-Lions monotonicity argument, we have, for any $t_0+1\leq t_1\leq t_2\leq T$,  
\begin{align}\label{LLmonotonicityDISC}
& \E\left[ e^{-\delta t_2}\int_{\T^d}(\shif u^\delta_{1,t_2}-\shif u^\delta_{2,t_2})(\shif m^\delta_{1,t_2}-\shif m^\delta_{2,t_2})-
e^{-\delta t_1}\int_{\T^d}(\shif u^\delta_{1,t_1}-\shif u^\delta_{2,t_1})(\shif m^\delta_{1,t_1}-\shif m^\delta_{2,t_1}) \right]  \notag\\
& \qquad \leq - \frac12 \E\left[ \int_{t_1}^{t_2} \int_{\T^d} e^{-\delta s} |D\shif u^\delta_1-D\shif u^\delta_2|^2 (\shif m^\delta_1+\shif m^\delta_2)\right] . 
\end{align} 
As in the proof of Lemma \ref{lem.keyuniqueCT}, we have
 \begin{align*}
& \E\left[\left| \int_{\T^d}(\shif u^\delta_{1,t}-\shif u^\delta_{2,t})(\shif m^\delta_{1,t}-\shif m^\delta_{2,t}) \right|\right]  \leq C_0(\alpha(t)\gamma(t))^{1/2}.
 \end{align*}
 Hence, letting $t_2\to\infty$ in \eqref{LLmonotonicityDISC} and using  the bound below for the $\shif m^\delta_ i$  given in Lemma \ref{lem.boundsDisc} as well as the fact that $D\shif u^\delta_1$ and $D\shif u^\delta_2$  are bounded we infer that, 
\be\label{jhbelvbhbfDISC}
\int_{t_1}^{\infty} e^{-\delta t} \beta(t)dt \leq C_0e^{-\delta t_1} (\alpha(t_1)\gamma(t_1))^{1/2}. 
\ee
Using Lemma \ref{lemCLLP2.1} applied to $\shif m^\delta_1-\shif m^\delta_2$ which satisfies an equation of the form \eqref{eq.kolmo} with $V= D\shif u^\delta_1$ and   $B= \shif m^\delta_2(D\shif u^\delta_1-D\shif u^\delta_2)$, we obtain, on any time interval $[t_1,t_2]$,  
$$
\|\shif m^\delta_{1,t_2}-\shif m^\delta_{2,t_2}\|_{L^2(\T^d)}\leq Ce^{-\omega (t_2-t_1)} \|\shif m^\delta_{1,t_1}-\shif m^\delta_{2,t_1}\|_{L^2(\T^d)} +C\left[\int_{t_1}^{t_2} \|\shif m^\delta_{2,s} D(\shif u^\delta_{1,s}-\shif u^\delta_{2,s})\|_{L^2(\T^d)}^2ds\right]^{1/2}.
$$
Using the bound above for $\shif m^\delta_2$ in Lemma \ref{lem.boundsDisc}, taking the square and the expectation, this implies, changing $\omega$ if necessary,
$$
\gamma(t_2) \leq C_0e^{-\omega (t_2-t_1)} \gamma(t_1) + C_0 \int_{t_1}^{t_2} \beta(t)dt .
$$

We now look at the equation satisfied by  $\shif u^\delta_1-\shif u^\delta_2$, which is of the form \eqref{eq.backward_stochastic} with $V= D\shif u^\delta_1$ and 
\be\label{lkazjenzrse}
A=  \frac12|D(\shif u^\delta_1-\shif u^\delta_2)|^2 -f(x,\shif m^\delta_1)+f(x,\shif m^\delta_2).
\ee
Using the $L^\infty$ bound for $D\shif u^\delta_1$ and $D\shif u^\delta_2$ in Lemma \ref{lem.boundsDisc} and the regularity of $f$ we infer by \eqref{zerldjkjnpimj} in Lemma \ref{lemCLLP2.2} that 
\begin{align*}
& \|\tilde{\shif u}^\delta_{1,t_1}-\tilde{\shif u}^\delta_{2,t_1} \|_{L^2(\T^d\times \Omega)} \\
& \qquad \leq C e^{-\omega(t_2-t_1)} \|\tilde{\shif u}^\delta_{1,t_2}-\tilde{\shif u}^\delta_{2,t_2}\|_{L^2(\T^d\times \Omega)} \\
& \qquad + C \int_{t_1}^{t_2} e^{-\omega (s-t_1)} (\|D(\shif u^\delta_{1,s}-\shif u^\delta_{2,s})\|_{L^2(\T^d\times \Omega)}+ \|\shif m^\delta_{1,s}-\shif m^\delta_{2,s}\|_{L^2(\T^d\times \Omega)}).
\end{align*}
Taking the square and using Cauchy-Schwarz inequality we infer that  (changing $\omega$ when necessary)
$$
\alpha(t_1)\leq C_0 e^{-\omega(t_2-t_1)} \alpha(t_2) +C_0 \int_{t_1}^{t_2} e^{-\omega (t-t_1)}\beta(t) dt + C_0 \sup_{t_1\leq t\leq t_2} e^{-\omega(t-t_1)}\gamma(t), 
$$
for some constant $C_0>0$ independent of the initial conditions. We can then let $t_2\to\infty$ to infer that
$$
\alpha(t_1)\leq C_0 \int_{t_1}^{\infty} e^{-\omega (t-t_1)}\beta(t) dt + C_0 \sup_{t\geq t_1} e^{-\omega(t-t_1)}\gamma(t). 
$$
Therefore $\alpha,\beta, \gamma$ satisfy \eqref{hypexpo2}, \eqref{hypexpo4}, \eqref{hypexpo1DISC} and \eqref{hypexpo3DISC} in the appendix. We conclude by using Lemma \ref{lem.expodecayDISC} in the appendix that there exists  constants $C,\omega>0$ such that, for any $t\geq t_0+1$,  
\begin{align*}
& \|\tilde{\shif u}^\delta_{1,t}-\tilde{\shif u}^\delta_{2,t}\|^2_{L^2(\T^d\times \Omega)}+\| \shif m^\delta_{1,t}-\shif m^\delta_{2,t}\|^2_{L^2(\T^d\times \Omega)}\\
&\qquad \leq Ce^{-\omega' t} \|\shif  m^\delta_{1,t_0+1}-\shif m^\delta_{2,t_0+1}\|^2_{L^2(\T^d\times \Omega)}.
\end{align*}
We can then complete the proof of the estimate of $D(u^\delta_{1,t}-u^\delta_{2,t})$ and of $ m^\delta_{1,t_0}-m^\delta_{2,t_0}$ in  $C^\alpha$ regularity as in the proof of Lemma \ref{lem.keyuniqueCT}. 
\end{proof}

\begin{proof}[Proof of Lemma \ref{lem.WmuDISC}]  The proof is the same as for Lemma \ref{lem.Wmu}.
\end{proof}

\begin{proof}[Proof of Theorem \ref{thm.Cudelta}] The proof of \eqref{ineqcontracdelta} is the same as for  \eqref{ineq.expo} in Theorem \ref{thm.main}. \\

Let us now show \eqref{ineqcontracdelta4}. Fix $T>0$ and let $(w^\delta, m^\delta, z^\delta)$ be the solution defined on $(-T,\infty)$ given by Lemma \ref{lem.WmuDISC} with initial condition $m^\delta=\bar \mu^\delta$ and let  $(w^T, m^T, z^T)$ be the weak solution given by lemma \ref{lem.Wmu}, defined on the time interval $(-T,T)$, with initial condition $m(-T)= \bar \mu$ and terminal condition $w$ where $w$ is  such that $Dw= \bar{\mathcal V}_T$ and $w(0)=0$. Recall that $m^\delta= \bar \mu^\delta$, $Dw^\delta = \bar{\mathcal V}^\delta$, $m^T=\bar \mu$ and $Dw^T =\bar{\mathcal V}$. By duality we have, for any $t_1,t_2$ with $-T\leq t_1\leq t_2\leq T$,   
\begin{align*}
&\E\left[ \inte (w^\delta_{t_2}-w_{t_2})(\bar \mu^\delta_{t_2}-\bar \mu_{t_2})\right]-\E\left[ \inte (w^\delta_{t_1}-w_{t_1})(\bar \mu^\delta_{t_1}-\bar \mu_{t_1})\right] +c\int_{t_1}^{t_2}\E\left[ \inte |\bar{\mathcal V}^\delta_t-\bar{\mathcal V}_t|^2\right]dt \\
& \leq  \int_{t_1}^{t_2} \E\left[ \inte \delta w^\delta_t(\bar \mu^\delta_t-\bar \mu_t)\right] dt .
\end{align*}
By the stationarity of $(\bar{\mathcal V}, \bar \mu)$ and $(\bar{\mathcal V}^\delta, \bar \mu^\delta)$ and the usual bounds we obtain:  
\begin{align*}
c \E\left[ \inte |\bar{\mathcal V}^\delta_0-\bar{\mathcal V}_0|^2dx\right] & = c \frac{1}{t_2-t_1}\E\left[ \int_{t_1}^{t_2} \inte |\bar{\mathcal V}^\delta_t-\bar{\mathcal V}_t|^2 dxdt\right]\\
& \leq \frac{C}{t_2-t_1}  +  \frac{1}{t_2-t_1} \int_{t_1}^{t_2} \E\left[ \inte \delta (w^\delta_t- \inte w^\delta_t)(\bar \mu^\delta_t-\bar \mu_t)dx\right] dt \\
&  \leq \frac{C}{t_2-t_1}+C\delta.
\end{align*}
We set $t_1=-T$, $t_2=T$ and let $T\to \infty$ to infer that 
\be\label{aulhiezjnrsdfg}
\E\left[ \inte |\bar{\mathcal V}^\delta_0-\bar{\mathcal V}_0|^2dx\right]\leq C\delta. 
\ee
Writing the equation for $\bar \mu^\delta-\bar \mu$ in the shifted variables and using Lemma \ref{lemCLLP2.1}, we find that there exists $\omega, C>0$ such that, for any $t\in [-T,T]$, 
\begin{align*}
& \E\left[ \inte |\bar \mu^\delta_0-\bar \mu_0|^2 \right]= \E\left[ \inte |\bar \mu^\delta_t-\bar \mu_t|^2 \right] \\
&\qquad   \leq Ce^{-\omega (t+T)} \E\left[ \inte |\bar \mu^\delta_{-T}-\bar \mu_{-T}|^2\right] +C\E\left[\int_{-T}^t \inte |\bar{\mathcal V}^\delta_s-\bar{\mathcal V}_s|^2 ds\right]\\
& \qquad = Ce^{-\omega (t+T)} \E\left[ \inte |\bar \mu^\delta_{0}-\bar \mu_{0}|^2\right] +C(t+T) \E\left[\inte |\bar{\mathcal V}^\delta_0-\bar{\mathcal V}_0|^2 \right].
\end{align*}
We choose $t=0$ and fix $T>0$ large enough so that  $Ce^{-\omega T}<1$ to conclude from \eqref{aulhiezjnrsdfg} that 
\be\label{aulhiezjnrsdfg2}
 \E\left[ \inte |\bar \mu^\delta_0-\bar \mu_0|^2\right] \leq C\delta. 
 \ee
 
Let now $(u^\delta, m^\delta, v^\delta)$ be the solution to \eqref{eq.MFGCNdisc} and $(w^\delta, n^\delta, z^\delta)$ be the solution to \eqref{eq.MFGCNdisc}-(i)-(ii), with initial condition $n^\delta_{t_0}=\bar \mu^\delta_{t_0}$, as defined on $(t_0,\infty)$ given by Lemma \eqref{lem.WmuDISC}. Recall that $Dw^\delta= \bar{\mathcal V}^\delta$ and $n^\delta= \bar \mu^\delta$. By Lemma \ref{lem.uniqueDISC} and comparison between the two solutions $(u^\delta, v^\delta)$ and $(w^\delta, z^\delta)$ of the same backward HJ equation with different source terms, we have, for any $t_0\leq t \leq T$ and any $x\in \T^d$, 
\be\label{kjehsrdbfkgnc}
\E\left[ \|\delta w^\delta_{t}- \delta  u^{\delta}_{t}\|_\infty\right]  \leq  \E\left[ \delta \int_{t}^Te^{-\delta (s-t)} \| f(\cdot, m^\delta_s)-f(\cdot, \bar \mu_s)\|_\infty ds+ \delta e^{-\delta (T-t)} \|w^\delta_{T}-  u^{\delta}_{T}\|_\infty \right].  
\ee 
Note that, by the regularity of $f$ and \eqref{ineqcontracdelta}, 
$$
\E\left[ \| f(\cdot, m^\delta_s)-f(\cdot, \bar \mu_s)\|_\infty \right] \leq C \|m^\delta_s-\bar \mu_s\|_{L^2_{\omega,x}} \leq C e^{-\omega (s-t_0)}. 
$$ 
So we can let $T\to \infty$ in \eqref{kjehsrdbfkgnc} to obtain, as  $w^\delta$ and $u^\delta$ are bounded, 
\be\label{kjehsrdbfkgnc4}
\E\left[ \|\delta w^\delta_{t}- \delta  u^{\delta}_{t}\|_\infty\right]   \leq  \E\left[  C \delta \int_{t}^\infty e^{-\omega (s-t_0)} e^{-\delta (s-t)}ds \right]\leq C\delta.  
\ee
On the other hand, by the stationarity of $\bar{\mathcal V}^\delta$ and $\bar \mu^\delta$, 
\begin{align*}
\E\left[\inte \delta w^{\delta}_{t}(x)dx\right] & =- \delta \E\left[ \int_t^\infty e^{-\delta (s-t)} \inte(\frac12 |\bar{\mathcal V}^\delta_s(x)|^2 - f(x, \bar \mu^\delta_s(t)))dxds \right] \\
& = - \E\left[ \inte (\frac12 |\bar{\mathcal V}^\delta_0(x)|^2 - f(x, \bar \mu^\delta_0(t)))dx \right] =: \bar \lambda^\delta. 
\end{align*}
By  \eqref{aulhiezjnrsdfg}, \eqref{aulhiezjnrsdfg2} and the definition of $\bar \lambda$ in Theorem \ref{thm.main},  we have 
$$
|\bar \lambda -\bar \lambda^\delta|\leq C\delta. 
$$
Combining this inequality with \eqref{kjehsrdbfkgnc4} and the Lipschitz bound in Lemma \ref{lem.boundsDisc}, we obtain 
\begin{align*}
\E\left[\|\delta u_t^\delta-\bar \lambda\|_\infty\right] & \leq  \delta \ {\rm ess-sup} \|Du^\delta\|_\infty + \E\left[ |\delta \int_{\T^d}u^\delta_{t}-\bar \lambda|\right] \\
& \leq C\delta+ \E\left[ \|\delta u^\delta_{t}- \delta  w^{\delta}_{t}\|_\infty\right] + \E\left[ |\delta \int_{\T^d}w^\delta_{t}-\bar \lambda|\right] 
 \leq 
 C\delta,
 \end{align*}
 from which we infer that \eqref{ineqcontracdelta3} holds. 
 \end{proof}

%%%%%%%%%%%%%%%%%%%%%%%%%%%
%%%%%%%%%%%%%%%%%%%%%%%%%%%
%%%%%%%%%%%%%%%%%%%%%%%%%%%
%%%%%%%%%%%%%%%%%%%%%%%%%%
\section{Estimates for linearized MFG systems}\label{sec.linear}

In this section we introduce the linearized systems associated with the MFG systems and study the regularity of the solution. These technical estimates will be instrumental for the construction of the ergodic master equation (Section \ref{sec.ergo}) and for the long-time behavior of the solution to the time dependent master equation (Section \ref{sec.longtimemaster}). \\

Let $\delta \geq 0$ and let us fix $(\shif u, \shif m, \shif M)$ a solution to the shifted MFG system \eqref{eq.MFGCNtilde} or to the shifted discounted MFG system \eqref{eq.MFGCNtildedelta} with initial condition $\hat m_0\in \mathcal P(\T^d)$ at time $t_0$. We consider the  solution $(\shif z, \boldsymbol{\varrho}, \shif N )$ of the following linearized system:
\be\label{eq.linearized_MFG_shifted}
\left\{
\begin{array}{ll}
    \ds d \shif z_t = \Bigl(\delta \shif z_t -\Delta \shif z_t + D  \shif u_t \cdot D  \shif z_t - \left<\frac{\delta \shif f_t}{\delta m}(\cdot, \shif m_t),\boldsymbol{\varrho}_t \right> \Bigr)dt + d \shif N_t \\[2ex]
    d\boldsymbol{\varrho}_t = \left(\Delta \boldsymbol{\varrho}_t +{\rm div}(\boldsymbol{\varrho}_t D  \shif u_t + \shif m_t D  \shif z_t)\right)dt \\[2ex]
    \boldsymbol \varrho_{t_0} = \rho_0, \qquad  {\rm in}\;  \T^d.
\end{array}
\right.
\ee
This system is considered in two settings: when $\delta=0$ and $(\shif u, \shif m, \shif M)$ is the solution to the shifted finite horizon MFG system \eqref{eq.MFGCNtilde} on the finite time interval $[t_0,T]$, that is $(\shif u^T, \shif m^T, \shif M^T)$, then we write $(\zTW,\rTW,\NTW)$ as the solution of \eqref{eq.linearized_MFG_shifted}, which is stated on $\Omega \times [t_0,T] \times \T^d$ with terminal condition,
$$
\shif z^T_T(x)= \left<\frac{\delta \shif g_T}{\delta m}(\cdot, \shif m_T^T),\boldsymbol{\varrho}^T_T \right>\qquad {\rm in}\;  \T^d.
$$
By a solution, we mean an $(\mathcal F_{t_0,t})_{t\in [t_0,T]}$-adapted process $(\zTW,\rTW,\NTW)$ with paths in \\
\(
C^0([t_0,T], C^{3}(\T^d) \times (C^{2}(\T^d))' \times [C^{1}(\T^d)]^d),
\)
such that
$$
\esssup \sup_{t \in [t_0,T]} \left( \|\zTW\|_{C^{3+\beta}} + \|\rTW\|_{(C^{2+\beta}(\T^d))'} + \|\NTW\|_{C^{1+\beta}} \right) < \infty
$$
(for any $0<\beta<\alpha$). 
The first equation in \eqref{eq.linearized_MFG_shifted} holds in the classical sense with probability one, while the second one is understood in the sense of distributions, again with probability one. The existence and the uniqueness of the solution is given in \cite[Theorem 4.4.2]{CDLL}. \\

When $\delta>0$ and $(\shif u, \shif m, \shif M)$ is the solution $(\shif u^\delta, \shif m^\delta, \shif M^\delta)$ to the shifted discounted MFG system \eqref{eq.MFGCNtildedelta}, then $(\zDW,\rDW,\NDW)$ stands for the solution of \eqref{eq.linearized_MFG_shifted} on $\Omega \times [t_0,\infty)\times \T^d$, which satisfies the bound 
\be\label{cond.unicite.zDW}
\sup_{t \geq 0} e^{-\delta t} \E\left[\|D{\shif z}^{\delta}_t\|_{L^2_x}^2\right]+\sup_{t\geq 1}e^{-\delta t} \E\left[  \normUU{\rDW}{L^2_x}^2\right]  <\infty. 
\ee
By a solution, we mean an $(\mathcal F_{t_0,t})_{t\in [t_0,\infty)}$-adapted process $(\zDW,\rDW,\NDW)$ with paths in \\ 
\(
C^0([t_0,T], C^{3}(\T^d) \times (C^{2}(\T^d))' \times [C^{1}(\T^d)]^d)
\) (for any $T>0$) 
such that, for any $T>0$, 
$$
\E\left[ \sup_{t \in [t_0,T]} \left( \|\zDW\|_{C^{3+\beta}} + \|\rDW\|_{(C^{2+\beta}(\T^d))'} + \|\NDW\|_{C^{1+\beta}} \right) \right] <\infty
$$
(for any $0<\beta<\alpha$). The first equation in \eqref{eq.linearized_MFG_shifted} holds in the classical sense with probability one, while the second one is understood in the sense of distributions, again with probability one.
We prove the existence of  such a solution and its uniqueness in Proposition \ref{prop.sol_shif.lin} below. We now state the central estimates of this section, in the two frameworks. 

\begin{thm}\label{thm.good_bound_for_linearized_system_finite}
There exists $C > 0$ such that, if  $(\shif u^T, \shif m^T, \shif M^T)$ is the solution of the shifted MFG system \eqref{eq.MFGCNtilde}, and if $(\zTW, \rTW, \NTW )$ is a solution of the shifted linearized system \eqref{eq.linearized_MFG_shifted} with $\delta=0$ and $\rho_{0} \in \left(C^{2+\beta}\right)'$ such that $\left< \rho_{0}, \mathbf{1}_{\T^d} \right> =0$, then we have
\be
\E\left[\normU{\shif z^T_{t_0} }{C^{3+\beta}}\right] \leq C \normDDHd{\rho_{0}}.
\ee
\end{thm}

\begin{thm}\label{thm.good_bound_for_linearized_system_discounted}
There exists $\delta_0,~C > 0$ such that, for any $\delta \in (0, \delta_0)$, if $(\shif u^\delta, \shif m^\delta, \shif M^\delta)$ is the solution of the shifted discounted MFG system \eqref{eq.MFGCNtildedelta}, and if $(\zDW, \rDW, \NDW )$ is a solution of the shifted linearized discounted system \eqref{eq.linearized_MFG_shifted} with $\rho_{0} \in \left(C^{2+\beta}\right)'$ such that $\left< \rho_{0}, \mathbf{1}_{\T^d} \right> =0$, then we have
\be\label{kjezhrbdfnl}
\E\left[\normU{\shif z^\delta_{t_0} }{C^{3+\beta}}\right]\leq C \normDDHd{\rho_{0}}.
\ee
\end{thm}

We only prove  Theorem \ref{thm.good_bound_for_linearized_system_discounted}, as the proof of Theorem \ref{thm.good_bound_for_linearized_system_finite} is analogous and simpler. In addition, we have to explain in the infinite horizon framework the existence and the uniqueness of a solution to the linearized system. \\

We fix $t_0\in \R$ and a solution $(\shif u^\delta, \shif m^\delta, \shif M^\delta)$ to the shifted discounted MFG system \eqref{eq.MFGCNtildedelta}. In order to build a solution for the shifted linearized system \eqref{eq.linearized_MFG_shifted}, we consider the following classical approximation: given $T>t_0$, we consider the solution $(\zDWT, \rDWT, \NDWT )$ to the finite horizon system:  
\be\label{eq.linearized_MFG_shiftedT}
\left\{
\begin{array}{ll}
 (i) &   \ds d \shif z^{\delta, T}_t = \Bigl(\delta \shif z^{\delta, T}_t -\Delta \shif z^{\delta, T}_t + D  \shif u^\delta_t \cdot D  \shif z^{\delta, T}_t - \left<\frac{\delta \shif f_t}{\delta m}(\cdot, \shif m^\delta_t),\boldsymbol{\varrho}^{\delta, T}_t \right> \Bigr)dt + d \shif N^{\delta, T}_t \\[2ex]
  (ii) &   d\boldsymbol{\varrho}^{\delta, T}_t = \left(\Delta \boldsymbol{\varrho}^{\delta, T}_t +{\rm div}(\boldsymbol{\varrho}^{\delta, T}_t D  \shif u^\delta_t + \shif m^\delta_t D  \shif z^{\delta, T}_t)\right)dt \\[2ex]
  (iii) &  \boldsymbol \varrho^{\delta,T}_{t_0} = \rho_0, \qquad   \shif z^{\delta, T}_T=0 \qquad {\rm in}\;  \T^d.
\end{array}
\right.
\ee
This system is stated on $\Omega\times [t_0,T]\times \T^d$. 
The existence of a solution  to \eqref{eq.linearized_MFG_shiftedT} and its uniqueness follow exactly the same argument as in  \cite[Theorem 4.4.2]{CDLL}. Our aim is to show that the limit as $T\to\infty$ of $(\zDWT, \rDWT, \NDWT )$ is a solution to \eqref{eq.linearized_MFG_shifted}.  For this we need  two preliminary estimates.  The first one is rather rough in its dependence on $\delta$, but holds for general terminal conditions. The second one gives a decay in  time of the solution with terminal condition \eqref{eq.linearized_MFG_shiftedT}-(iii). 

\begin{lem} \label{lem.esti.zDWT1} There exists $\delta_0,\tau>0$ and, for any $\beta\in (0,1)$, a constant $C_\beta$ such that, for any $\delta \in (0,\delta_0)$, any $t_0<T$ and any classical solution  $(\zDWT, \rDWT, \NDWT )$ to \eqref{eq.linearized_MFG_shifted}-(i)-(ii) on the time interval $[t_0,T]$ (with arbitrary initial and  terminal conditions), we have
\begin{align}\label{kzjhejhezrbnldkj}
&\sup_{t_0+\tau \leq t \leq T} \E\left[ e^{-\delta (t-t_0)} \sup_{t - \tau \leq s \leq t} \| \tilde{\shif z}^{\delta,T}_s\|_{C^{3+\beta}}^2\right]+ \E\left[ \sup_{t_0 \leq t \leq T} e^{- \delta  (t-t_0)} \normUU{\rDWT}{(C^{2+\beta})'}^2\right]\notag \\
&\qquad + \E\left[\sup_{t_0+2 \leq t \leq T} e^{-\delta  (t-t_0)}  \normtwox{\boldsymbol{\varrho}_t^{\delta,T}}^2\right] \  \leq \ C_\beta  \left(  \E\left[ \normU{\boldsymbol{\varrho}_{0}^{\delta,T}}{(C^{2+\beta})'}^2\right] +\E\left[\|\tilde{\shif z}^{\delta,T}_{T}\|_{C^{3+\beta}}^2\right] \right) \ ,
\end{align}
where we have set, as usual, $\tilde{\shif z}^{\delta,T}_t:= \zDWT-\int_{\T^d} \zDWT$. 
\end{lem}

\begin{proof}[Proof of Lemma \ref{lem.esti.zDWT1}] As the proof is long, we divided it into 4 steps. In a first step we estimate $\rDWT$ in function of a quantity $\kappa$ depending on $\boldsymbol{\varrho}_{0}^{\delta,T}$, $D\shif z^{\delta,T}_T$ and the $C^{3+\beta}-$norm of $\tilde{\shif z}^{\delta,T}_0$. The main part of the proof is devoted to get rid of the dependence of $\kappa$ with respect to the last quantity, which is not given. For this we obtain  in the second step a bound of $D\shif z^{\delta,T}_t$ in $L^2_tL^\infty_x$ still in function of $\kappa$. In the next step we improve this pointwise in time estimate into a locally uniform one. Finally we show that we can bootstrap this last estimate on $\tilde{\shif z}^{\delta,T}_t$ into a $C^{3+\beta}$ one in space, locally uniform in time: this eventually allows to close all the previous estimates. To simplify the notation, we assume that $t_0=0$, as the general case involves exactly the same argument.\\ 

\textbf{Step 1:}
We first estimate $\rDWT$. By Lasry-Lions duality, the following stochastic process is a martingale,
\begin{align*}
\left( e^{-\delta t} \int_{\T^d} \zDWT \rDWT - \int_{\T^d} \shif z^{\delta,T}_0 \boldsymbol{\varrho}_{0}^{\delta,T} \right) + \int_0^t e^{-\delta s} \left( \int_{\T^d} \boldsymbol{\varrho}_s^{\delta,T} \left< \frac{\delta \shif f_s}{\delta m}, \boldsymbol{\varrho}_s^{\delta,T} \right> +  \int_{\T^d} \shif m^\delta_s |D  \shif z^{\delta,T}|^2 \right) ds.
\end{align*}
Hence by taking the expectation and using the monotonicity of $\shif f$, we have, for any $0\leq t_1\leq t_2\leq T$, 
\be\label{lakjzerndlf}
\EP{\int_{t_1}^{t_2} e^{-\delta t} \int_{\T^d} \mDW |D  \zDWT|^2 dx dt } \leq e^{-\delta t_1}\EP{\int_{\T^d} \shif z^{\delta,T}_{t_1} \boldsymbol{\varrho}^{\delta,T}_{t_1} } - e^{-\delta t_2}\EP{\int_{\T^d} \shif z^{\delta,T}_{t_2} \boldsymbol{\varrho}^{\delta,T}_{t_2} }.
\ee
%As $\shif z^{\delta,T}_T=0$,  we find with  $t_2=T$ that $\EP{\int_{\T^d} \shif z^{\delta,T}_{t_1} \boldsymbol{\varrho}^{\delta,T}_{t_1} }\geq 0$ for any $t_1\in [0,T]$. On the other hand, taking $t_1= 0$ and $t_2$ arbitrary, we infer that 
Taking $t_1=0$ and $t_2=T$ and using the fact that $\langle \boldsymbol{\varrho}^{\delta,T}_{t} , 1\rangle =0$ for any $t$, we infer that 
\be\label{lakjzerndlf2}
\EP{\int_{0}^{t_2} e^{-\delta t} \int_{\T^d} \mDW |D  \zDWT|^2 dx dt } \leq \EP{\int_{\T^d} \tilde{\shif z}^{\delta,T}_{0} \boldsymbol{\varrho}_{0}^{\delta,T} } 
- e^{-\delta T}\EP{\int_{\T^d} \tilde{\shif z}^{\delta,T}_{T} \boldsymbol{\varrho}^{\delta,T}_{T} } 
%= \EP{\int_{\T^d} \tilde{\shif z}^{\delta,T}_{0} \boldsymbol{\varrho}_{0} } .
\ee

Note that the  equation for $\boldsymbol{\rho}^{\delta,T}$ in \eqref{eq.linearized_MFG_shifted} can be solved $\omega$ by $\omega$. Fix $t>0$ and $\xi\in L^2_{\omega, x}$ and consider the following dual equation, that we also solve $\omega$ by $\omega$ (it is not adapted): 
\be\label{eq:bakward_many_omegas}
\left\{
\begin{array}{ll}
    - \partial_t \psi_s  -\Delta \psi_s + D  \shif u^\delta_s \cdot D  \psi_s = 0 & \text{in~} (0,t)\times \T^d \times \Omega, \\
    \psi_t(x,\omega) = \xi(x,\omega) & \text{in~} \T^d \times \Omega.
\end{array}
\right.
\ee
Thanks to the uniform bound of $D\uDW$, Lemma 7.4 of  \cite{CLLP2} states that there exists $\lambda>0$ (depending only on $\|D  \shif u^\delta\|_\infty$, and thus independent of $\delta$ in view of Lemma \ref{lem.boundsDisc}) such that
\be\label{lejkhrzndfg}
\normtwox{\psi_s -\hat \psi_s} \leq C e^{-\lambda (t-s)} \normtwox{\xi} \qquad \forall s\in [0,t],
\ee
where $\hat \psi_s=\int_{\T^d} \psi_s$. 
The equation satisfied by $\tilde \psi_s = \psi_s - \hat \psi_s$ is
\[
\left\{
\begin{array}{ll}
    - \partial_t \tilde \psi_s -\Delta \tilde \psi_s + \left( D  \shif u^\delta_s \cdot D  \tilde \psi_s - \int_{\T^d} (D  \shif u^\delta_s \cdot D  \tilde \psi_s) \right) = 0 & \text{in~} (0,t)\times \T^d \times \Omega, \\
    \tilde \psi_t(x,\omega) = \tilde \xi(x,\omega) & \text{in~} \T^d \times \Omega.
\end{array}
\right.
\]
Hence, multiplying the above equation by $ \tilde \psi_s$,  integrating on  $[0,t]\times \T^d $ and using \eqref{lejkhrzndfg}, we obtain the inequality,
\be\label{lkjazensrmmjj}
\int_{0}^t  \int_{\T^d} |D  \psi_s|^2 \leq  \normtwox{\xi}^2 + C \int_{0}^t  \int_{\T^d} |\psi_s-\hat \psi_s|^2 \leq C \normtwox{\xi}^2.
\ee
Next we note that, if $t\geq 2$, then 
\be\label{ljhqdbslfdjkb}
    \normDDH{\psi_s - \hat \psi_s} \leq C \normtwox{\xi} 
%    \leq C \normDDH{\xi} & 
\qquad     \forall s \in [0, t-1],
\ee
where $C$ is independent of $t\geq 2$ and $\delta$. 
Indeed,  by \eqref{lkjazensrmmjj} and Poincar\'e inequality, $\tilde \psi$ is bounded in $L^2((0,t)\times \T^d)$ by $C \normtwox{\xi}$. Then \cite[Theorem III.8.1]{LSU} implies that $\tilde \psi$ is bounded in $L^\infty([0,t-1/4]\times \T^d)$ by $C \normtwox{\xi}$.  Therefore, by \cite[Theorem III.11.1]{LSU},  $\sup_{s\in [0,t-1/2]} \|D\tilde \psi (s)\|_{C^\beta}$ is bounded by $C_\beta \normtwox{\xi}$ for any $\beta\in (0,1)$. This implies that the map $(s,x)\to D  \shif u^\delta_s \cdot D  \tilde \psi_s - \int_{\T^d} (D  \shif u^\delta_s \cdot D  \tilde \psi_s) $ is continuous and uniformly bounded in $C^{\beta}$ by $C_{\beta} \normtwox{\xi}$. We then infer by \cite[Theorem 5.1.4]{Lun} that, for  $s \in [0, t-1]$ and $\beta'\in (0, \beta)$, 
\be\label{ljhqdbslfdjkb4}
%\left\{
%\begin{array}{ll}
%    \normDDH{\psi_s - \hat \psi_s} \leq C \normDDH{\xi} & \forall s \in [t-1, t],  \\[1.5ex]
    \|\tilde \psi_s\|_{C^{2+\beta'}} \leq C \|\tilde \psi_{s+1}\|_{C^0} +C \Bigl\|D  \shif u^\delta_s \cdot D  \tilde \psi_s - \int_{\T^d} (D  \shif u^\delta_s \cdot D  \tilde \psi_s) \Bigr\|_{C^{0,\beta'}([s,s+1]\times \T^d)} \leq C \normtwox{\xi} . 
%    \leq C \normDDH{\xi} & 
%\;     \forall s \in [0, t-1]
%\end{array}
%\right.
\ee
Thus  \eqref{ljhqdbslfdjkb} holds. Note also that, if $t\in [0,2]$, then, by the regularity in space of the coefficient of Equation \eqref{eq:bakward_many_omegas}, we have, for any $\beta \in (0,1)$, 
\be\label{ljhqdbslfdjkb2}
   \|\psi_s\|_{C^{2+\beta}}\leq C_\beta \|\xi\|_{C^{2+\beta}} 
\qquad     \forall s \in [0, t].
\ee
Using the equation satisfied by $\psi$ and by $\rDWT$, we obtain by duality:
\begin{align*}
& \int_{\T^d} \rDWT \xi  = - \int_{0}^t  \int_{\T^d} \shif m_s^\delta D  \shif z_s^{\delta,T} \cdot D  \psi_s + \int_{\T^d} \psi_{0} \boldsymbol{\varrho}_{0}^{\delta,T} \\
& \leq C \left( \int_{0}^t  \int_{\T^d} \shif m_s^\delta | D  \shif z_s^{\delta,T} |^2 \right)^{\frac 12} \left( \int_{0}^t  \int_{\T^d} \shif m_s^\delta | D  \psi_s |^2 \right)^{\frac 12}  + \normDDH{\psi_{0}-\hat \psi_{0}} \normU{\boldsymbol{\varrho}_{0}^{\delta,T}}{(C^{2+\beta})'}.
\end{align*}
Let us first use this inequality to bound $\rDWT$ for $t\geq 1$. Recalling that $ C^{-1} \leq \shif m_s^\delta \leq C$ for any $s \geq 1$ as shown in Lemma 
%\ref{lem.keyuniqueCT} and 
\ref{lem.boundsDisc}, we have thanks to the estimates on $\psi$ in \eqref{lkjazensrmmjj} and in \eqref{ljhqdbslfdjkb} that, for $t\geq 2$, 
\begin{align*}
\int_{0}^t  \int_{\T^d} \shif m_s^\delta | D  \psi_s |^2
& \leq \left( \int_{0}^{ 1}  \int_{\T^d} \shif m_s^\delta  \right) \sup_{ 0 \leq s \leq  1} \normC{ D  \psi_s}^2 + C \int_{ 1}^t \int_{\T^d} |D  \psi_s|^2 \\[0.5em]
& \leq C \sup_{0 \leq s \leq  1} \normD{ \psi_s -\hat \psi_s }^2 + C \normtwox{\xi}^2 \leq C\|\xi\|_{L^2}^2 . 
\end{align*}
Hence, for $t\geq 1$, we obtain that 
\[
 \int_{\T^d} \rDWT \xi \leq C  \|\xi\|_{L^2}\left( \left( \int_{0}^t  \int_{\T^d} \shif m_s^\delta | D  \shif z_s^{\delta,T} |^2 \right)^{\frac 12} + \normU{\boldsymbol{\varrho}_{0}^{\delta,T}}{(C^{2+\beta})'} \right) . 
\]
This implies that 
$$
e^{-\delta t} \| \rDWT\|_{L^2_x}^2 \leq C \left( \int_{0}^t e^{-\delta s} \int_{\T^d} \shif m_s^\delta | D  \shif z_s^{\delta,T} |^2  + \normU{\boldsymbol{\varrho}_{0}^{\delta,T}}{(C^{2+\beta})'}^2 \right) . 
$$
We  take the sup over $t\geq 1$, then the expectation and use \eqref{lakjzerndlf2}  to derive that 
\begin{align*}
& \E\left[ \sup_{t\geq 1} e^{-\delta t} \normUU{\rDWT}{L^2}^2\right]  \leq C\E\left[ \int_{0}^T  \int_{\T^d} e^{-\delta s}\shif m_s^\delta | D  \shif z_s^{\delta,T} |^2 ds \right] + C \E\left[ \normU{\boldsymbol{\varrho}_{0}^{\delta,T}}{(C^{2+\beta})'}^2\right] \notag\\ 
&  \leq C( \EP{\int_{\T^d} \tilde{\shif z}^{\delta,T}_{0} \boldsymbol{\varrho}_{0}^{\delta,T} } 
- e^{-\delta T}\EP{\int_{\T^d} \tilde{\shif z}^{\delta,T}_{T} \boldsymbol{\varrho}^{\delta,T}_{T} } ) +  C \E\left[ \normU{\boldsymbol{\varrho}_{0}^{\delta,T}}{(C^{2+\beta})'}^2\right]\notag \\
&  \leq C \E\left[\|\tilde{\shif z}^{\delta,T}_{0}\|_{C^{2+\beta}}^2\right]^{\frac 12} \E\left[ \normU{\boldsymbol{\varrho}_{0}^{\delta,T}}{(C^{2+\beta})'}\right]^{\frac 12}
+ C e^{-\delta T} \E\left[\|\tilde{\shif z}^{\delta,T}_{T}\|_{L^2}^2\right]^{\frac 12} \E\left[ \normU{\boldsymbol{\varrho}^{\delta,T}_{T}}{L^2}^2\right]^{\frac 12}+  C \E\left[ \normU{\boldsymbol{\varrho}_{0}^{\delta,T}}{(C^{2+\beta})'}^2\right] \ .
\end{align*}
 Applying the above estimate with $t=T$, we infer that
$$
e^{-\delta T} \E\left[ \normU{\boldsymbol{\varrho}^{\delta,T}_{T}}{L^2}\right] \leq C \E\left[\|\tilde{\shif z}^{\delta,T}_{0}\|_{C^{2+\beta}}^2\right]^{\frac 12}  \E\left[ \normU{\boldsymbol{\varrho}_{0}^{\delta,T}}{(C^{2+\beta})'}\right]^{\frac 12}+ C e^{-\delta T} \E\left[\|\tilde{\shif z}^{\delta,T}_{T}\|_{L^2}^2\right]+  C \E\left[ \normU{\boldsymbol{\varrho}_{0}^{\delta,T}}{(C^{2+\beta})'}^2\right]. 
$$
Thus
\begin{align}\label{azelkjrsdnfg3}
\E\left[ \sup_{t\geq 2} e^{-\delta t} \normUU{\rDWT}{L^2}^2\right] \leq  \kappa \ , 
\end{align}
where we have set
\be\label{def.kappa}
\kappa= \E\left[\|\tilde{\shif z}^{\delta,T}_{0}\|_{C^{2+\beta}}^2\right]^{\frac 14}  \E\left[ \normU{\boldsymbol{\varrho}_{0}^{\delta,T}}{(C^{2+\beta})'}^2\right]^{\frac 14}+ C e^{-\delta T/2} \E\left[\|\tilde{\shif z}^{\delta,T}_{T}\|_{C^{3+\beta}}^2\right]^{\frac 12}+\E\left[ \normU{\boldsymbol{\varrho}_{0}^{\delta,T}}{(C^{2+\beta})'}^2\right]^{\frac 12}.
\ee
We now address the case $t\in [0,2]$. By \eqref{ljhqdbslfdjkb2} we have 
\begin{align*}
\int_{0}^t  \int_{\T^d} \shif m_s^\delta | D  \psi_s |^2
& \leq \left( \int_{0}^{ t}  \int_{\T^d} \shif m_s^\delta  \right) \sup_{ 0 \leq s \leq  t} \normC{ D  \psi_s}^2 \leq C \|\xi\|_{C^{2+\beta}}^2.
\end{align*}
Thus 
\[
 \int_{\T^d} \rDWT \xi \leq C  \|\xi\|_{C^{2+\beta}}\left[ \left( \int_{0}^t  \int_{\T^d} \shif m_s^\delta | D  \shif z_s^{\delta,T} |^2 \right)^{\frac 12} + \normU{{\rho}_{0}}{(C^{2+\beta})'} \right] , 
\]
which leads as above to
\begin{align*}
\E\left[ \sup_{t\in [0,2]} \normUU{\rDWT}{(C^{2+\beta})'}^2\right] &  \leq C \kappa^2. 
\end{align*}
Combining this inequality with \eqref{azelkjrsdnfg3} we conclude that 
\be\label{azelkjrsdnfg}
\E\left[ \sup_{t\in [0,T]} e^{-\delta t} \normUU{\rDWT}{(C^{2+\beta})'}^2\right]\leq C\kappa^2. 
\ee\\

\textbf{Step 2:} We now apply Lemma \ref{lemCLLP2.2} between $t$ and $T$ (for $t\in [0,T]$) to estimate  $D \shif z^{\delta,T}$. Assuming that $\delta <\lambda$, where $\lambda$ is the constant given in Lemma \ref{lemCLLP2.2}, we have by \eqref{azelkjrsdnfg}: 
\begin{align*}
\|\tilde{\shif z}^{\delta,T}_t\|_{L^2_{\omega,x}} & \leq Ce^{-\lambda (T-t)}\|\tilde{\shif z}^{\delta,T}_T\|_{L^2_{\omega,x}}+ C\int_t^T e^{-\lambda(s-t)}\| \lg \frac{ \delta \shif f_s}{\delta m}, \boldsymbol{\rho}^{\delta,T}_s\rg \|_{L^2_{\omega,x}}ds \\ 
& \leq C e^{\delta t/2} \kappa+ C\int_t^T e^{-\lambda(s-t)} \E\left[  \normUU{\boldsymbol{\rho}^{\delta,T}_s}{(C^{2+\beta})'}^2\right]^{1/2}  ds \\
& \leq  C e^{\delta t/2} \kappa+ C\kappa \int_t^T e^{-\lambda(s-t)} e^{\delta s/2} ds \; \leq C e^{\delta t/2} \kappa. 
\end{align*}
Our aim is to improve this estimate into a bound on the $C^3$ regularity of $\tilde{\shif z}^{\delta,T}$. For this, recall that, by Lemma \ref{lemCLLP2.2}, we have, for any $0\leq t_0\leq t\leq T$,  
$$
(t-t_0) \normtwo{D  \shif z^{\delta,T}_{t_0}}   \leq C(t-t_0+1) \left( \normtwo{\tilde{\shif z}^{\delta,T}_t} + \int_{t_0}^t (\normtwo{\tilde{\shif z}^{\delta,T}_s}+ \normtwo{ \lg \frac{ \delta \shif f_s}{\delta m}, \boldsymbol{\rho}^{\delta,T}_s\rg})ds \right) .
$$
Choosing $t=t_0+1$, this implies that, for $t_0\leq T-1$,  
\be\label{azelkjrsdnfg2}
 \normtwo{D  \shif z^{\delta,T}_{t_0}}  \leq C e^{\delta t_0/2} \kappa . 
\ee

Next we  prove an $L^\infty$ estimate for $D\shif z^{\delta,T}$ on the time interval $[T-1,T]$. Recalling formula (4.2) in \cite{CDLL}, we have 
\begin{align*}
 D\shif z^{\delta, T}_t(x) & = 
\E\Bigl[\Gamma_{T-t}\ast D\shif z^{\delta, T}_T(x) +  \int_t^{T} D\Gamma_{s-t}\ast \left(D\shif u^\delta_s \cdot D\shif z^{\delta, T}_s \right)(x)ds \\
& \qquad +
\int_t^{T}  \Gamma_{s-t} \ast D \left( \delta \shif z^{\delta,T}_s +\lg \frac{ \delta  \shif f_s}{\delta m}, \boldsymbol{\rho}^{\delta,T}_s\rg \right) (x) ds  \bigg| \mathcal F_t \Bigr],
\end{align*}
where $\Gamma$ is the heat kernel on $\T^d$. 
The Lipschitz bound on $\shif u^{\delta}$ then imply:
\begin{align*}
\E\left[ \|D\shif z^{\delta, T}_t\|_\infty^2\right] & \leq C \E\left[ \|D\shif z^{\delta, T}_T\|_\infty^2\right]+  
C \E \left[ \left( \int_t^{T} \| D\Gamma_{s-t}\ast D\shif z^{\delta, T}_s\|_\infty ds\right)^2\right]  \\
& \qquad + C \E \left[  \left( \int_t^{T} \left\|  \Gamma_{s-t} \ast D \left( \delta \shif z^{\delta,T}_s +\lg \frac{ \delta  \shif f_s}{\delta m}, \boldsymbol{\rho}^{\delta,T}_s\rg \right) \right\|_\infty ds \right)^2 \right]. 
\end{align*}
We estimate each term separately. We note first that 
$$
\E\left[ \|D\shif z^{\delta, T}_T\|_\infty^2\right]\leq Ce^{\delta T} \kappa^2.
 $$
 Second, 
\begin{align*}
\E & \left[ \left( \int_t^{T}  \| D\Gamma_{s-t}\ast D\shif z^{\delta, T}_s\|_\infty ds \right)^2 \right] \leq \E \left[ \left( \int_t^{T} C(s-t)^{-1/2} \| D\shif z^{\delta, T}_s\|_\infty ds \right)^2 \right] \\
&\leq C \E \left[ \int_t^T (s-t)^{-1/2} ds \cdot \int_t^T (s-t)^{-1/2} \| D\shif z^{\delta, T}_s\|_\infty^2 ds \right] \leq C \int_t^T (s-t)^{-1/2} \E \left[ \| D\shif z^{\delta, T}_s\|_\infty^2 \right] ds.
\end{align*}
Finally, by the regularity of $f$ and \eqref{azelkjrsdnfg},  
\begin{align*}
\E \left[ \left( \int_t^{T} \left\|  \Gamma_{s-t} \ast D \left( \delta \shif z^{\delta,T}_s +\lg \frac{ \delta  \shif f_s}{\delta m}, \boldsymbol{\rho}^{\delta,T}_s\rg \right) \right\|_\infty ds \right)^2 \right]
 \leq \int_t^{T} \delta^2 \E\left[\| D \shif z^{\delta,T}_s \|_\infty^2 \right] + \E\left[\left\| D\lg \frac{ \delta  \shif f_s}{\delta m}, \boldsymbol{\rho}^{\delta,T}_s\rg \right\|_\infty^2 \right] ds \\
\leq \int_t^{T} \delta^2 \E \left[ \| D\shif z^{\delta, T}_s\|_\infty^2\right] ds + C \int_t^{T}\E\left[\|\boldsymbol{\rho}^{\delta,T}_s \|^2_{(C^{2+\beta})'}\right] ds \leq \int_t^{T} \delta^2 \E \left[ \| D\shif z^{\delta, T}_s\|_\infty^2\right] ds + C e^{\delta T} \kappa^2. 
\end{align*}
We deduce that $\E\left[ \|D\shif z^{\delta, T}_t\|_\infty\right]$ satisfies the differential inequality: 
$$
\E\left[ \|D\shif z^{\delta, T}_t\|_\infty^2 \right] \leq \int_t^{T} \left( C(s-t)^{-1/2} +\delta^2 \right) \E \left[ \| D\shif z^{\delta, T}_s\|_\infty^2 \right] ds+ C e^{\delta T} \kappa^2.
$$
Following \cite[Lemma 7.1.1]{Hen}, this implies the existence of a constant $C>0$ such that 
\be\label{qousildjfmkcvBIS}
\E\left[ \|D\shif z^{\delta, T}_t\|_\infty^2 \right] \leq C e^{\delta T} \kappa^2 \qquad \forall t\in [T-1,T].
\ee
We now extend this estimate to $[0,T-1]$. Let us write the representation formula for $D\shif z^{\delta, T}$: for $t\in [0,T-1]$ and $h\in (0,1]$, 
\begin{align*}
D\shif z^{\delta, T}_t(x) = \E\bigg[ \Gamma_h\ast D\shif z^{\delta, T}_{t+h} (x) +\int_t^{t+h} & D\Gamma_{s-t}\ast \left(D\shif u^\delta_s \cdot D\shif z^{\delta, T}_s \right)(x) ds \\
& + \int_t^{t+h} \Gamma_{s-t} \ast D \left( \delta \shif z^{\delta,T}_s +\lg \frac{ \delta  \shif f_s}{\delta m}, \boldsymbol{\rho}^{\delta,T}_s\rg \right) (x) ds \ \bigg|\ \mathcal F_t \bigg].
\end{align*}
Thus, again, by the regularity of $\shif u^\delta$, we have
\begin{align*}
\E\left[ \|D\shif z^{\delta, T}_t\|_\infty^2 \right] \leq  
\E\left[\| \Gamma_h\ast D\shif z^{\delta, T}_{t+h} \|_\infty^2 \right] &+ C \E \left[ \left( \int_t^{t+h} \| D\Gamma_{s-t}\ast D\shif z^{\delta, T}_s\|_\infty ds \right)^2 \right]\\
& \qquad + \int_t^{t+h}\E\left[ \left\|  \Gamma_{s-t} \ast D \left( \delta \shif z^{\delta,T}_s +\lg \frac{ \delta  \shif f_s}{\delta m}, \boldsymbol{\rho}^{\delta,T}_s\rg \right) \right\|_\infty^2\right] ds \ . 
\end{align*}
We again estimate each term separately. Note first that 
\begin{align*}
\| \Gamma_h\ast D\shif z^{\delta, T}_{t+h} \|_\infty  \leq \|\Gamma_h\|_{L^2_x}\|D\shif z^{\delta, T}_{t+h} \|_{L^2_x}\leq Ch^{-d/4} \|D\shif z^{\delta, T}_{t+h} \|_{L^2_x}, 
\end{align*}
so that, by \eqref{azelkjrsdnfg2} and  \eqref{qousildjfmkcvBIS}, 
\begin{align*}
\E\left[ \| \Gamma_h\ast D\shif z^{\delta, T}_{t+h} \|_\infty^2 \right]   \leq Ch^{-d/2} e^{\delta (t+h)}  \kappa^2.
\end{align*}
Second, exactly as above, 
\begin{align*}
\E \left[ \left( \int_t^{t+h} \| D\Gamma_{s-t}\ast D\shif z^{\delta, T}_s\|_\infty ds \right)^2 \right] \leq \int_t^{t+h} C(s-t)^{-1/2} \E \left[ \| D\shif z^{\delta, T}_s\|_\infty^2 \right] ds.
\end{align*}
Third, again by the regularity of $f$ and \eqref{azelkjrsdnfg},  
{\small
\begin{align*}
\int_t^{t+h}\E\left[ \left\|  \Gamma_{s-t} \ast D \left( \delta \shif z^{\delta,T}_s +\lg \frac{ \delta  \shif f_s}{\delta m}, \boldsymbol{\rho}^{\delta,T}_s\rg \right) \right\|_\infty^2 \right] ds
\leq \int_t^{t+h} \delta^2 \E\left[\| D \shif z^{\delta,T}_s \|_\infty^2\right] + \E\left[\left\| D\lg \frac{ \delta  \shif f_s}{\delta m}, \boldsymbol{\rho}^{\delta,T}_s\rg \right\|_\infty^2 \right] ds \\
\leq \int_t^{t+h} \delta^2 \E\left[\| D \shif z^{\delta,T}_s \|_\infty^2\right] ds + C \int_t^{t+h}\E\left[\|\boldsymbol{\rho}^{\delta,T}_s \|_{C^{2+\beta})'}^2\right] ds \leq \int_t^{t+h} \delta^2 \E\left[\| D \shif z^{\delta,T}_s \|_\infty^2 \right] ds + C e^{\delta(t+h)} \kappa^2. 
\end{align*}
}
Fix $K$ large enough to be chosen later and let 
$$
t_0=\sup\left\{t \in [0,T], \; \E\left[ \|D\shif z^{\delta, T}_t\|_\infty^2 \right]> K e^{\delta t}\kappa^2 \right\}.
$$
If the condition in the right-hand side is never met we set $t_0=0$. By \eqref{qousildjfmkcvBIS}, we can choose  $K$  large enough (independently of $\delta$ and $\kappa$) such that $t_0\leq T-1$. Assume for a while that $t_0>0$. Then, collecting the estimates above, we have  
\begin{align*} 
&\E\left[ \|D\shif z^{\delta, T}_{t_0}\|_\infty^2 \right]= K e^{\delta t_0}\kappa^2 \\
&\qquad \leq 
C h^{-d/2} e^{\delta(t_0+h)} \kappa^2 + \int_{t_0}^{t_0+h} \left( C(s-t_0)^{-1/2} + \delta^2 \right) \E \left[ \| D\shif z^{\delta, T}_s\|_\infty^2 \right] ds
+ C e^{\delta (t_0+h)}  \kappa^2 \\ 
&\qquad  \leq C h^{-d/2} e^{\delta (t_0+h)} \kappa^2 + K \kappa^2 \int_{t_0}^{t_0+h} \left( C(s-t_0)^{-1/2} + \delta^2 \right)  e^{\delta s} ds
+ C e^{\delta(t_0+h)} \kappa^2. 
\end{align*}
We simplify the expression to find: 
\begin{align*} 
K & \leq 
C h^{-d/2} e^{\delta h} + K \int_{0}^{h} \left( C s^{-1/2} + \delta^2 \right)  e^{\delta s} ds
+ C e^{\delta h} . 
\end{align*}
We first choose $h>0$ small enough such that $ \int_{0}^{h} ( C s^{-1/2} + \delta^2)  e^{\delta s} ds\leq 1/2$, which is independent of $\delta \in (0, 1]$ and $K$. Then we choose 
$K> 2(C h^{-d/2} e^{\delta h} + C e^{\delta h})$ to find a contradiction. Therefore $t_0=0$, which  proves that 
\be\label{bazelkjrsdnfg22}
\E\left[ \|D\shif z^{\delta, T}_t\|_\infty^2 \right] \leq C e^{\delta t} \kappa^2 \qquad \forall t\in [0,T].
\ee

\textbf{Step 3:} In this step we now improve inequality \eqref{bazelkjrsdnfg22} into and estimates which hold locally uniformly in time. Let us write again the representation formula for $D \shif z^{\delta, T}$: for $t \in [\tau, T]$ and $s \in [t - \tau, t]$, with $\tau$ positive and small enough, we have
\begin{align*}
D\shif z^{\delta, T}_s(x) = \E\bigg[ \Gamma_{ t-s} \ast D\shif z^{\delta, T}_{ t } (x) +\int_s^{t} & D\Gamma_{r-s}\ast \left(D\shif u^\delta_r \cdot D\shif z^{\delta, T}_r \right)(x) dr \\
& + \int_s^{t} \Gamma_{r-s} \ast D \left( \delta \shif z^{\delta,T}_r +\lg \frac{ \delta  \shif f_r}{\delta m}, \boldsymbol{\rho}^{\delta,T}_r\rg \right) (x) dr \ \bigg|\ \mathcal F_s \bigg].
\end{align*}
Thus, by Young's inequality, the regularity of $\shif u^\delta$, properties of the heat kernel $\Gamma$, and regularity of $f$, there holds:
\begin{align*}
\|D\shif z^{\delta, T}_s\|_\infty &\leq \E \left[ \left\| \Gamma_{ t-s} \ast D\shif z^{\delta, T}_{ t } \right\|_{\infty} \bigg| \mathcal{F}_s \right]\\
& \quad + \E \left[ \int_s^{t} \left\| D\Gamma_{r-s}\ast \left( D \shif u^\delta_r \cdot D\shif z^{\delta, T}_r \right) \right\|_\infty + \left\|  \Gamma_{r-s} \ast D \left( \delta \shif z^{\delta,T}_r +\lg \frac{ \delta  \shif f_r}{\delta m}, \boldsymbol{\rho}^{\delta,T}_r\rg \right) \right\|_\infty dr \bigg| \mathcal{F}_s \right]\\
& \leq C \, \E \bigg[ \|D\shif z^{\delta, T}_{t} \|_{\infty} + \sup_{t-\tau \leq \tilde s \leq t} \int_{\tilde s}^{t} \left( \,( r - \tilde s \, )^{-\frac 12} + \delta \right)  \| D\shif z^{\delta, T}_r\|_\infty +  \|\boldsymbol{\rho}^{\delta,T}_r \|_{(C^{2+\beta})'} dr \bigg| \mathcal{F}_s \bigg] \\
& \leq C \, \E \bigg[ \|D\shif z^{\delta, T}_{t} \|_{\infty} + \sqrt{\tau} \sup_{t-\tau \leq r \leq t}  \| D\shif z^{\delta, T}_r\|_\infty + \tau \sup_{t-\tau \leq r \leq t} \|\boldsymbol{\rho}^{\delta,T}_r \|_{(C^{2+\beta})'} \bigg| \mathcal{F}_s \bigg],
\end{align*}
where the last inequality holds since,
\begin{align*}
\int_{\tilde s}^{t} \left( \,( r - \tilde s \, )^{-\frac 12} + \delta \right) & \| D\shif z^{\delta, T}_r\|_\infty dr \leq \int_{\tilde s}^{t} \left( \,( r - \tilde s \, )^{-\frac 12} + \delta \right) dr \cdot \sup_{\tilde s \leq r \leq t} \| D\shif z^{\delta, T}_r\|_\infty \\
& \leq \left( (t-\tilde s)^{\frac 12} + \delta(t - \tilde s) \right) \sup_{t-\tau \leq r \leq t} \| D\shif z^{\delta, T}_r\|_\infty \leq C \sqrt{\tau} \sup_{t-\tau \leq r \leq t} \| D\shif z^{\delta, T}_r\|_\infty.
\end{align*}
Hence, by Doob's inequality, inequalities \eqref{azelkjrsdnfg} and  \eqref{bazelkjrsdnfg22}, and the assumption that $\tau$ is small enough, we derive:
\begin{align*}
\E \bigg[ & \sup_{t-\tau \leq s \leq t} \|D\shif z^{\delta, T}_s\|_\infty^2 \bigg] \\
&\leq C \E \Bigg[ \sup_{t-\tau \leq s \leq t} \E \bigg[ \|D\shif z^{\delta, T}_{t} \|_{\infty} + \sqrt{\tau} \sup_{t-\tau \leq r \leq t}  \| D\shif z^{\delta, T}_r\|_\infty + \tau \sup_{t-\tau \leq r \leq t} \|\boldsymbol{\rho}^{\delta,T}_r \|_{(C^{2+\beta})'} \bigg| \mathcal{F}_s \bigg]^2 \Bigg]\\
& \leq C \EP{ \left(  \|D\shif z^{\delta, T}_{t} \|_{\infty} + \sqrt{\tau} \sup_{t-\tau \leq s \leq t}  \| D\shif z^{\delta, T}_s\|_\infty + \tau \sup_{t-\tau \leq s \leq t} \|\boldsymbol{\rho}^{\delta,T}_s \|_{(C^{2+\beta})'} \right)^2} \\
& \leq C \tau \EP{ \sup_{t-\tau \leq s \leq t} \| D\shif z^{\delta, T}_s\|_\infty^2 } + C \EP{ \|D\shif z^{\delta, T}_{t} \|_{\infty}^2 } +  C \tau^2 \EP{ \sup_{t-\tau \leq s \leq t} \|\boldsymbol{\rho}^{\delta,T}_s \|_{(C^{2+\beta})'}^2 } \\
& \leq \frac 12 \E \bigg[ \sup_{t-\tau \leq s \leq t} \|D\shif z^{\delta, T}_s\|_\infty^2 \bigg] + C e^{\delta t} \kappa^2.
\end{align*}
Therefore,
\be\label{azelkjrsdnfg22}
\sup_{\tau \leq t \leq T} \E \bigg[ e^{-\delta t} \sup_{t-\tau \leq s \leq t} \|D\shif z^{\delta, T}_s\|_\infty^2 \bigg] \leq C \kappa^2.
\ee

\textbf{Step 4:} We finally bootstrap the above proof to obtain a higher order estimate on $D\shif z^{\delta, T}$. Let us consider the equation satisfied by $D\shif z^{\delta,T}\cdot e$ for some direction $e\in\R^d$,  $|e|\leq 1$: it is an equation of the form \eqref{eq.backward_stochastic},  stated on $[0,T]\times \T^d$,  with
$$
A=D^2 \shif u^\delta e \cdot D\shif z^{\delta,T} +  \lg \frac{ \delta D_e \shif f_s}{\delta m}, \boldsymbol{\rho}^{\delta,T}_s\rg.
$$
By \eqref{azelkjrsdnfg} and \eqref{azelkjrsdnfg22}, we have 
$$
\sup_{\tau \leq t \leq T} \E \bigg[ e^{-\delta t} \sup_{t-\tau \leq s \leq t} \|A_s\|_\infty^2 \bigg] \leq C \kappa^2.
$$
Estimating  $\|D(D\shif z^{\delta,T}\cdot e)\|_{L^2_{\omega,x}}$ as we did for \eqref{azelkjrsdnfg2}, and  $\E[\|D(D\shif z^{\delta,T}\cdot e)\|_{\infty}]$, as we did for \eqref{qousildjfmkcvBIS}, and then arguing exactly as above, we obtain 
$$
\sup_{\tau \leq t \leq T} \E \bigg[ e^{-\delta t} \sup_{t-\tau \leq s \leq t} \|D^2\shif z^{\delta, T}_s\|_\infty^2 \bigg] \leq C \kappa^2.
$$
Taking a second time the derivative in space of the equation for $\shif z^{\delta,T}$   (recalling that $\shif u^\delta$ is bounded in $C^{4+\beta}$ for any $\beta\in (0, \alpha)$) we conclude in the same way that 
$$
\sup_{\tau \leq t \leq T} \E \bigg[ e^{-\delta t} \sup_{t-\tau \leq s \leq t} \|D^3\shif z^{\delta, T}_s\|_\infty^2 \bigg] \leq C \kappa^2.
$$
Finally, we estimate the increment $D^3\shif z^{\delta, T}_t(x+h)-D^3\shif z^{\delta, T}_t(x)$ in the same way to conclude that, for any $\beta'\in (0,\beta)$,  
\be\label{oijzkendflk}
\sup_{\tau \leq t \leq T} \E \bigg[ e^{-\delta t} \sup_{t-\tau \leq s \leq t} \|D^3\shif z^{\delta, T}_s\|_{C^{\beta'}}^2 \bigg] \leq C \kappa^2.
\ee
By the definition of $\kappa$ in \eqref{def.kappa}, this shows that, for any $\beta\in (0,\alpha)$ and $t\in[\tau,T]$,
\begin{align*}
& \E \bigg[ e^{-\delta t} \sup_{t-\tau \leq s \leq t} \|\tilde{\shif z}^{\delta, T}_t\|_{C^{3+\beta}}^2\bigg]^{\frac 12}  \\
& \qquad \leq C_\beta \Bigl(\E\left[\|\tilde{\shif z}^{\delta,T}_{0}\|_{C^{2+\beta}}^2\right]^{\frac 14}  \E\left[ \normU{\boldsymbol{\varrho}_{0}^{\delta,T}}{(C^{2+\beta})'}^2\right]^{\frac 14}+ \E\left[\|\tilde{\shif z}^{\delta,T}_{T}\|_{C^{3+\beta}}^2\right]^{\frac 12}+ \E\left[ \normU{\boldsymbol{\varrho}_{0}^{\delta,T}}{(C^{2+\beta})'}^2\right]^{\frac 12} \Bigr). \notag
\end{align*}
We choose $t=\tau$ and use Young's inequality to conclude that 
$$
\E\left[\|\tilde{\shif z}^{\delta,T}_0\|_{C^{2+\beta}}\right]\leq C (  \E\left[ \normU{\boldsymbol{\varrho}_{0}^{\delta,T}}{(C^{2+\beta})'}^2\right]^{\frac 12} +\E\left[\|\tilde{\shif z}^{\delta,T}_{T}\|_{C^{3+\beta}}^2\right]^{\frac 12}).
$$ 
This implies in turn that 
$$
\kappa\leq  C ( \E\left[ \normU{\boldsymbol{\varrho}_{0}^{\delta,T}}{(C^{2+\beta})'}^2\right]^{\frac 12} +\E\left[\|\tilde{\shif z}^{\delta,T}_{T}\|_{C^{3+\beta}}^2\right]^{\frac 12}).
$$
We then derive \eqref{kzjhejhezrbnldkj} from  \eqref{azelkjrsdnfg3}, \eqref{azelkjrsdnfg} and \eqref{oijzkendflk}. 
\end{proof}

Our next step is to obtain decay estimates of the solution of the approximate linearized system uniform in the parameter $\delta$. 

\begin{lem} \label{lem.esti.zDWT2}  There exists $\lambda'>0$ and $\delta_0>0$, depending only on $f$ and $d$, such that, if $\delta\in (0,\delta_0)$ and if   $(\zDWT, \rDWT, \NDWT )$ is the solution to \eqref{eq.linearized_MFG_shifted}-(i)-(ii) on the time interval $[t_0,T]$ with an arbitrary initial condition $\boldsymbol{\varrho}_{0}^{\delta,T}$ and a vanishing terminal condition $\tilde{\shif z}^{\delta,T}_{T}=0$, then, for any $t_1\in [0,T]$, 
\begin{align*}
&\sup_{t_1+\tau \leq t \leq T} \E\left[ e^{-\delta (t-t_0)} \sup_{t - \tau \leq s \leq t} \| \tilde{\shif z}^{\delta,T}_s\|_{C^{3+\beta}}^2\right]+ \E\left[ \sup_{t_1 \leq t \leq T} e^{- \delta (t-t_0)} \normUU{\rDWT}{(C^{2+\beta})'}^2\right] \\
%+ \E\left[ \sup_{t_1+2 \leq t \leq T} e^{-\delta t}\normtwox{\boldsymbol{\varrho}_t^{\delta,T}}^2\right] 
%\notag \\
%& \ 
& \qquad \qquad \leq C e^{-\lambda' (t_1-t_0)}  \E\left[\|\boldsymbol{\varrho}_{0}^{\delta,T}\|_{(C^{2+\beta})'}^2\right].
\end{align*}
\end{lem}

\begin{proof} As before we do the proof when $t_0=0$. The key step of  the proof is the following decay estimate: there exist constants $\delta_0, \lambda' > 0$ and $C > 0$, independent of $T$, such that for any $\delta \in (0, \delta_0)$,
\be\label{estikeykey}
\|\tilde{\shif z}^{\delta,T}_t\|_{L^2_{\omega,x}}^2 + \|\boldsymbol{\varrho}^{\delta,T}_t\|_{L^2_{\omega,x}}^2 \leq C e^{-\lambda' (t-1)} \|\boldsymbol{\varrho}^{\delta,T}_2\|_{L^2_{\omega,x}}^2 \qquad \forall\, 2 \leq t < T.
\ee
To prove \eqref{estikeykey}, we aim to apply Lemma~\ref{lem.expodecayDISC}, with
\[
\left\{
\begin{array}{ll}
\alpha(t) = \|\tilde{\shif z}^{\delta,T}_t\|_{L^2_{\omega,x}}^2, \quad \beta(t) = \|D \shif z^{\delta,T}_t\|_{L^2_{\omega,x}}^2, \quad \gamma(t) = \|\boldsymbol{\varrho}^{\delta,T}_t\|_{L^2_{\omega,x}}^2 & \text{for } 2 \leq t < T, \\[1.2ex]
\alpha(t) = \beta(t) = \gamma(t) = 0 & \text{for } t \geq T,
\end{array}
\right.
\]
where $\alpha$, $\beta$, and $\gamma$ are c\`adl\`ag, as required in Appendix~\ref{sec:AppendExpo}. \\

The proof closely follows that of Lemma~\ref{lem.uniqueDISC}. Recall that on $[2,T]$, the estimates from Lemma~\ref{lem.boundsDisc} hold: $C^{-1} \leq \shif m^\delta \leq C$ and $|D\shif u^\delta| \leq C$. Using the estimate on $\shif m^\delta$ and the fact that $\int_{\mathbb{T}^d} \boldsymbol{\varrho}^{\delta,T}_t = 0$, we obtain by duality and for any $t_1 \geq 2$:
\begin{align*}
\mathbb{E} \left[ \int_{t_1}^{T} e^{-\delta t} \int_{\mathbb{T}^d} |D \shif z^{\delta,T}_t|^2 \, dx \, dt \right]
& \leq C \mathbb{E} \left[ \int_{t_1}^{T} e^{-\delta t} \int_{\mathbb{T}^d} \shif m^\delta |D \shif z^{\delta,T}_t|^2 \, dx \, dt \right] \\
& \leq e^{-\delta t_1} \mathbb{E} \left[ \int_{\mathbb{T}^d} \shif z^{\delta,T}_{t_1} \boldsymbol{\varrho}^{\delta,T}_{t_1} \, dx \right] \\
& = e^{-\delta t_1} \mathbb{E} \left[ \int_{\mathbb{T}^d} \tilde{\shif z}^{\delta,T}_{t_1} \boldsymbol{\varrho}^{\delta,T}_{t_1} \, dx \right] \\
& \leq e^{-\delta t_1} \left( \alpha(t_1) \gamma(t_1) \right)^{1/2}.
\end{align*}

This gives \eqref{hypexpo1DISC} since $\beta(t) = 0$ for $t \geq T$. Moreover, \eqref{hypexpo2} holds for all $1 \leq t_1 \leq t_2 < T$ by Lemma~\ref{lemCLLP2.1} applied to $\boldsymbol{\varrho}^{\delta,T}$, and for $t_2 \geq T$ since $\gamma(t) = 0$ for $t \geq T$. Estimate \eqref{hypexpo4} holds by the Poincar\'e inequality and the fact that $\alpha(t) = 0$ for $t \geq T$.\\

Finally, \eqref{hypexpo3DISC} is a direct consequence of Lemma~\ref{lemCLLP2.2} applied to $\shif z^{\delta,T}$, which implies that, for any $2 \leq t_1 \leq t_2 \leq T$,
\[
\alpha(t_1)^{1/2} \leq C e^{-\lambda(t_2 - t_1)} \alpha(t_2)^{1/2} + C \int_{t_1}^{t_2} e^{-\lambda(s - t_1)} \|\boldsymbol{\varrho}^{\delta,T}_s\|_{L^2_{\omega,x}} \, ds.
\]
Letting $t_2 = T$ yields:
\[
\alpha(t_1)^{1/2} \leq C \int_{t_1}^{T} e^{-\lambda(s - t_1)} \|\boldsymbol{\varrho}^{\delta,T}_s\|_{L^2_{\omega,x}} \, ds 
\leq \sup_{s \geq t_1} e^{-\lambda(s - t_1)/2} \|\boldsymbol{\varrho}^{\delta,T}_s\|_{L^2_{\omega,x}} \int_{t_1}^{\infty} e^{-\lambda(s - t_1)/2} \, ds.
\]
Squaring both sides gives:
\[
\alpha(t_1) \leq C \sup_{t \geq t_1} e^{-\lambda(t - t_1)} \gamma(t),
\]
which establishes \eqref{hypexpo3DISC}. Therefore, by Lemma~\ref{lem.expodecayDISC}, there exist constants $\delta_0 > 0$, $\lambda' > 0$, and $C > 0$ independent of $T$ such that, for all $\delta \in (0, \delta_0)$,
\[
\alpha(t) + \gamma(t) \leq C e^{-\lambda'(t - 1)} \gamma(2) \qquad \forall\, t \geq 2.
\]
This concludes the proof of \eqref{estikeykey}. \\ 

Combining inequality \eqref{kzjhejhezrbnldkj} and  \eqref{estikeykey} we infer that
\be\label{kzjhejhezrbnldkj3}
\|\tilde{\shif z}^{\delta,T}_t\|_{L^2_{\omega,x}}^2 + \|\boldsymbol{\varrho}^{\delta,T}_t\|_{L^2_{\omega,x}}^2 \leq C e^{-\lambda' t} \|\rho_0\|^2_{(C^{2+\beta})'} \qquad \forall\, 2 \leq t < T.
\ee
Furthermore, using estimate \eqref{zerldjkjnpimj2} from Lemma~\ref{lemCLLP2.2}, we obtain a similar control on $D\shif z^\delta$:
\be\label{kzjhejhezrbnldkj4}
\|D \shif z^{\delta,T}_t\|_{L^2_{\omega,x}}^2 \leq C e^{-\lambda' t} \|\rho_0\|^2_{(C^{2+\beta})'} \qquad \forall\, 2 \leq t \leq T - 1.
\ee
Hence we obtain 
\be\label{kzjhejhezrbnldkjBIS}
\E\left[\|D{\shif z}^{\delta,T}_t\|_{L^2_x}^2\right]+\E\left[  \normUU{\rDWT}{L^2_x}^2\right] \leq C e^{-\lambda' t}  \|\rho_{0}\|_{(C^{2+\beta})'}^2, \qquad \forall~2 \leq t \leq T-1.
\ee

Now note that, for any $t_1\in [2, T-2]$, $(\zDWT, \rDWT, \NDWT )$ is the solution to \eqref{eq.linearized_MFG_shifted} on the time interval $[t_1,T]$ with an initial condition $\boldsymbol{\varrho}_{t_1}^{\delta,T}$ and a vanishing terminal condition  $\tilde{\shif z}^{\delta,T}_T=0$. Thus, applying first Lemma \ref{lem.esti.zDWT1} and then \eqref{kzjhejhezrbnldkjBIS} we obtain:
\begin{align*}
&\sup_{t_1+\tau \leq t \leq T} \E\left[ e^{-\delta t} \sup_{t - \tau \leq s \leq t} \| \tilde{\shif z}^{\delta,T}_s\|_{C^{3+\beta}}^2\right]+ \E\left[ \sup_{t_1 \leq t \leq T} e^{- \delta t} \normUU{\rDWT}{(C^{2+\beta})'}^2\right]+ \E\left[ \sup_{t_1+2 \leq t \leq T} e^{-\delta t}\normtwox{\boldsymbol{\varrho}_t^{\delta,T}}^2\right] \notag \\
&\qquad \qquad \leq C_\beta  \E\left[\|\boldsymbol{\varrho}_{t_1}^{\delta,T}\|_{(C^{2+\beta})'}^2\right] \ \leq C_\beta  \left[\|\boldsymbol{\varrho}_{t_1}^{\delta,T}\|_{L^2_x}^2\right] \ \leq C e^{-\lambda' t_1}  \|\rho_{0}\|_{(C^{2+\beta})'}^2.
\end{align*}
As this estimate also holds for $t_1\in [0,2]$ in view of Lemma \ref{lem.esti.zDWT1}, the proof of the lemma is complete. 
\end{proof}

We now address the existence and the uniqueness of the solution to \eqref{eq.linearized_MFG_shifted}. Note that Theorem \ref{thm.good_bound_for_linearized_system_discounted} is included in the following proposition.

\begin{prop}\label{prop.sol_shif.lin} There exists a unique solution $(\zDW, \rDW, \NDW )$ to \eqref{eq.linearized_MFG_shifted} such that \eqref{cond.unicite.zDW} holds. In addition, \eqref{kjezhrbdfnl} holds. 
%In addition, we have positive constants $\tau$ and, for any $\beta\in (0,\alpha)$, $C_\beta$ such that,
%\be\label{lkjnsrdlfkjn1}
%\sup_{t\geq t_0+\tau} \E\left[ e^{-\delta (t-t_0)} \sup_{t-\tau \leq s \leq t} \|{\shif z}^{\delta}_s\|_{C^{3+\beta}}^2\right]+ \E\left[ \sup_{t \geq t_0} e^{-\delta (t-t_0)} \normUU{\rDW}{(C^{2+\beta})'}^2\right] \leq C_\beta  \|\rho_{0}\|_{(C^{2+\beta})'}^2
%\ee
%and
%\be\label{maejzjnretf}
%\E\left[\int_{\T^d} \zDW \rDW\right]\geq 0\qquad \forall t\geq t_0.
%\ee
\end{prop}

\begin{proof} As before we do the proof when $t_0=0$. \\

\textbf{Step 1:}
We first prove the existence of $D \zDW$ and $\rDW$. Let us warn the reader that  $\shif z^\delta$ will only be built later, so that  $D \zDW$ is for the moment merely a notation. For $2< T_1< T_2$, let us set $(\shif z_t, \boldsymbol{\varrho}_t, \shif N_t)= (\shif z^{\delta, T_2}_t-\shif z^{\delta, T_1}_t, \boldsymbol{\varrho}^{\delta,T_2}_t-\boldsymbol{\varrho}^{\delta,T_1}_t, \shif N^{\delta, T_2}_t- \shif N^{\delta, T_1}_t)$. We note that  $(\shif z_t, \boldsymbol{\varrho}_t, \shif N_t)$ solves  \eqref{eq.linearized_MFG_shifted}-(i)-(ii) on $[0,T_1]$, that $\boldsymbol{\varrho}$ vanishes at $t=0$ and that $\shif z_{T_1}= \shif z^{\delta,T_2}_{T_1}$. Lemma \ref{lem.esti.zDWT1} implies that, for any horizon $\theta>0$, 
\begin{align*}
\sup_{\tau \leq t \leq \theta} \E\left[ e^{-\delta t} \sup_{t - \tau \leq s \leq t} \| \tilde{\shif z}_s\|_{C^{3+\beta}}^2\right]+ \E\left[ \sup_{0 \leq t \leq \theta} e^{- \delta t} \normUU{\boldsymbol{\varrho}}{(C^{2+\beta})'}^2\right]  \leq C_\beta \E\left[\|\tilde{\shif z}^{\delta,T_2}_{T_1}\|_{C^{3+\beta}}^2\right] .
\end{align*}
We can estimate the right-hand side by applying  Lemma \ref{lem.esti.zDWT1} 
to $(\shif z^{\delta, T_2}_t, \boldsymbol{\varrho}^{\delta,T_2}_t, \shif N^{\delta, T_2}_t)$:
$$
\E\left[ e^{-\delta T_1} \| \tilde{\shif z}^{\delta, T_2}_{T_1}\|_{C^{3+\beta}}^2\right] \leq C_\beta e^{-\lambda' T_1} \|\rho_0\|_{(C^{2+\beta})'}^2. 
$$
Hence, for any $\delta$ small enough, 
\begin{align*}
 \E\left[  \sup_{0\leq t \leq \theta}  \| \tilde{\shif z}_t\|_{C^{3+\beta}}^2\right]
 + \E\left[ \sup_{0 \leq t \leq \theta}  \normUU{\boldsymbol{\varrho}}{(C^{2+\beta})'}^2\right]  
 \leq C_{\beta,\theta,\delta} \ e^{-\lambda' T_1/2} \|\rho_0\|_{(C^{2+\beta})'}^2.
\end{align*}

This proves that $(D  \shif z^{\delta,T})$ is a Cauchy sequence in $L^2(\Omega, C^0([0,\theta], C^2(\T^d)))$ while $(\boldsymbol{\varrho}^{\delta, T})$ is a Cauchy sequence in $L^2(\Omega; C^0([0,\theta]; (C^{2+\beta})'))$ as $T\to \infty$. Let us denote by $D\shif z^\delta$ and $\boldsymbol{\varrho}^{\delta}$ their respective limits, which are defined on $\Omega\times (0,\infty)\times \T^d$ as $\theta$ is arbitrary. Recalling Lemma \ref{lem.esti.zDWT2}, $D  \shif z^{\delta,T}$ and $\boldsymbol{\varrho}^{\delta, T}$ satisfy the following decay estimate: for any $t_1\geq 0$, 
\be\label{aezolijrsdfkgjh}
\sup_{t_1+\tau \leq t } \E\left[ e^{-\delta t} \sup_{t - \tau \leq s \leq t} \| D\shif z^{\delta}_s\|_{C^{2+\beta}}^2\right]
+ \E\left[ \sup_{t_1 \leq t } e^{- \delta t} \normUU{\boldsymbol{\varrho}^{\delta}_t}{(C^{2+\beta})'}^2\right]
\leq C_\beta e^{-\lambda' t_1}  \|\rho_{0}\|_{(C^{2+\beta})'}^2.
\ee
Recall again that, above,  we temporarily just take $D \shif z^\delta$ as a notation. We will later prove that it is indeed the derivative of a map $\shif z^\delta$ to be defined.\\

\textbf{Step 2:}
In this step, we obtain a representation formula for $D \shif z^\delta$. Thanks to the representation formula for $D \shif z^{\delta,T}$, for any $0 \leq t_1 < t_2   \leq T$, we have:
%\be\label{tmp.azer1}
$$
\begin{aligned}
D \shif z^{\delta,T}_{t_1}(x) = & \mathbb{E}\bigg[ e^{-\delta(t_2 - t_1)} \Gamma_{t_2 - t_1} \ast D \shif z^{\delta,T}_{t_2} + \int_{t_1}^{t_2} e^{-\delta(s - t_1)} \Gamma_{s - t_1} \ast D \left(D \shif u^\delta_s \cdot D \shif z^{\delta,T}_s\right)(x) \, ds \\
&\qquad \qquad \qquad - \int_{t_1}^{t_2} e^{-\delta(s - t_1)} \Gamma_{s - t_1} \ast D \left\langle \frac{\delta \shif f_s}{\delta m}(\cdot, \shif m^\delta_s), \boldsymbol{\varrho}^{\delta,T}_s \right\rangle \, ds \ \bigg| \ \mathcal{F}_{t_1} \bigg].
\end{aligned}
%\ee
$$
In view of the convergence of $(D \shif z^{\delta,T})$ to $D \shif z^{\delta}$ in $L^2(\Omega, C^0([0,\theta], C^2(\T^d)))$ and of the convergence of $(\boldsymbol{\varrho}^{\delta, T})$ to $\boldsymbol{\varrho}^{\delta}$ in $L^2(\Omega; C^0([0,\theta]; (C^{2+\beta})'))$ for any $\theta\geq 0$, and recalling the regularity of $f$, we infer that, for any $0 \leq t_1 < t_2$,  
 $$
\begin{aligned}
D \shif z^{\delta}_{t_1}(x) & =  \mathbb{E}\bigg[ e^{-\delta(t_2 - t_1)} \Gamma_{t_2 - t_1} \ast D \shif z^{\delta}_{t_2} + \int_{t_1}^{t_2} e^{-\delta(s - t_1)} \Gamma_{s - t_1} \ast D \left(D \shif u^\delta_s \cdot D \shif z^{\delta}_s\right)(x) \, ds \\
&\qquad \qquad \qquad - \int_{t_1}^{t_2} e^{-\delta(s - t_1)} \Gamma_{s - t_1} \ast D \left\langle \frac{\delta \shif f_s}{\delta m}(\cdot, \shif m^\delta_s), \boldsymbol{\varrho}^{\delta}_s \right\rangle \, ds \ \bigg| \ \mathcal{F}_{t_1} \bigg] .
%\\
%& =  \mathbb{E}\bigg[ e^{-\delta(t_2 - t_1)} \Gamma_{t_2 - t_1} \ast D \shif z^{\delta}_{t_2} + \int_{t_1}^{t_2} e^{-\delta(s - t_1)} D\Gamma_{s - t_1} \ast  \left(D \shif u^\delta_s \cdot D \shif z^{\delta}_s\right)(x) \, ds \\
%&\qquad \qquad \qquad - \int_{t_1}^{t_2} e^{-\delta(s - t_1)} D\Gamma_{s - t_1} \ast  \left\langle \frac{\delta \shif f_s}{\delta m}(\cdot, \shif m^\delta_s), \boldsymbol{\varrho}^{\delta}_s \right\rangle \, ds \ \bigg| \ \mathcal{F}_{t_1} \bigg]. 
\end{aligned}
%\ee
$$
By the decay estimate \eqref{aezolijrsdfkgjh} we can let $t_2\to \infty$ and obtain the following representation of $D \shif z^{\delta}_{t}$ for any $t \geq 0$:
\be\label{tmp.azer4}
\begin{aligned}
D \shif z^{\delta}_{t} = \E\bigg[ \int_{t}^{\infty} e^{-\delta (s-t)} D\Gamma_{s-t}\ast \left( D \shif u^\delta_s \cdot D  \shif z^{\delta}_s - \left<\frac{\delta \shif f_s}{\delta m}(\cdot, \shif m^\delta_s),\boldsymbol{\varrho}^{\delta}_s \right> \right) ds \bigg|\ \mathcal F_{t}\bigg].
\end{aligned}
\ee

\textbf{Step 3:}
We now prove that the ``$D$'' in $D \shif z^\delta$ represents an actual derivative and check that there exists $\NDW$ such that $(\zDW, \NDW)$  solves \eqref{eq.linearized_MFG_shifted}-(i). \\
%In order to construct $\shif z^\delta$, notice the representation formula of $\shif z^{\delta,T}$,
%\begin{align*}
%\shif z^{\delta,T}_t(x) &= \E\bigg[ \int_t^T e^{-\delta (s-t)} \Gamma_{s-t}\ast \left( D  \shif u^\delta_s \cdot D  \shif z^{\delta,T}_s - \left<\frac{\delta \shif f_s}{\delta m}(\cdot, \shif m^\delta_s),\boldsymbol{\varrho}^{\delta,T}_s \right> \right)(x)ds \ \bigg|\ \mathcal F_t \bigg]\\
%& = \E\left[ \int_t^\infty k^{\delta,t,T}_s (x) ds \bigg| \mathcal{F}_t \right]. 
%\end{align*}
%where $k_s^{\delta,t,T}$ is defined as:
%\[
%k_s^{\delta,t,T} (x) = 
%\left\{
%\begin{array}{ll}
%e^{-\delta (s-t)} \Gamma_{s-t}\ast \left( D  \shif u^\delta_s \cdot D  \shif z^{\delta,T}_s - \left<\frac{\delta \shif f_s}{\delta m}(\cdot, \shif m^\delta_s),\boldsymbol{\varrho}^{\delta,T}_s \right> \right)(x) &,\qquad \forall~ t \leq s \leq T, \\
%0 &,\qquad \forall~ s > T.
%\end{array}
%\right.
%\]

Let us define the random field $\mathcal{Z}^{\delta}: \Omega \times [0,\infty)\times\T^d \to \R $ as:
\be\label{formulaZ}
\mathcal{Z}^{\delta} (t,x) = \int_t^\infty 
e^{-\delta (s-t)} \Gamma_{s-t}\ast \left( D  \shif u^\delta_s \cdot D  \shif z^{\delta}_s - \left<\frac{\delta \shif f_s}{\delta m}(\cdot, \shif m^\delta_s),\boldsymbol{\varrho}^{\delta}_s \right> \right)(x) ds,\qquad \forall~s \geq t.
\ee
We view this formula as a Bochner integral in the separable Banach space $L^1( \Omega, C^1(\T^d))$. 
We now prove that it is spatially differentiable with, a.s., 
\be\label{formulaDZ}
D \mathcal{Z}^{\delta} (t,x) = \int_t^\infty 
e^{-\delta (s-t)} D \Gamma_{s-t}\ast \left( D  \shif u^\delta_s \cdot D  \shif z^{\delta}_s - \left<\frac{\delta \shif f_s}{\delta m}(\cdot, \shif m^\delta_s),\boldsymbol{\varrho}^{\delta}_s \right> \right)(x) ds,
\ee
with estimate
\be\label{formulaDZesti}
\EP{\normU{D \mathcal{Z}^\delta}{{C}^0([0,T]\times\T^d)}} < \infty \qquad \forall T>0.
\ee
Indeed, we  note  that  
\begin{align*}
& \EP{ \sup_{t\in [0,T]} \left\| \int_t^\infty e^{-\delta (s-t)} D \Gamma_{s-t}\ast \left( D  \shif u^\delta_s \cdot D  \shif z^{\delta}_s - \left<\frac{\delta \shif f_s}{\delta m}(\cdot, \shif m^\delta_s),\boldsymbol{\varrho}^{\delta}_s \right> \right) ds \right\|_{C^0} } \\
& \leq \EP{ \sup_{t\in [0,T]}  \int_t^\infty e^{-\delta (s-t)}  (s-t)^{- 1/2}  \left( \| D  \shif z^{\delta}_s\|_{C^0} + \|\boldsymbol{\varrho}^{\delta}_s\|_{(C^{2+\beta})'} \right)ds } \\
& \leq \EP{ \int_0^\infty e^{-\delta s}  s^{- 1/2}  \left( \sup_{s\leq t\leq s+T}\| D  \shif z^{\delta}_{t}\|_{C^0} +\sup_{s\leq t\leq s+T} \|\boldsymbol{\varrho}^{\delta}_{t}\|_{(C^{2+\beta})'} \right)ds } \\
& \leq  \int_0^\infty e^{-\delta s}  s^{- 1/2}  \left( \EP{\sup_{s\leq t\leq s+T}\| D  \shif z^{\delta}_{t}\|_{C^0} }+\EP{\sup_{s\leq t\leq s+T} \|\boldsymbol{\varrho}^{\delta}_{t}\|_{(C^{2+\beta})'}} \right)ds .
\end{align*}
The decay estimate \eqref{aezolijrsdfkgjh}  easily implies that 
$$
 \EP{\sup_{s\leq t\leq s+T}\| D  \shif z^{\delta}_{t}\|_{C^0} }+\EP{\sup_{s\leq t\leq s+T} \|\boldsymbol{\varrho}^{\delta}_{t}\|_{(C^{2+\beta})'}} \leq C_{T,\tau,\beta} e^{-\lambda' s/2} \|\rho_{0}\|_{(C^{2+\beta})'}^2.
 $$
 This proves that 
 \begin{align*}
& \EP{ \sup_{t\in [0,T]} \left\| \int_t^\infty e^{-\delta (s-t)} D \Gamma_{s-t}\ast \left( D  \shif u^\delta_s \cdot D  \shif z^{\delta}_s - \left<\frac{\delta \shif f_s}{\delta m}(\cdot, \shif m^\delta_s),\boldsymbol{\varrho}^{\delta}_s \right> \right) ds \right\|_{C^0} } \leq C_{T,\tau,\beta} \|\rho_{0}\|_{(C^{2+\beta})'}^2,   
\end{align*}
so that $\mathcal{Z}^{\delta}$ is $C^1$ with a derivative given by \eqref{formulaDZ} and satisfying \eqref{formulaDZesti}. \\

Now Lemma \ref{lem.extension} given below states that there is a version of $\EPt{\mathcal{Z}^\delta(t,\cdot)}$ taking paths in $C^0([0,\infty),C^1(\T^d))$ such that its spatial derivative can be written as $\EPt{D\mathcal{Z}^\delta(t,\cdot)}$, which is equal to $D \zDW$ by equality \eqref{tmp.azer4}. This leads us to set 
\be\label{formulaz}
\zDW = \EPt{\mathcal{Z}^\delta(t,\cdot)}
\ee
and to conclude that $D \zDW$ is indeed the derivative of $\zDW$ and hence not just a notation. In view of the regularity of $D \zDW$, we also note that 
$(\zDW)$ has path in $C^0([0,\infty], C^{3}(\T^d))$. Using the representation formulas \eqref{formulaZ} and \eqref{formulaz} and the tower property, we derive that for any $0 \leq t_1 < t_2 < \infty$,
$$
\shif z_{t_1}^{\delta} = \E\left[e^{-\delta (t_2-t_1)} \Gamma_{t_2-t_1} \ast \shif z_{t_2}^{\delta} + \int_{t_1}^{t_2} e^{-\delta (s-t_1)} \Gamma_{s-t_1}\ast \left( D  \shif u^\delta_s \cdot D  \shif z^{\delta}_s - \left<\frac{\delta \shif f_s}{\delta m}(\cdot, \shif m^\delta_s),\boldsymbol{\varrho}^{\delta}_s \right> \right)ds \ \Bigl|\ \mathcal F_{t_1}\right].
$$
Let $\NDW$ be defined by
\[
\NDW(x) = \zDW(x) - \shif z^\delta_0 (x) + \int_0^t  \left( \Delta \shif z^{\delta}_s - \delta \shif z^{\delta}_s - D  \shif u^\delta_s \cdot D  \shif z^{\delta}_s + \left<\frac{\delta \shif f_s}{\delta m}(\cdot, \shif m^\delta_s),\boldsymbol{\varrho}^{\delta}_s \right> \right) ds.
\]
Arguing as in  the proof of \cite[Lemma 4.3.2]{CDLL}, we have that $(\shif z^\delta, \shif N^\delta)$ is a classical solution of the following equation on $[0, t_2]$ given the terminal condition $\shif z^\delta_{t_2}$,
\[
d \shif z^{\delta}_t = \Bigl(\delta \shif z^{\delta}_t -\Delta \shif z^{\delta}_t + D  \shif u^\delta_t \cdot D  \shif z^{\delta}_t - \left<\frac{\delta \shif f_t}{\delta m}(\cdot, \shif m^\delta_t),\boldsymbol{\varrho}^{\delta}_t \right> \Bigr)dt + d \shif N^{\delta}_t,
\]
Since $t_2$ is arbitrary, we conclude that $\shif z^\delta$ indeed solves the above equation on the time interval $\R^+$. In view of the  regularity of $ \shif z^{\delta}$,  $\shif N^{\delta}$ has path in $C^0([0,\infty), C^{1})$.
\\

\textbf{Step 4:} We finally show  that $(\zDW,\rDW, \NDW)$ is the unique solution of \eqref{eq.linearized_MFG_shifted}. Passing to the limit in the equation satisfied by $\boldsymbol{\varrho}^{\delta,T}$ shows that  $\boldsymbol{\varrho}^{\delta}$ solves  
$$
  d\boldsymbol{\varrho}^\delta_t = \left(\Delta \boldsymbol{\varrho}^\delta_t +{\rm div}(\boldsymbol{\varrho}^\delta_t D  \shif u^\delta_t + \shif m^\delta_t D  \shif z^\delta_t)\right)dt , \qquad 
    \boldsymbol \varrho^\delta_{0} = \rho_0.
$$
Thus $(\zDW, \rDW, \NDW )$ satisfies \eqref{eq.linearized_MFG_shifted}.  Letting $T\to \infty$ in \eqref{kzjhejhezrbnldkj} implies that 
$$
\sup_{t\in [0,\infty)} e^{-\delta t} \E\left[\|D{\shif z}^{\delta}_t\|_{L^2_x}^2\right]+\sup_{t\in [1,\infty)}e^{-\delta t} \E\left[  \normUU{\rDW}{L^2_x}^2\right] \leq C \|\rho_{0}\|_{(C^{2+\beta})'}^2.
$$
Hence  \eqref{cond.unicite.zDW} holds. Note that 
$$
\|\mathcal Z^\delta(0, \cdot)\|_{C^0} \leq \int_0^\infty e^{-\delta s}\left( \|D\shif z^\delta_s\|_{C^0}+ \|\boldsymbol{\varrho}^{\delta}_t \|_{(C^{2+\beta})'}\right) ds.
$$
Hence, by \eqref{aezolijrsdfkgjh},  
$$
\E\left[ \|\shif z^\delta_0\|_{C^0}\right] \leq C \|\rho_0\|_{(C^{2+\beta})'}.
$$
Using again \eqref{aezolijrsdfkgjh} to estimate the higher order derivatives in space of $\shif z^\delta_0$, we conclude that \eqref{kjezhrbdfnl} holds. \\

%, \eqref{kzjhejhezrbnldkjBIS}, and the discussion above, we have, for any $\beta\in (0,\alpha)$,  
%\be\label{tmp.azert3}
%\begin{aligned}
%\sup_{t\geq\tau} \E\left[ e^{-\delta t} \sup_{t-\tau \leq s \leq t} \| {\shif z}^{\delta}_s\|_{C^{3+\beta}}^2 \right] + \sup_{t\geq\tau} &~\E \left[ e^{-\delta t} \sup_{t-\tau \leq s \leq t} \| {\shif N}^{\delta}_s\|_{C^{1+\beta}}^2 \right] \\
%&+ \E\left[ \sup_{t \geq 0} e^{-\delta t} \normUU{\rDW}{(C^{2+\beta})'}^2\right] \leq C_\beta  \|\rho_{0}\|_{(C^{2+\beta})'}^2
%\end{aligned}
%\ee
%and
%$$
%\sup_{t\in [0,\infty)} e^{-\delta t} \E\left[\|D{\shif z}^{\delta}_t\|_{L^2_x}^2\right]+\sup_{t\in [1,\infty)}e^{-\delta t} \E\left[  \normUU{\rDW}{L^2_x}^2\right] \leq C \|\rho_{0}\|_{(C^{2+\beta})'}^2.
%$$
%Hence  \eqref{cond.unicite.zDW} holds. \\
%Inequality \eqref{maejzjnretf} also holds because it holds for the $(\zDWT, \rDWT, \NDWT )$.  \\

Let us finally check that condition \eqref{cond.unicite.zDW} suffices to ensure the uniqueness of the solution. As the system is linear, we just need to prove that, if  $(\zDW, \rDW, \NDW )$ is a solution  to \eqref{eq.linearized_MFG_shifted} with $\rho_0=0$ and is such that \eqref{cond.unicite.zDW} holds, then $(\zDW, \rDW, \NDW )=0$. Using the monotonicity argument as above, we have, since $\rho_0=0$, and for any $t_2\geq 1$: 
\begin{align*}
\EP{\int_{0}^{t_2} e^{-\delta t} \int_{\T^d} \mDW |D  \shif z^\delta_t|^2 dx dt }   
& \leq - e^{-\delta t_2}\EP{\int_{\T^d} \shif z^\delta_{t_2} \boldsymbol{\varrho}^\delta_{t_2} } \\
& \leq 
e^{-\delta t_2}(\EP{\|D\shif z^\delta_{t_2}\|_{L^2_x}})^{1/2}  (\EP{\|\boldsymbol{\varrho}^\delta_{t_2}\|_{L^2_x}})^{1/2} \leq C,
\end{align*}
where we used condition \eqref{cond.unicite.zDW} for the last inequality. Hence
$$
\EP{\int_{0}^{\infty} e^{-\delta t} \int_{\T^d} \mDW |D  \shif z^\delta_t|^2 dx dt }  \leq C. 
$$
We can bootstrap this inequality as we did several times above to show that, in fact, 
$$
\EP{\int_{0}^{\infty} e^{-\delta t} \int_{\T^d} \mDW |D  \shif z^\delta_t|^2 dx dt }  =0. 
$$
The conclusion then follows easily.
\end{proof}

In the proof we have used the following variation on  \cite[Lemma 4.3.4]{CDLL}. We include it for the sake of completeness. 

\begin{lem}\label{lem.extension}
Let $\mathcal{U} : \Omega\times  [t_0, T] \times \mathbb{T}^d \to \mathbb{R}$ be a random field with continuous paths (in the variable $(t,x) \in [t_0,T] \times \mathbb{T}^d$). If, for some $k \geq 1$, the paths of $\mathcal{U} $ are $k$-times differentiable in the space variable, the derivatives up to the order $k$ having $(t,x)$-jointly continuous paths and satisfying
\[
\EP{ \| D^k \mathcal{U} \|_{{C}^0([t_0,T]\times\T^d)} } < \infty,
\]
then we can find a version of the random field $[t_0,T] \times \mathbb{T}^d \ni (t,x) \mapsto \mathbb{E}[\mathcal{U}(t,x)|\mathcal{F}_t]$ that is progressively measurable and that has paths in $C^0([t_0,T], {C}^k(\mathbb{T}^d))$, the derivative of order $k$ written $[t_0,T] \times \mathbb{T}^d \ni (t,x) \mapsto \mathbb{E}[D^k \mathcal{U}(t,x)|\mathcal{F}_t]$.
\end{lem}

Let us recall that  \cite[Lemma 4.3.4]{CDLL} states the result under the condition that $\| D^k \mathcal{U} \|_{{C}^0([t_0,T]\times\T^d)}$ is bounded in $L^\infty(\Omega)$. 

\begin{proof}
Denote by $w$ the pathwise modulus of continuity of $D^k \mathcal{U}$ on the compact set $[t_0,T] \times \mathbb{T}^d$. Namely we define $\omega$ by $\omega$ that,
\[
w(\delta) = \sup_{x,y \in \mathbb{T}^d: |x - y| \leq \delta} \sup_{s,t \in [t_0,T]: |s - t| \leq \delta} | D^k \mathcal{U}(s,x) - D^k \mathcal{U}(t,y)|, \quad \delta > 0.
\]
Hence we have, for any $\delta > 0$ that,
\[
w(\delta) < 2 \|D^k \mathcal{U}\|_{{C}^0([t_0,T]\times\T^d)},
\]
where $\EP{ \|D^k \mathcal{U}\|_{{C}^0([t_0,T]\times\T^d)} } < \infty$. By Doob's inequality, we have that, for any $p \geq 1$,
\[
\forall \varepsilon > 0, \quad \mathbb{P} \left( \sup_{s \in [t_0,T]} \EP{ w \left( \frac 1p \right) \bigg| \mathcal{F}_s } \geq \varepsilon \right) \leq \varepsilon^{-1} \EP{w \left( \frac 1p \right)}
\]
the right-hand side converging to 0 as $p$ tends to $\infty$, thanks to Lebesgue's dominated convergence theorem. We can then conclude as in \cite[Lemma 4.3.4]{CDLL}.
\end{proof}

%%%%%%%%%%%%%%%%%%%%%%%%%%%
%%%%%%%%%%%%%%%%%%%%%%%%%%%
%%%%%%%%%%%%%%%%%%%%%%%%%%%
%%%%%%%%%%%%%%%%%%%%%%%%%%
\section{The ergodic master equation} \label{sec.ergo}

In this section  we build a weak solution to the master cell problem:
\begin{align}\label{eq.ergomaster} 
& \lambda -(1+\cwn) \Delta \chi (x,m)+\frac12| D_x\chi(x,m)|^2 -(1+\cwn)\inte \dive_y(D_m\chi(x,m,y))m(dy) \notag\\ 
& \qquad \qquad + \inte D_m\chi(x,m,y)\cdot D_x\chi(y,m) m(dy)- 2\cwn \inte \dive_x(D_m\chi(x,m,y))m(dy)\notag \\ 
& \qquad \qquad  -\cwn \inte\inte {\rm Tr}(D^2_{mm} \chi(x,m,y,z))m(dy)m(dz) = f(x,m). 
\end{align}
The equation is stated in $\T^d\times \mathcal P(\T^d)$. Here the unknowns  are the constant $ \lambda$ and the map $\chi:\T^d\times \mathcal P\to \R$.  By a weak solution, we mean the following:

\begin{defn}\label{defweakmastercell}
We say that a pair \( (\lambda, \chi) \), where \( \lambda \in \mathbb{R} \) and \( \chi \colon \T^d \times \mathcal{P}(\T^d) \to \mathbb{R} \), is a \emph{weak solution} of the master cell problem~\eqref{eq.ergomaster} if the following conditions are satisfied:
\begin{itemize}
\item[(a)] The function \( \chi \) is \( C^3 \) in \( x \), and both \( \chi \) and \( D_x \chi \) are globally Lipschitz continuous in \( \T^d \times \mathcal{P}(\T^d) \), and monotone in the sense that
\[
\int_{\T^d} \left( \chi(x, m) - \chi(x, m') \right) \, d(m - m')(x) \geq 0 \qquad \text{for all } m, m' \in \mathcal{P}(\T^d);
\]
    
\item[(b)] For any initial condition \( \hat{m}_0 \in \mathcal{P}(\T^d) \) and any \( T > 0 \), let \( (\bar{\shif{u}}, \bar{\shif{m}}, \bar{\shif{M}}) \) denote the unique classical solution to the shifted MFG system:
\be\label{eq.MFGCellshif}
\left\{
\begin{array}{cl}
    (i) & d\bar{\shif{u}}_t = \left( \lambda - \Delta \bar{\shif{u}}_t + \tfrac{1}{2} |D \bar{\shif{u}}_t|^2 - \shif{f}_t(x, \bar{\shif{m}}_t) \right) dt + d\bar{\shif{M}}_t, \\ [1ex]
    (ii) & d\bar{\shif{m}}_t = \left( \Delta \bar{\shif{m}}_t + \operatorname{div}(\bar{\shif{m}}_t D \bar{\shif{u}}_t) \right) dt, \\ [1ex]
    (iii) & \bar{\shif{m}}_0 = \hat{m}_0, \qquad \bar{\shif{u}}_T(x) = \chi(x + \sqrt{2\cwn} W_T, \, (\operatorname{Id} + \sqrt{2\cwn} W_T)\sharp \bar{\shif{m}}_T).
\end{array}
\right.
\ee
Then the solution satisfies
\[
\bar{\shif{u}}_t(x) = \chi\left(x + \sqrt{2\cwn} W_t, \, (\operatorname{Id} + \sqrt{2\cwn} W_t)\sharp \bar{\shif{m}}_t\right) \quad \text{a.s., for all } (t,x) \in [0, T] \times \T^d.
\] 
\end{itemize}
\end{defn}

The existence and the uniqueness of the solution to \eqref{eq.MFGCellshif} is a consequence of \cite[Section 5.5]{CDLL}, see also the discussion in Subsection \ref{subsec.HJMFG}. 

\begin{rmk}\label{another_condition_weak_ergodic_master}
Recalling  Remark \ref{rem.prop.shifumVsumextended}, we can  replace condition (b) of Definition \ref{defweakmastercell} by the following condition (b'): 
\begin{itemize}
\item[(b')] For any initial condition \( \hat{m}_0 \in \mathcal{P}(\T^d) \) and any \( T > 0 \), let \( (\bar{{u}}, \bar{{m}}, \bar{{v}}) \) denote the unique weak solution to 
\begin{equation}\label{eq.MFGCell}
\left\{
\begin{array}{cl}
    (i) & d\bar{u}_t = \left( \lambda - (1+\cwn) \Delta \bar{u}_t + \tfrac{1}{2} |D \bar{u}_t|^2 - {f}_t(x, \bar{m}_t) - 2 \cwn \div(\bar v_t) \right) dt + \sqrt{2\cwn} \, \bar v_t \cdot d W_t, \\ [1ex]
    (ii) & d\bar{{m}}_t = \left( (1+\cwn) \Delta \bar{{m}}_t + \operatorname{div}(\bar{{m}}_t D \bar{{u}}_t) \right) dt - \sqrt{2\cwn} D \bar m_t \cdot d W_t, \\ [1ex]
    (iii) & \bar{{m}}_0 = \hat{m}_0, \qquad \bar{{u}}_T(x) = \chi(x , \, \bar{{m}}_T).
\end{array}
\right.
\end{equation}
Then this solution satisfies
\[
\bar{{u}}_t(x) = \chi\left(x, \, \bar{{m}}_t\right) \quad \text{a.s., for all } (t,x) \in [0, T] \times \T^d.
\] 
\end{itemize}
\end{rmk}

Here is the main result of the section: 

 \begin{thm}\label{lem.existence_weak_solution_cell_Master_Equation} Let $\bar \lambda$ be defined by \eqref{defbarlambda}. There is a map $\chi:  \T^d \times \mathcal{P}(\T^d) \to \mathbb{R}$ such that the pair $(\bar \lambda,\chi)$ is a weak solution to \eqref{eq.ergomaster}, in the sense of Definition \ref{defweakmastercell}. 
  \end{thm}
 
\begin{rmk} Let us note that the solution of \eqref{eq.ergomaster} is unique in the following sense: if $(\lambda', \chi')$ is another weak solution to \eqref{eq.ergomaster}, then $\lambda'=\bar \lambda$ and there is a constant $c$ such that $\chi'= \chi+c$. The equality $\lambda'=\bar \lambda$ is a simple consequence of Theorem \ref{thm.main} and the definition of weak solution of \eqref{eq.ergomaster}, while the uniqueness of $\chi$ (up to an additive constant) will be an outcome of Theorem \ref{thm.CVU} below. 
\end{rmk} 

In order to build a solution to the master cell problem, we  start with a solution to the discounted master equation (see \cite[Lemma 5.1.1]{CDLL} in the finite horizon case): 
\begin{align}\label{eq.discmaster} 
&\delta \chi^\delta(x,m)-(1+\cwn) \Delta \chi^\delta(x,m)+\frac12|D_x\chi^\delta(x,m)|^2 -(1+\cwn)\inte \dive_y(D_m\chi^\delta(x,m,y))m(dy) \notag\\ 
& \qquad \qquad + \inte D_m\chi^\delta(x,m,y)\cdot D_xU(t,y,m) m(dy)-\sqrt{2\cwn} \inte \dive_x(D_m\chi^\delta(x,m,y))m(dy)\notag \\ 
& \qquad \qquad  -\cwn \inte\inte {\rm Tr}(D^2_{mm} \chi^\delta(x,m,y,z))m(dy)m(dz) = f(x,m). 
\end{align}

The solution of this equation can be built as follows: Let $\hat m_0\in \mathcal P(\T^d)$ and let $(u^\delta, m^\delta, v^\delta)$ be the solution to \eqref{eq.MFGCNdisc} defined on the time interval $[0,\infty)$ and with initial condition $m^\delta_0=\hat m_0$. Let us set 
\be\label{chideltavsam}
\chi^\delta(x,\hat m_0) = u^\delta_0(x)
\ee
(Note that  $u^\delta_0$ is deterministic). Then following the construction of \cite{CDLL} (see also the discussion in \cite{CP19}), one can check that $\chi^\delta$ is a classical solution to \eqref{eq.discmaster}. We now discuss the regularity of $\chi^\delta$, uniform in $\delta$. 

\begin{lem}\label{em.estichidelta} We have, for any $\beta \in (0,1)$ and any $m\in \mathcal P(\T^d)$,
\be\label{kzljrkerkgfjn}
\normC*{\delta \chi^\delta(\,\cdot~,m)}+ 
\normDDH*{D_x \chi^\delta(\,\cdot~,m)}+
\normU{D_m\chi^\delta(\,\cdot~, m, ~\cdot\,)}{2+\beta,1+\beta}  \leq C
\ee
where $C$ depends on $f$ and $\beta$, but not on $m$ and $\delta$.
\end{lem}

\begin{proof}  Fix $m_0\in \mathcal P(\T^d)$ which is deterministic and let $(\shif u^\delta, \shif m^\delta, \shif M^\delta)$ solve \eqref{eq.MFGCNtildedelta}.
To stress the fact  that $\uDW$ depends on the initial condition $m_0$, we may also write $\uDW(x;m_0)$. Recalling \eqref{chideltavsam}, 
the first two estimates in \eqref{kzljrkerkgfjn} are equivalent to saying that
$$
\normC*{\delta \shif u^\delta_0(\,\cdot~; m_0 )} + 
\left\|D \shif u^\delta_0(\,\cdot~;m_0)\right\|_{C^{2+\beta}} \leq C.
$$
By \eqref{esti1111} in Lemma \ref{lem.boundsDisc}, we know that $\|\delta \shif u^\delta\|_\infty+ \|D \shif u^\delta\|_\infty \leq C$ a.s. Let us improve the second inequality. For $t\geq 0$, set $\tilde{\shif u}^\delta_t= {\shif u}^\delta_t-\inte {\shif u}^\delta_t$. Then $\tilde {\shif u}^\delta_t$ is given by 
$$
\tilde {\shif u}^\delta_t(x)= \E \left[a_t(x)\ |\ \mathcal F_{t}\right], 
$$
where, for $s\in [t,t+1]$,  
$$
a_s(x)= \Gamma_{t+1-s}\ast \tilde{\shif u}^\delta_{t+1}(x) + \int_s^{t+1} \Gamma_{u-t}\ast \tilde h_u(x)du.
$$
Here $\Gamma$ is the heat kernel and 
$$
h_s(x) = \delta \shif u^\delta_s(x)+ \frac12 |D\shif u^\delta_s(x)|^2-\shif f_s(x,\shif m^\delta_s(x)), \qquad \tilde h_s(x)= h_s(x)-\inte h_s(x).
$$
Note that $a$ solves the backward equation (with random coefficient, but solved $\omega$ by $\omega$ and not adapted) 
$$
- \partial_t a_s(x) -\Delta a_s(x) +h_s(x)=0\; \text{in}\; [t,t+1]\times \T^d, \qquad a_{t+1}(x)= \tilde{\shif u}^\delta_{t+1}(x). 
$$
Arguing as in the proof of Lemma \ref{lem.Wmu}, one can check that $a_t$ is bounded in $C^{3+\beta}(\T^d)$ uniformly in $t$ and $\delta$, for any $\beta\in (0,1)$. This in turn implies that $\tilde {\shif u}_t$ is uniformly bounded in $C^{3+\beta}(\T^d)$ for $t\geq 0$. This shows that $\|D\tilde {\shif u}_t\|_{C^{2+\beta}}\leq C$. \\

In order to show the last inequality in \eqref{kzljrkerkgfjn}, let us use the representation formula \cite[Proposition 5.2.4]{CDLL} in a closely related context: for any map $\rho_0 \in \left( C^{2+\beta} \right)'$, with $\lg \rho_0, {\bf 1}_{\T^d}\rg = 0$, we have 
\be\label{rep.deltaUdelta}
\inte \frac{\delta \chi^\delta}{\delta m}(x,m_0,y)\rho_0(dy) = \shif z^\delta_0(x) ,
\ee
where $(\shif z^\delta,\boldsymbol{\varrho}^\delta,\shif N^\delta)$ is the unique solution to the  linearized system 
\begin{align}\label{eq.MFGihLS}
& \ds d \zDW = \Bigl(\delta \zDW -\Delta \zDW + D  \uDW \cdot D  \zDW - \left<\frac{\delta \shif f_t}{\delta m}(\cdot, \mDW),\rDW \right> \Bigr)dt + d \NDW \qquad {\rm in}\; (0,+\infty)\times \T^d,\notag \\
&d\rDW=\left(\Delta \rDW +{\rm div}(\rDW D  \uDW + \mDW D  \zDW)\right)dt \qquad {\rm in}\; (0,+\infty)\times \T^d, \\
&  \varrho_0^\delta= \rho_0\; {\rm in}\;  \T^d,  \notag
\end{align}
satisfying condition \eqref{cond.unicite.zDW}. 
In order to stress that the solution of $\zDW$  depends on the initial value $\rho_0$, we again write $\zDW(x; \rho_0)$. As in \cite[Proposition 5.2.4]{CDLL}, we have that for $|l|, |k| \leq 2$,
\[
\partial_x^l \partial_y^k\frac{\delta \chi^\delta}{\delta m}(x,m_0,y) = \partial_x^l \shif z^\delta_0(x; \partial^k\delta_y) \ .
\]
Therefore the estimate of the third term in \eqref{kzljrkerkgfjn} is a consequence of the following estimate:
\be\label{es.linearized_MFG_delta_shift}
\normDDH{ \shif z^\delta_0(~\cdot~; \rho_0)} \leq C \normDDHd{\rho_0},
\ee
which itself is nothing but the statement of Theorem \ref{thm.good_bound_for_linearized_system_discounted}. 
\end{proof}

% 
% 
% \begin{thm}\label{lem.existence_weak_solution_cell_Master_Equation} Let $\bar \lambda$ be defined by \eqref{defbarlambda}. Then the pair $(\bar \lambda,\chi)$ is a weak solution to \eqref{eq.ergomaster}, in the sense of Definition \ref{defweakmastercell}. 
% \end{thm}
 
\begin{proof}[Proof of Theorem \ref{lem.existence_weak_solution_cell_Master_Equation}] In view of the estimates \eqref{kzljrkerkgfjn} of Lemma \ref{em.estichidelta}, we can extract a subsequence $\delta_n\to 0$ (still denoted $\delta$) such that $\chi^\delta-\chi^\delta(0, \delta_0)$ and $D_x\chi^\delta$  converges uniformly to some maps $\chi=\chi(x,m)$ and $D_x\chi=D_x\chi(x,m)$ which are respectively $C^3$ and $C^2$ in $x$ and both Lipschitz in $\T^d \times \mathcal P(\T^d)$. We now check that the pair $(\bar \lambda,\chi)$ is a weak solution to \eqref{eq.ergomaster}. 

One can check easily that the $\chi^\delta$, and thus also $\chi$, are monotone (as in the proof of \cite[Lemma 4.6]{CP19}). The regularity of $\chi$ is a consequence of the regularity of the $(\chi^\delta)$. Fix $\hat m_0\in \mathcal P(\T^d)$.  Now let $(\shif u^\delta, \shif m^\delta, \shif M^\delta)$ be the solution to the shifted discounted problem  \eqref{eq.MFGCNtildedelta} on the time interval $[0,\infty)$ and with initial condition $\hat m_0$ and let $(\shif u, \shif m, \shif M)$ be the solution to \eqref{eq.MFGCellshif} with initial condition $\hat m_0$  at time $0$. By the usual Lasry-Lions argument, we have 
\begin{align*}
& \E\left[ \int_{\T^d}(\shif u^\delta_{T}-\shif u_{T})(\shif m^\delta_{T}-\shif m_{T})\right]\\ 
& \qquad  \leq - \frac12 \E\left[ \int_{0}^{T} \int_{\T^d} |D  \shif u^\delta- D  \shif u|^2 (\shif m^\delta+\shif m)\right] +
\E\left[ \int_0^T\inte  (\delta \shif u^\delta -  \bar \lambda)(\shif m^\delta-\shif m)\right]. 
\end{align*} 
We note that, by \eqref{chideltavsam}, and denoting by $\tau_T$ the random map $\tau_T(x)= x+\sqrt{2\cwn}W_T$, 
\begin{align*}
&\E\left[ \int_{\T^d}(\shif u^\delta_{T}-\shif u_{T})(\shif m^\delta_{T}-\shif m_{T})\right]\\
= &\E\left[ \int_{\T^d}( \chi^\delta(x+\sqrt{2\cwn}W_T, \tau_T\sharp \shif m^\delta_T)-\chi(x+\sqrt{2\cwn} W_T, \tau_T\sharp \shif m_T))(\shif m^\delta_{T}-\shif m_{T})\right]\\ 
= &\E\left[ \int_{\T^d}( \chi^\delta(x, \tau_T\sharp \shif m^\delta_T)-\chi(x, \tau_T\sharp \shif m_T))(\tau_T\sharp\shif m^\delta_{T}-\tau_T\sharp\shif m_{T})\right]\\
\geq & \E\left[ \int_{\T^d}( \chi^\delta(x, \tau_T\sharp \shif m^\delta_T)-\chi(x, \tau_T\sharp \shif m^\delta_T))(\tau_T\sharp\shif m^\delta_{T}-\tau_T\sharp\shif m_{T})\right] \qquad \text{(by monotonicity of $\chi$)}\\ 
= &\E\left[ \int_{\T^d}( \chi^\delta(x, \tau_T\sharp \shif m^\delta_T)-\chi^\delta(0,\delta_0)-\chi(x, \tau_T\sharp \shif m^\delta_T))(\tau_T\sharp\shif m^\delta_{T}-\tau_T\sharp\shif m_{T})\right] \\ 
\geq &-2\| \chi^\delta-\chi^\delta(0,\delta_0)-\chi\|_\infty. 
\end{align*}
On the other hand, by \eqref{ineqcontracdelta3} in Theorem \ref{thm.Cudelta}, we have 
$$
\left| \E\left[ \int_0^T\inte  (\delta \shif u^\delta -  \bar \lambda)(\shif m^\delta-\shif m)\right] \right| \leq C\delta. 
$$
This implies that for $\delta_n \to 0$ such that $\chi^{\delta_n} - \chi^{\delta_n}(0,\delta_0) \to \chi$, we have,
$$
\lim_{n \to \infty}  \E\left[ \int_{0}^{T} \int_{\T^d} |D  \shif u^{\delta_n}- D  \shif u|^2 (\shif m^{\delta_n}+\shif m)\right] =0. 
$$
As, for any $t_0>0$, there exists $C_{t_0}>0$ such that $\shif m^{\delta_n}_t\geq C_{t_0}^{-1}$ and  $\shif m_{t_0}\geq C_{t_0}^{-1}$ for $t\in [t_0,T]$, we infer that 
$$
\lim_{n\to \infty} \E\left[ \int_{t_0}^{T} \int_{\T^d} |D\shif u^{\delta_n} -D\shif u|^2\right] =0 . 
$$
As a consequence, using the equations for $\shif m^{\delta_n}$ and $\shif m$ implies that
\be\label{kjsdfkngnkjn}
\lim_{n \to \infty} \E\left[{\bf d}_1 (\shif m^{\delta_n}_t, \shif m_t)\right]=0 \qquad \forall t\in  [0,T]. 
\ee
We now apply the comparison principle between $\shif u^\delta- \chi^\delta(0, \delta_0)$ and $\shif u$, which both satisfy a Hamilton-Jacobi equation on $[0,T]$ with a slightly different right-hand side and sightly different terminal condition. We have 
\begin{align*}
& \E\left[ \sup_{t,x}| \shif u^{\delta_n}_t(x) -  \chi^{\delta_n}(0, \delta_0)-\shif u_t(x)|\right]\\ 
& \qquad  \leq  \E\left[  \int_0^T \sup_x |-\shif f_t(x,\shif m^{\delta_n}_t)+\shif f_t(x,\shif m_t)+  \delta_n (\shif u^{\delta_n}(x)- \chi^{\delta_n}(0, \delta_0)) -\bar \lambda|\right] \\ 
& \qquad\qquad +  \E\left[  \sup_x | \chi^{\delta_n}(x, (Id+\sqrt{2\cwn}W_T)\sharp \shif m^\delta_T)- \chi^\delta(0, \delta_0)-\chi(x, (Id+\sqrt{2\cwn}W_T)\sharp  \shif m_T) | \right]
\end{align*}
which, in view of  \eqref{ineqcontracdelta3} in Theorem \ref{thm.Cudelta}, \eqref{kjsdfkngnkjn} and the uniform convergence of $\chi^\delta - \chi^\delta(0, \delta_0)$ to $\chi$, proves that the right-hand side of the above inequality tends to $0$. On the other hand, we know that 
$$
 \shif u^\delta_t(x)= \chi^\delta(x+\sqrt{2\cwn}W_t, (Id+\sqrt{2\cwn}W_t)\sharp m^\delta_t)),
 $$
 so that, letting $\delta\to 0$ shows that 
 $$
 \shif u_t(x) = 
 \chi(x+\sqrt{2\cwn}W_t, (Id+\sqrt{2\cwn}W_t)\sharp\shif m_t) \qquad \text{a.s.}. 
 $$
Thus  $\chi$ is a weak solution to \eqref{eq.ergomaster}. 
\end{proof}

We now explain  the relationship between $(\bar{\mathcal{V}},\bar \mu)$ and $\chi$.

\begin{thm}\label{prop.fullcorrector} Let $(\bar{\mathcal V}, \bar \mu)$ be the stationary process defined in Lemma \ref{lem.Wmu}. Then $\bar \mu_t$ solves
\be\label{eq.barmuchi}
d\bar \mu_t= ((1+\cwn) \Delta \bar \mu_t + \dive(\bar \mu_t D  \chi(x, \bar \mu_t))dt - \sqrt{2\cwn} D  \bar \mu_t\cdot dW_t, \qquad t\in \R.
\ee
Moreover, $D_x\chi(x, \bar \mu_t) = \bar{\mathcal V}_t(x)$. Finally, if we set $\bar u_t(x)= \chi(t, x, \bar \mu_t)$, we can find a process $\bar v$ with paths in $H^2(\T^d)$ such that the stationary process $(\bar u,\bar \mu, \bar v)$ solves:
\be\label{eq.MFGstable}
\left\{
\begin{aligned}
& (i)\; d\uS= \big( \bar \lambda - (1+\cwn)\Delta \uS +\frac12 |D \uS|^2 -f(x,\bar \mu_t) -2\cwn{\rm div}(\vS) \big)dt +\sqrt{2\cwn} \vS\cdot dW_t \\
& (ii)\; d\bar \mu_t=\left((1+\cwn)\Delta \bar \mu_t +{\rm div}(\bar \mu_t D  \uS)\right)dt -\sqrt{2\cwn} D \bar \mu_t \cdot dW_t
\end{aligned} 
\right.
\ee
where, a.s., the first equation holds in a strong sense and the second equation in the sense of distribution.
\end{thm}

\begin{proof} Let us first prove that that $\bar \mu$ solves \eqref{eq.barmuchi}. By the definition of the stationary process $(\bar{\mathcal V}^\delta, \bar \mu^\delta)$, for any $T\in \R$, the unique solution $(w^\delta, m^\delta, z^\delta)$ to \eqref{eq.MFGCNdisc} on the time interval $(T, \infty)$ with initial condition $m_{T}= \bar \mu^\delta_T$ satisfies $Dw^\delta_t= \bar{\mathcal V}^\delta_t$ and $m^\delta_t= \bar \mu^\delta_t$ for any $t\in [T,\infty)$. As $\chi^\delta$ is a classical solution to the master equation \eqref{eq.discmaster} we know from a similar argument as in \cite[Lemma 5.1.1]{CDLL} for the time dependent master equation (see also \cite{CP19}) that, on $[T, \infty)\times \T^d$,  
$$
d\bar \mu^\delta_t= ((1+\cwn) \Delta \bar \mu^\delta_t + \dive(\bar \mu^\delta_t D_x  \chi^\delta(\cdot, \bar \mu^\delta_t))dt - \sqrt{2\cwn} D  \bar \mu^\delta_t\cdot dW_t
$$
and that $w^\delta_t(x) = \chi^\delta(x, \bar \mu^\delta_t),$
and thus 
$
Dw^\delta_t(x) = D_x\chi^\delta(x, \bar \mu^\delta_t).
$
We  now pass to the limit as $\delta\to 0$. As we know from Theorem \ref{thm.Cudelta} that $\bar \mu^\delta$ converges to $\bar \mu$ in $L^2([T,T'], L^2_{\omega, x})$ (for any $T'>T$) and as $D_x \chi^\delta$ converges uniformly to $D_x\chi$, we derive that $\bar \mu$ solves \eqref{eq.barmuchi} on the whole time interval $\R$ as $T$ is arbitrary.  Passing to the limit in the equality $\bar{\mathcal V}^\delta_t=Dw^\delta_t=D_x\chi^\delta(\cdot, \bar \mu^\delta_t)$, we also infer from Theorem \ref{thm.Cudelta} that  $\bar{\mathcal V}_t=D_x\chi(\cdot, \bar \mu_t)$ for any $t\in \R$. \\

Fix $T>0$ large and let $(v,z)$ be the unique strong solution on $(-T,T)$ of the HJ equation 
\be\label{eq.almost_stable2}
\left\{
\begin{aligned}
    d w_t & = \left( \bar \lambda - (1+\cwn) \Delta w_t +\frac12 |D w_t|^2 - f(\cdot, \bar{\mu}_t) - 2\cwn \div(z_t) \right)dt + \sqrt{2\cwn} z_t dW_t, \\
    w_T &  = \chi(\cdot, \bar{\mu}_T),
\end{aligned}
\right.
\ee
Our aim is to check that $w_t=\chi(\cdot, \bar \mu_t)$ for any $t\in [-T,T]$. In order to prove that claim, let $T_1>T+2$ and  \( (\bar{{u}}^{T_1}, \bar{{m}}^{T_1}, \bar{{v}}^{T_1}) \) denote the unique weak solution on $[-T_1,T_1]$ to the MFG system
$$
\left\{
\begin{array}{cl}
    (i) & d\bar{u}^{T_1}_t = \left( \lambda - (1+\cwn) \Delta \bar{u}^{T_1}_t + \tfrac{1}{2} |D \bar{u}^{T_1}_t|^2 - {f}_t(x, \bar{m}^{T_1}_t) - 2 \cwn \div(\bar v^{T_1}_t) \right) dt + \sqrt{2\cwn} \, \bar v^{T_1}_t \cdot d W_t, \\ [1ex]
    (ii) & d\bar{{m}}^{T_1}_t = \left( (1+\cwn) \Delta \bar{{m}}^{T_1}_t + \operatorname{div}(\bar{{m}}^{T_1}_t D \bar{{u}}^{T_1}_t) \right) dt - \sqrt{2\cwn} D \bar m^{T_1}_t \cdot d W_t, \\ [1ex]
    (iii) & \bar{{m}}^{T_1}_{-T_1} = \bar \mu_{-T}, \qquad \bar{{u}}^{T_1}_{T_1}(x) = \chi(x , \, \bar{{m}}_{T_1}).
\end{array}
\right.
$$
As $\chi$ is a weak solution to the ergodic master equation \eqref{eq.ergomaster}, we have
\[
\bar{{u}}^{T_1}_t(x) = \chi\left(x, \, \bar{{m}}^{T_1}_t\right) \quad \text{a.s., for all } (t,x) \in [-T_1, T_1] \times \T^d.
\] 
% PB: $\bar \mu_{-T}$ PAS CONSTANT, DONC NE COLLE PAS AVEC LA DEF. 
On the other hand,  Theorem \ref{thm.main} states that 
$$
\E\left[ \|\bar m^{T_1}_t-\bar \mu_t\|_{C^\beta}\right] \leq C( e^{\omega(t+T_1)}+e^{\omega(T_1-t)})\qquad \forall t\in [-T_1+2,T_1-2].
$$
Therefore we have, for any $t\in [-T,T]$,  
\begin{align*}
\E \bigg[  \| \chi\left(x, \, \bar \mu_t\right)-w_t\|_{C^0}  \bigg]  & \leq C \EP{ {\bf d}_1 (\bar \mu_t, \bar m^{T_1}_t) }+ 
\E \bigg[  \| \chi\left(x, \, \bar \mu^{T_1}_t\right)-w_t\|_{C^0}  \bigg]  \\
&\leq  C( e^{\omega(t+T_1)}+e^{\omega(T_1-t)})+ \E \bigg[  \| \bar u^{T_1}_t-w_t\|_{C^0}  \bigg] 
\end{align*}
where, by comparison on the solutions of HJ equation (see Proposition \ref{prop.HJ}),  
\begin{align*}
\E \bigg[  \| \bar u^{T_1}_t-w_t\|_{C^0}  \bigg] 
%
%
% \E\left[  \|\bar m^{T_1}_t-\bar \mu_t\|_{C^\beta}\right] \\
%& \leq \E \bigg[  \| \chi\left(x, \, \bar{{m}}^{T_1}_t\right)-w_t\|_\infty  \bigg]  &= \E \bigg[ \sup_{t\in [-T,T]} \| \bar u^{T_1}_t-w_t\|_\infty  \bigg] \\ 
& \leq \EP{ \normC{\chi(\cdot,\bar \mu_T) - \chi(\cdot, \bar m^{T_1}_T) } } + \EP{ \int_{t_0}^T \normC{ f(\cdot,\bar \mu_s) - f(\cdot, \bar m^{T_1}_s) } ds }\\
& \leq C \, \EP{ {\bf d}_1 (\bar \mu_T, \bar m^{T_1}_T) } + C \, \EP{ \int_{t_0}^T {\bf d}_1 (\bar \mu_s, \bar m^{T_1}_s) \, ds } \\
& \leq C(e^{-\omega (T+T_1)}+e^{-\omega (T_1-T)}) +C \int_t^T (e^{-\omega (s+T_1)}+ e^{-\omega (T_1-s)})ds\notag \\
& \leq C e^{-\omega(T_1-T)} .
%\to 0 \quad \text{(as $T_1\to\infty$)}. 
\end{align*}
Letting $T_1\to\infty$ proves that $w_t(x)= \chi\left(x, \, \bar \mu_t\right)$ for any $(t,x)\in [-T,T]\times \T^d$ a.s.. \\ 

Let now $T'>T>0$ and $(v^T,z^T)$ and $(v^{T'},z^{T'})$ be respectively the unique strong solution of the HJ equation \eqref{eq.almost_stable2} on the time intervals $[-T,T]$ and $[-T',T']$ respectively, and with terminal condition $w^T_T = \chi(\cdot, \bar{\mu}_T)$ and $w^{T'}_{T'}  = \chi(\cdot, \bar{\mu}_{T'})$ respectively. We have just proved that $v^{T'}_T=\chi(\cdot, \bar{\mu}_T)$. Thus $(v^T,z^T)$ and $(v^{T'},z^{T'})$ solves the same HJ on $[-T,T]$ and therefore coincide on that interval. This shows the existence of a process $(\bar u,\bar v)$ defined on the whole time interval $\R$ such that $(\bar u,\bar v)=(u^T,v^T)$ on $[-T,T]$. 

Note that $\bar u_t= \chi(\cdot, \bar{\mu}_t)$ and therefore $\bar u$ is a stationary process since $\bar \mu$ is stationary. Using the uniqueness of $z^T$ also implies that $\bar v$ is stationary (see the argument in Lemma \ref{lem.Wmu} for instance). Finally, as $D\bar u_t= D_x\chi (\cdot, \bar{\mu}_t)=\bar{\mathcal V}_t$, we infer that the triple $(\bar u, \bar \mu, \bar z)$ is an entire solution to \eqref{eq.MFGstable}. 

\end{proof}

%%%%%%%%%%%%%%%%%%%%%%
%%%%%%%%%%%%%%%%%%%%%%

\section{The long-time behavior of the master equation}\label{sec.longtimemaster}

We finally consider the master equation on $\R^-\times\T^d\times\mathcal{P}(\T^d)$,
{\small
\begin{align}\label{eq.finitemaster} 
&- \partial_t U(t,x,m)-(1+\cwn) \Delta U(t,x,m)+\frac12|D  U(t,x,m)|^2 - f(x,m) \notag \\
& -(1+\cwn)\inte \dive_y(D_m U(t,x,m,y))m(dy) + \inte D_m U(t,x,m,y) \cdot D  U(t,y,m) m(dy)\\ 
&  - 2\cwn \inte \dive_x(D_m U(t,x,m,y))m(dy) -\cwn \inte\inte {\rm Tr}(D^2_{mm} U(t,x,m,y,z))m(dy)m(dz) = 0 \notag,
\end{align}}
with terminal condition,
\be\label{termcondmasterT}
U(0,x,m) = g(x,m).
\ee

A solution (or classical solution) of \eqref{eq.finitemaster}-\eqref{termcondmasterT} is a continuous and bounded map $U:(-\infty, 0]\times \T^d\times\mathcal{P}(\T^d)$, such that all derivatives involved in equation \eqref{eq.finitemaster} exist, are continuous and bounded, which satisfies the equation and the terminal condition for any $(t,x,m)\in  (-\infty, 0]\times \T^d\times\mathcal{P}(\T^d)$. Recall that the existence and the uniqueness of this solution is proved in \cite{CDLL}. 

The main result of this part is the following: 

\begin{thm}\label{thm.CVU} Under our standing assumption, and if $\chi$ is a weak solution to the ergodic master equation \eqref{eq.ergomaster}, then there exists a constant $c\in \R$ such that,
\be\label{cv.Uinfty}
\lim_{T\to \infty} U(-T, x,m)  -\bar \lambda T = \chi(x,m) +c,
\ee
uniformly with respect to $(x,m)\in \T^d\times \mathcal P(\T^d)$. Moreover, for any $t_0>0$, $D_xU(t-T, x,m)$ converges to $D_x\chi(x,m)$ uniformly with respect to $(t,x,m)\in[0,t_0]\times\T^d\times \mathcal{P}(\T^d)$ as $T\to \infty$.
\end{thm}

\begin{cor} \label{cor.uT-lambda}
Let $c$ be the constant given in Theorem \ref{thm.CVU}. For $T>0$ and $m_0 \in \mathcal{P}(\T^d)$, let $(u^T, m^T, v^T)$ be the solution of \eqref{eq.MFGCN}-\eqref{eq.MFGCNbc} with $t_0=0$. Then for any $t_1 \geq 0$, the following convergence holds:  
$$
\Big[ \uT(x) -\bar \lambda (T-t) \Big] - \Big[ \chi( x, \bar m_t) + c \Big] \xrightarrow[T\to\infty]{} 0 \qquad \text{a.s.},
$$
uniformly in $(t,x)\in[0,t_1]\times\T^d$ for any $t_1>0$, where $\bar m_t$ solves,
$$
d\mS = \left((1+\cwn)\Delta \mS +{\rm div}(\mS D_x \chi( x, \mS))\right)dt -\sqrt{2\cwn} D \mS \cdot dW_t,\qquad \bar m_0= m_0.
$$
Moreover, for any $\theta\in(0,1)$,
\[
\Big[ u^T_{\theta T}(x) - (1-\theta) \bar \lambda T \Big] - \Big[ \chi( x, \bar \mu_{\theta T}) + c \Big] \xrightarrow[T\to\infty]{} 0 \qquad \text{in $L^2_\omega$},
\]
uniformly for $x\in\T^d$, where $\bar \mu$ is the stationary solution in Lemma \ref{lem.Wmu}.
\end{cor} 

\begin{rmk}
In particular, Corollary \ref{cor.uT-lambda} implies that,
\[
\lim_{T\to\infty}u^T_0(x) -\bar \lambda T \to \chi( x, m_0) + c.
\]

Note also that, for any fixed $t>0$ and $ x \in \T^d$, the long-time average of the solution holds a.s.: 
\[
\lim_{T\to\infty} \frac{\uT(x)}{T} = \bar \lambda \qquad \text{a.s.}.
\]
\end{rmk}

The proof of Corollary \ref{cor.uT-lambda} is given at the end of the section.  \\

We now start the proof of Theorem \ref{thm.CVU}. We will make a systematic use of the following representation formula \cite[Lemma 5.1.1 \& Proposition 5.2.4]{CDLL}:
\be\label{repformulaU}
U(t-T,x + \sqrt{2\cwn} W_t , \mT) = U(t-T,x + \sqrt{2\cwn} W_t , (\operatorname{Id}+\sqrt{2\cwn}W_t)\sharp \mTW) = \uTW(x),
\ee
and
\be
\partial_x^l \partial_y^k \frac{\delta U}{\delta m}(-T,x,m_0,y) = \partial_x^l \shif z^T_0(x; \partial^k\delta_y) ,
\ee
where $(\uTW,\mTW,\MTW)$ is the solution of the  shifted MFG system,
\be\label{eq.shifshif}
\left\{
\begin{array}{ll}
    d \uTW = \left(-\Delta \uTW +\frac12 |D  \uTW|^2 -\shif f_t(x,\mTW) \right)dt +d\MTW & {\rm in}\; (0,T)\times \T^d, \\[2ex]
    d \mTW = \left(\Delta \mTW +{\rm div}(\mTW D  \uTW)\right)dt & {\rm in}\; (0,T)\times \T^d, \\[2ex]
    \shif m_{0}^T= m_0, \qquad \shif u^T_{T}(x)=\shif g_T(x,\shif m^T_T) & {\rm in}\;  \T^d,
\end{array}
\right.
\ee
(where as usual $\shif f$ and $\shif g$ are defined in \eqref{def.shiffg}) and $(\zTW,\rTW,\NTW)$ is the solution of the following linearized MFG system,
{\small
\be%\label{eq.linearized_MFG_shifted}
\left\{
\begin{array}{ll}
    \ds d \zTW = \Bigl( -\Delta \zTW + D  \uTW \cdot D  \zTW - \left<\frac{\delta \shif f_t}{\delta m}(\cdot, \mTW),\rTW \right> \Bigr)dt + d \NTW & {\rm in}\; (0,T)\times \T^d, \\[2ex]
    d\rTW=\left(\Delta \rTW +{\rm div}(\rTW D  \uTW + \mTW D  \zTW)\right)dt & {\rm in}\; (0,T)\times \T^d, \\[2ex]
    \boldsymbol \varrho_0^T= \partial^k\delta_y, \qquad \shif z^T_T(x) = \left<\frac{\delta \shif g_T}{\delta m}(\cdot, \shif m_T^T),\boldsymbol{\varrho}^T_T \right> & {\rm in}\;  \T^d.
\end{array}
\right.
\ee}
 We also give the definition of the weak solution of the master problem \eqref{eq.finitemaster} on the whole time interval. 

\begin{defn}\label{defweakfinitemaster}
We say that a function \( V \colon \R \times \T^d \times \mathcal{P}(\T^d) \to \mathbb{R} \) is a \emph{weak solution} of the master equation~\eqref{eq.finitemaster} if the following conditions are satisfied:
\begin{itemize}
    \item[(i)] The function \( V \) is \( C^3 \) in \( x \), and both \( V \) and \( D_x V \) are globally Lipschitz continuous in \( \T^d \times \mathcal{P}(\T^d) \), uniformly in \( t \in \mathbb{R} \). Moreover, \( V \) is monotone in the sense that
    \[
    \int_{\T^d} \left( V(t, x, m) - V(t, x, m') \right) \, d(m - m')(x) \geq 0 \qquad \text{for all } m, m' \in \mathcal{P}(\T^d), \text{ and all } t \in\R;
    \]

    \item[(ii)] For any $t_0<T$ and initial condition \( \hat{m}_0 \in \mathcal{P}(\T^d) \) at time $t_0$, let \( (\bar{\shif{u}}, \bar{\shif{m}}, \bar{\shif{M}}) \) denote the unique classical solution to the shifted system~\eqref{eq.shifshif} with initial condition $\hat m_0$ at time $t_0$ and terminal condition
    \[
    \bar{\shif{u}}_T(x) = V\left(T,\, x + \sqrt{2\cwn} (W_T-W_{t_0}),\, (\operatorname{Id} + \sqrt{2\cwn} (W_T-W_{t_0})) \sharp \bar{\shif{m}}_T \right).
    \]
   Then the solution satisfies a.s., for all $(t,x) \in [t_0, T] \times \T^d$, 
    \[
    \bar{\shif{u}}_t(x) = V\left(t,\, x + \sqrt{2\cwn} (W_t-W_{t_0}),\, (\operatorname{Id} + \sqrt{2\cwn} (W_t-W_{t_0})) \sharp \bar{\shif{m}}_t \right) .
    \]
\end{itemize}
\end{defn}

\begin{rmk} As in Remark \ref{another_condition_weak_ergodic_master}, we can equivalently replace condition (ii) by a symmetric condition on the original (i.e., non shifted) MFG system. 
\end{rmk}

The proof of the long-time behavior of $U(-T, \cdot, \cdot)$ as $T\to-\infty$ requires several steps. The first one is a  uniform regularity result in space and measure for $U$: 
\begin{lem}\label{lem.uniform_bound_finitemaster}
Let $U$ be the classical solution to the master equation \eqref{eq.finitemaster} with terminal condition \eqref{termcondmasterT}. Then  there exists a constant $C>0$ independent of time such that,
\be\label{es.uniform_bound_finitemaster_x_m}
\sup_{t \leq 0, m \in \mathcal{P}(\T^d) } \|D_xU(t,\cdot, m)\|_{C^{2+\beta}} + \normU{D_m U (t, \cdot, m, \cdot)}{C^{2+\beta,1+\beta}} \leq C.
\ee
%{\color{red} WE NEED ESTIMATES ON $D^2_{mm}U$ FOR THE REGULARITY IN TIME}. 
\end{lem}

\begin{proof} The  regularity in space of $D_xU$ can be obtained by the representation formula \eqref{repformulaU}, Lemma \ref{lem.regularity} (for the $L^\infty$ bound)  and the argument in the proof of Lemma \ref{lem.Wmu} (for higher regularity). This allows to bound  the first term in \eqref{es.uniform_bound_finitemaster_x_m} (note that, at this stage, we do not control the growth of $U$). We now turn to the regularity of $U$ in the measure variable. The argument is the same as in Lemma \ref{em.estichidelta}, except that we use Theorem \ref{thm.good_bound_for_linearized_system_finite} instead of Theorem \ref{thm.good_bound_for_linearized_system_discounted}. Thus the second term in \eqref{es.uniform_bound_finitemaster_x_m}  is bounded as well.
\end{proof}

The next step is a crucial bound on the difference between $U$ and $\bar \lambda t$: 
\begin{lem}\label{lem.7.6} Under the assumption of Lemma \ref{lem.uniform_bound_finitemaster}, there is a constant $c\in \R$ such that  
\be\label{limintU}
\lim_{t\to -\infty} \E\left[ \inte U(t,x, \bar \mu_0) dx\right] + \bar \lambda t = c. 
\ee
In particular, there is a constant $C>0$ such that, for any $t\leq0$, 
$$
\sup_{(x,m)\in (\T^d\times \mathcal P(\T^d))}\left| U(t,x,m)+\bar \lambda t\right| \leq C.
$$
\end{lem}

\begin{proof}
Let us come back to Lemma \ref{lem.keyCV}, which, combined with the representation formula \eqref{repformulaU}, states that there is a real constant $\bar c$ such that 
\be\label{kdjfnvlkjn}
\lim_{t\to -\infty} U(t,\cdot, \bar \mu_t) - \bar u_t(\cdot) = \bar c\qquad \text{in}\; L^1(\Omega), 
\ee
the convergence being uniformly in $x$. Here  $(\bar u, \bar \mu, \bar v)$ is the solution to \eqref{eq.MFGCN} on the time interval $(-\infty, 0)$ with terminal condition at time $t=0$ given by $\bar u_0$, where $\bar u_0$ is the $\mathcal F_0-$random variable such that $D\bar u_0=\bar{\mathcal V}_0$, $\bar u_0(0)=0$ (recall that the stationary solution $(\bar{\mathcal V}, \bar \mu)$ is defined in Lemma \ref{lem.Wmu}). 

By the equation satisfied by $\bar u$, we have, for any $t<0$,  
\begin{align*}
\E\left[ \inte \bar u_0-\inte \bar u_t\right]  = \int_t^0 \E\left[ \inte (\frac12 | \bar{\mathcal V}_s|^2-f(x, \bar \mu_s))\right] ds  =   \bar \lambda t, 
\end{align*}
as $(\bar{\mathcal V}, \bar \mu)$ is stationary and $\bar \lambda$ is given by \eqref{defbarlambda}. We now use that  $U$ is a deterministic map, that $\bar \mu$ is stationary and that \eqref{kdjfnvlkjn} holds to get 
\begin{align*}
\lim_{t\to -\infty} \E\left[ \inte U(t, x, \bar \mu_0)dx\right] +\bar \lambda t & = \lim_{t\to -\infty} \E\left[ \inte U(t, x, \bar \mu_t)dx\right] +\bar \lambda t\\ 
&= \lim_{t\to -\infty} \E\left[ \inte \bar u_t(x)dx\right]+ \bar c+\bar \lambda t = \E\left[ \inte \bar u_0(x)dx \right] + \bar c . 
\end{align*}
Hence  \eqref{limintU} holds with $c= \E\left[ \inte \bar u_0(x)dx \right] + \bar c$. As $U$ is globally Lipschitz on $\T^d \times \mathcal P(\T^d)$ uniformly in $t$, we infer that 
\begin{align*}
 \limsup_{t\to -\infty}\sup_{(x,m) \in \T^d \times \mathcal P(\T^d)} \left|   U(t,x, m) + \bar \lambda t\right| 
%\\ 
%& \qquad 
& \leq \limsup_{t\to -\infty} \left|   \E\left[ \inte U(t, x, \bar \mu_0)dx\right] +\bar \lambda t \right| + C \\
 &  \leq  \left| \E\left[ \inte \bar u_0(x)dx \right] + \bar c\right| + C. 
\end{align*}
This implies that $ U(t,\cdot, \cdot)  + \bar \lambda t$ is globally bounded.
\end{proof}

Finally we show a uniform regularity in time of $U$: 
\begin{lem}\label{lem.uniform_time_regularity_finitemaster} Under the assumption of Lemma \ref{lem.uniform_bound_finitemaster}, there is a constant $C>0$ such that, for any $h\in (0,1)$ and $t< t+h\leq0$,
\be\label{es.uniform_bound_finitemaster_t}
\sup_{m\in\mathcal{P}(\T^d)} \normD{U(t+h,\cdot,m) - U(t,\cdot,m)} \leq Ch^{\frac{1}{2}}
\ee
\end{lem}

\begin{proof} We recall that, by Section~5.4.2 of \cite{CDLL}, we have, for each $T>0$, $h>0$, $x\in \T^d$ and $m_0\in \mathcal P(\T^d)$,  
$$
U(-T+h,x+\sqrt{2\cwn}W_h,m_h^T) = U(-T+h,x+\sqrt{2\cwn}W_h,(\operatorname{Id}+\sqrt{2\cwn}W_h) \sharp \shif m_h^T) = \shif u^T_h(x), 
$$
where $(\uTW,\mTW,\MTW)$ is the solution of the MFG system \eqref{eq.shifshif}. Let us note that a.s. we have,
\[
\shif u_0^T(x) = \EP{ \Gamma_h \ast \shif u^T_h(x) + \int_0^h \Gamma_{s} \ast ( \frac{1}{2} |D  \shif u_s^T|^2 - \shif f_s )(x) ds },
\]
where $\Gamma$ is the heat kernel. 
Hence by Lemma \ref{lem.uniform_bound_finitemaster}
\begin{align*}
& | U (-T+h,x,m_0) - U(-T,x,m_0) | \\
& \leq \bigg|  U(-T+h,x,m_0) - \EP{ U(-T+h,x+\sqrt{2\cwn}W_h,m^T_{h})} \bigg| + \bigg|  \EP{   \shif u^T_h(x)} - \shif u^T_0(x) \bigg| \\
& \leq C\EP{ |\sqrt{2\cwn}W_h | + {\bf d}_1( m^T_h, m_0)} + \EP{ \bigg|  ( \Gamma_h -\delta_0)\ast \shif u^T_h(x) + \int_0^h \Gamma_{s} \ast ( \frac{1}{2} |D  \shif u_s^T|^2 - \shif f_s )(x) ds   \bigg| } \\
& \leq C\Bigl( \EP{\mathbf{d}_1(\shif m_h^T, m_0)} + \sqrt{2 \cwn h} + \sqrt{h} C \sup_{t \leq 0} \EP{\normD{\tilde{\shif u}^T_t}} + h \cdot C  ( \|D  \shif u^T\|^2_\infty+ \|f\|_\infty)
  \Bigr) \leq Ch^{1/2}. 
\end{align*}
where the third inequality is true by the fact that for any $0 \leq \alpha' < \alpha \leq 1$, $0<h\leq1$ and $\varphi \in C^{n+\alpha}(\T^d)$, there exists a constant $C$ depending only on $\alpha'$, $\alpha$, and $d$ such that,
\[
\normU{(\Gamma_h - \delta_0) \ast \varphi }{C^{n+\alpha'}} = \normU{(\Gamma_h - \delta_0) \ast \tilde\varphi }{C^{n+\alpha'}} \leq C h^{\frac{\alpha-\alpha'}{2}} \normU{\tilde\varphi}{C^{n+\alpha}}.
\]
Thanks to the same estimate for $D_x U$, we conclude the proof.
\end{proof}

\begin{proof}[Proof of Theorem \ref{thm.CVU}] 
Let us define $V^T: \mathbb{R} \times \T^d \times \mathcal{P}(\T^d) \to \mathbb{R}$ by
\[
V^T(t,x,m)= 
\left\{ 
\begin{array}{lc}
    U( t - T , x , m ) + \bar \lambda (t - T) &, \text{ for $t \leq T$,}  \\
    g(x,m) &, \text{ for $t > T$,}
\end{array}
\right.
\]
Combining Lemmata \ref{lem.uniform_bound_finitemaster}, \ref{lem.7.6} and \ref{lem.uniform_time_regularity_finitemaster}, we know that $\{V^T\}_{T \geq 0}$ is uniformly bounded and equi-continuous on $\mathbb{R} \times \T^d \times \mathcal{P}(\T^d)$. More precisely, we have in fact  that for any $\beta\in(0,1)$, there exists $C>0$ such that,
\[
\begin{aligned}
\sup_{\substack{t \in \R \\ |h| \leq 1}} h^{-\frac{1}{2}} & \normU{V^T(t+h,\cdot,\cdot) - V^T(t,\cdot,\cdot)}{C^1} \\
+ \sup_{\substack{t \in \R \\ m \in \mathcal{P}(\T^d)}} & \|D_x V^T(t,\cdot, m)\|_{C^{2+\beta}} 
+ \normU{D_m V^T(t,\cdot,m,\cdot)}{2+\beta,1+\beta} \leq C.
\end{aligned}
\]
This shows that the family $\{V^T\}$ is uniformly $\frac{1}{2}$-Holder continuous in time and Lipschitz in space and measure. By the Arzela-Ascoli Theorem, there exists a sequence $T_n \to +\infty$ such that: (a) $V^{T_n}$ converges locally uniformly to a limit $V=V(t,x,m)$; (b) the first to third space derivatives of $V^{T_n}$ converge locally uniformly to the corresponding derivatives of $V$. Here, local uniform convergence means that for any $\tau > 0$, the convergence is uniform on compact sets $[-\tau, \tau] \times \T^d \times \mathcal{P}(\T^d)$. \\

Let us first note that $V-\bar\lambda t$ is indeed a weak solution to the master equation \eqref{eq.finitemaster} on $\R\times \T^d\times \mathcal P(\T^d)$ as defined in Definition \ref{defweakfinitemaster}. To prove this, let us fix $\hat m_0\in \mathcal P$ (for simplicity of notation, we do the proof for $t_0=0$) and let us compare $(\shif u^n_t,\shif m^n_t , \shif M^n_t)$ and $(\shif u^\infty_t, \shif m^\infty_t, \shif M^\infty_t)$ on $[0, T]$, which are respectively the solutions of equations \eqref{eq.MFGCellshif} with initial condition $\hat m_0$ and terminal conditions,
\[
\shif u^n_T(x) = V^{T_n}(T,x + \sqrt{2\cwn} W_T, (\operatorname{Id} + \sqrt{2\cwn} W_T )\sharp \shif m^n_T),
\]
and
\[
\shif u^\infty_T (x) = V(T,x + \sqrt{2\cwn} W_T, (\operatorname{Id} + \sqrt{2\cwn} W_T )\sharp \shif m^\infty_T).
\]
Thanks to the It\^{o}'s formula (see \cite[Lemma 4.3.11]{CDLL}) and the monotonicity of $f$, we have that,
\[
\EP{\int_{\T^d} (\shif u^n_T - \shif u^\infty_T ) (\shif m^n_T - \shif m^\infty_T ) } \leq - \frac{1}{2} \EP{\int_0^T\int_{\T^d} |D \shif u^n_t - D \shif u^\infty_t |^2 (\shif m^n_t + \shif m^\infty_t) }.
\]
Note that for the left hand side of the above inequality, we have that,
\[
\begin{aligned}
  & \EP{\int_{\T^d} (\shif u^n_T - \shif u^\infty_T ) (\shif m^n_T - \shif m^\infty_T ) } \\
= & \EP{\int_{\T^d} \Big( U(T-T_n, \tau^T, \tau^T\sharp \shif m^n_T) - V(T,\tau^T, \tau^T\sharp \shif m^\infty_T) \Big) (\shif m^n_T - \shif m^\infty_T ) }    \\
= & \EP{\int_{\T^d} \Big( U(T-T_n, x, \tau^T\sharp \shif m^n_T) - V(T, x, \tau^T\sharp \shif m^\infty_T) \Big) d(\tau^T\sharp\shif m^n_T - \tau^T\sharp\shif m^\infty_T )(x) }    \\
\geq & \EP{\int_{\T^d} \Big( U(T-T_n, x, \tau^T\sharp \shif m^n_T) - V(T, x, \tau^T\sharp \shif m^n_T) \Big) d(\tau^T\sharp\shif m^n_T - \tau^T\sharp\shif m^\infty_T )(x)} \\
\geq & - 2 \normU{U(T-T_n,\cdot~,\cdot) - V(T,\cdot~,\cdot~)}{\infty},
\end{aligned}
\]
where $\tau^T:\T^d\times\Omega\to\T^d$ is defined as $\tau^T(x,\omega) = x +W_T(\omega)$ and we used the monotonicity of $V$ in the first inequality. Hence, by combining the above inequalities with Lemma \ref{lem.regularity}, we obtain for any $0 < t_0 < T$ that
\[
\lim_{n\to\infty} \EP{\int_{t_0}^T \int_{\T^d} |D \shif u^n_t - D \shif u^\infty_t |^2 } = 0 \ , 
\]
thus deriving that,
\[
\lim_{n\to\infty} \EP{{\bf d}_1 (\shif m^n_t, \shif m^\infty_t)} = 0, \qquad \forall~ t \in [0,T].
\]
We now apply the comparison principle between $\shif u^n$ and $\shif u^\infty$, which both satisfy a Hamilton-Jacobi equation on $[0,T]$ with a slightly different right-hand side and slightly different terminal condition. We have 
\begin{align*}
\E\left[ \sup_{t,x}| \shif u^n_t(x) -\shif u^\infty_t(x)|\right] \leq  & \E\left[  \int_0^T \sup_x |-\shif f_t(x,\shif m^n_t)+\shif f_t(x,\shif m^\infty_t)|\right]  \\ 
& \qquad\qquad +  \E\left[  \sup_x | U(T-T_n, x, \tau^T\sharp \shif m^n_T)-V(T,x, \tau^T\sharp  \shif m^\infty_T) | \right].
\end{align*}
Since $f$, $U$, and $V$ are Lipschitz continuous and that $U(T-T_n,\cdot~,\cdot~)$ converges to $V(T,\cdot~,\cdot~)$, the right hand side of the above equation tend to 0. On the other hand,
\[
\shif u^n_t (x) = V^{T_n} (t, x+\sqrt{2\cwn}W_t, (\operatorname{Id}+ \sqrt{2\cwn}W_t)\sharp\shif m^n_t),
\]
so that, letting $n\to \infty$ shows that,
 $$
 \shif u^\infty_t(x) = 
 V(t, x+\sqrt{2\cwn}W_t, (\operatorname{Id}+ \sqrt{2\cwn}W_t)\sharp\shif m^\infty_t) \qquad \text{a.s.}. 
 $$
Thus  $V-\bar\lambda t$ is indeed a weak solution to \eqref{eq.ergomaster}. \\

Next we show that \eqref{cv.Uinfty} holds. Let $c$ be the constant satisfying \eqref{limintU}.  Lemma \ref{lem.7.6} and the convergence of $V^{T_n}$ imply  that 
\be\label{uhjgkzedrfv}
\E\left[\inte V(t,x , \bar \mu_0)dx\right] = c\qquad \forall t\in \R.
\ee
Our goal is to show that $V-  \chi$ is constant, namely
\be\label{V-chi=cte}
V(t,x,m)-\chi(x,m) = c- \E\left[\inte \chi(y , \bar \mu_0)dy\right] . 
\ee

Fix $\hat m_0 \in \mathcal P(\T^d)$. Let $m$ and $\mu$ be respectively the solution to 
$$
dm_t = ((1+\cwn)\Delta m_t -{\rm div}(m_t D_x V(t,x, m_t)))dt -\sqrt{2\cwn} Dm_t \cdot dW_t, \qquad m_{0}= \hat m_0
$$
and 
$$
d\mu_t = ((1+\cwn)\Delta \mu_t -{\rm div}(\mu_t D_x \chi(x, \mu_t)))dt -\sqrt{2\cwn} D\mu_t \cdot dW_t, \qquad \mu_{0}= \hat m_0
$$
We also set 
$$
u_t(x)= V(t,x, m_t)\qquad {\rm and}\qquad w_t(x)= \chi(x, \mu_t).
$$
 Recall that by the definition of weak solution, for any $T>0$ there exists $v^T$ and $z^T$ such that 
$(u,m,v^T)$ and $(w,\mu,z^T)$ solve \eqref{eq.MFGCN} on the time interval $[0,T]$. 
Thus, by Theorem  \ref{thm.main} and the space regularity of $V$ and $\chi$, 
$$
\E\left[ \|m_t - \bar \mu_t\|^2_{L^2_x}\right]+ \E\left[ \| \mu_t - \bar \mu_t\|^2_{L^2_x}\right] \leq C (e^{-\omega t}+ e^{-\omega(T-t)})\qquad \forall t\in (2,T-2). 
$$
Note that both $m_t$ and $\mu_t$ do not depend on $T$.  By letting $T\to\infty$, we find that both $m_t$ and $\mu_t$ converge to $\bar \mu_t$ as $t\to\infty$. \\

Moreover, after the usual time shift \eqref{shifumVsum}, $(\shif u^T,\shif m, \shif M^T)$  and $(\shif w^T, \boldsymbol{\mu},\shif N^T)$ (for some martingale processes $\shif M^T$ and $\shif N^T$) solve the MFG system \eqref{eq.MFGCellshif}  with the same initial condition $m_0$ and with the  terminal conditions $V(T, \cdot+\sqrt{2\cwn} (W_T-W_{t_0}),(\operatorname{Id} +\sqrt{2\cwn} (W_T-W_{t_0})\sharp \shif m_T)$ and $\chi(\cdot+\sqrt{2\cwn} (W_T-W_{t_0}),(\operatorname{Id} +\sqrt{2\cwn} (W_T-W_{t_0}))\sharp \shif m_T),$ respectively.\\

Now by Lasry-Lions monotonicity argument, the fact that $m_0=\mu_0=\hat m_0$ and Poincarr\'e's inequality, we have that,
\begin{align*}
\EP{\int_0^T\int_{\T^d} |D \shif u^T_t - D \shif w^T_t |^2 (\shif m_t + \boldsymbol{\mu}_t) } & \leq - 2 \EP{\int_{\T^d} (\shif u^T_T - \shif w^T_T ) (\shif m_T - \boldsymbol{\mu}_T ) } \\
& \leq C \|D\shif u^T_T - D\shif w^T_T\|_{L^2_{\omega,x}}   \|\shif m_T - \boldsymbol{\mu}_T \|_{L^2_{\omega,x}}.
\end{align*}
This can equivalently be written as 
\begin{align*}
\EP{\int_0^T\int_{\T^d} |D u_t - D w_t |^2 (m_t + \mu_t) } & \leq C \|D u_T - D w_T\|_{L^2_{\omega,x}}   \| m_T - \mu_T \|_{L^2_{\omega,x}}.
\end{align*}
Since $ u_t$ and $ w_t$ are uniformly Lipschitz, we derive from the convergence of $m_t$ and $\mu_t$ in $L^2_{x,\omega}$ to $\bar \mu_t$ as $t\to\infty$ that,
\[
\lim_{T\to\infty} \EP{\int_0^T\int_{\T^d} |D  u_t - D  w_t |^2 ( m_t + \mu_t) } = 0.
\]
Then the bound below on $ m_t$ and  $ \mu_t$ given in Lemma \ref{lem.regularity} implies that, $D  u_t (x) = D  w_t (x)$ a.s., which in turn entails that $ m_t = \mu_t$ a.s.. 
Hence there is a real constant $\kappa$, which still may depend on the choice of $\hat m_0$ at this point, such that $u_0-w_0= \kappa$. To complete the proof we  need to show that $\kappa=c- \E\left[\inte \chi(y , \bar \mu_0)dy\right]$, which will imply \eqref{V-chi=cte}. \\

For this, let us  integrate the equations satisfied by $u$ and $w$ on $[t_0,T]\times \T^d$: we get, since  $(Du_t,m_t)=(Dw_t,\mu_t)$, 
\begin{align*}
\kappa & = \inte (u_0-w_0)dx\\ 
 & = \E\left[ -\bar \lambda (T-t_0)-\int_{t_0}^T \inte (\frac12 | Du_t(x)|^2- f(x, m_t))d)dt + \inte V(T, x, m_T)dx\right] \\
& \qquad - \E\left[  -\bar \lambda (T-t_0)-\int_{t_0}^T \inte  (\frac12 | Dw_t(x)|^2+ f(x, \mu_t))dxdt + \inte \chi(x, \mu_T)dx\right] \\
& = \E\left[ \inte (V(T, x, m_T)- \chi (x, \mu_T)) dx\right]. 
\end{align*}
We know that 
$$
\lim_{T\to\infty} \E\left[ \| m_T-\bar \mu_T\|^2_{L_x^2} + \| \mu_T-\bar \mu_T\|^2_{L_x^2} \right]=0
$$
and that $V$ and $\chi$ are uniformly Lipschitz continuous in the space and measure arguments. Hence, by stationary of $\bar \mu$ and   \eqref{uhjgkzedrfv}, 
\begin{align*}
\kappa & = \lim_{T\to\infty} \E\left[ \inte (V(T, x, \bar \mu_T)- \chi (x, \bar \mu_T)) dx\right] \\ 
& = \lim_{T\to\infty} \E\left[ \inte (V(T, x, \bar \mu_0)- \chi (x, \bar \mu_0)) dx\right]  = c- \E\left[ \inte \chi (x, \bar \mu_0) dx\right]. 
\end{align*}
This shows \eqref{V-chi=cte} and completes the proof of the theorem. 
\end{proof} 

\begin{proof}[Proof of Corollary \ref{cor.uT-lambda}]
As $(u^T,m^T,v^T)$ is the solution of \eqref{eq.MFGCN}-\eqref{eq.MFGCNbc}, we know that $\uT(x) = U(t - T, x, \mT)$ and that $\mT$ solves the McKean-Vlasov equation
\[
d \mT = \left((1+\cwn)\Delta \mT +{\rm div}(\mT D_x U(t-T, \cdot~, \mT))\right)dt -\sqrt{2\cwn} D  \mT \cdot dW_t,\qquad m^T_0= m_0.
\]
Note that, by It\^o-Wenzell formula for Schwartz distribution-valued processes (see \cite[Theorem 1.1]{K11}), $\mTW=(\operatorname{Id}-\sqrt{2\cwn}W_t)\sharp \mT$ is the unique solution of the following equation,
\[
d \mTW = \left( \Delta \mTW +{\rm div}(\mTW D_x U(t-T, \tau_t, \tau_t \sharp \mTW))\right)dt, \qquad {\shif m}^T_0= m_0,
\]
which is a McKean-Vlasov equation solved $\omega$ by $\omega$, where $\tau_t:\Omega\times\T^d\to\T^d$, $\tau_t(x,\omega):=x+\sqrt{2\cwn}W_t(\omega)$. In addition, the unique solution of this equation satisfies the following time regularity,
\[
\esssup \sup_{s \neq t} \frac{{\mathbf d}_1(\mTW,\shif m_s^T)}{|t-s|^{\frac{1}{2}}} \leq C (1+\sup_{-T\leq t\leq 0} \normU{D_x U(t,\cdot~,\cdot~)}{\infty}).
\]
Since the map \( x \mapsto D_x U(t, \cdot, m) \) is uniformly bounded in \( C^1(\T^d) \) with respect to \( t \) and \( m \), as shown in Lemma~\ref{lem.uniform_bound_finitemaster}, it follows that, for any \( t_1 > 0 \), the family \( \{ \shif m^T \}_{T \geq 0} \), when restricted to the interval \( [0, t_1] \), almost surely consists of functions that are uniformly \( \tfrac{1}{2} \)-H\"{o}lder continuous from \( [0, t_1] \) into \( \mathcal{P}(\T^d) \), where \( \mathcal{P}(\T^d) \) is endowed with the Wasserstein-1 metric \( \mathbf{d}_1 \) and is compact.\\

We now temporarily fix \( \omega \). We claim that, for this fixed $\omega$ and as $T\to\infty$, \( \{ \shif m^T \}_{T \geq 0} \) converges locally uniformly  in time to the solution  \( \shif m^\infty \colon [0, \infty)\to \mathcal{P}(\T^d) \) of the following McKean-Vlasov equation:
\be\label{tmp_MV_equation}
d \shif m^\infty_t = \left( \Delta \shif m^\infty_t + \operatorname{div}\left(\shif m^\infty_t D_x \chi(\tau_t, \tau_t \sharp \shif m^\infty_t)\right)\right) dt, \qquad \shif m^\infty_0 = m_0.
\ee
Indeed, by the Arzel\`{a}-Ascoli Theorem, for any subsequence of \( \{ \shif m^T \}_{T \geq 0} \), there exists a further subsubsequence (still denoted by \( \shif m^{T_n} \)) that converges locally uniformly, as \( n \to \infty \), to a limit function \( \shif m^\infty \colon [0, \infty) \to \mathcal{P}(\T^d) \), which is also \( \tfrac{1}{2} \)-H\"{o}der continuous. Namely,
\[
\sup_{0 \leq t \leq t_1} \mathbf{d}_1(m_t^{T_n}, m_t^\infty) \xrightarrow[n \to \infty]{} 0,\qquad \forall t_1>0.
\]
Let us check that \( \shif m^\infty \) is the weak solution of \eqref{tmp_MV_equation}. Indeed, for the fixed \( \omega \), and for every \( \varphi \in C_c^\infty([0,\R) \times \T^d) \), we have the following identity for \( \shif m^{T_n} \):
\[
\begin{aligned}
- \int_0^{\infty} \int_{\T^d} & \shif m_t^{T_n}(x) \partial_t \varphi(t,x)\, dx dt 
- \int_{\T^d} \varphi(0,x)\, m_0(dx)
= \\
& \int_0^{\infty} \int_{\T^d} \left( \Delta \varphi(t,x) - D \varphi(t,x) \cdot D_x U(t - T_n, \tau^t(x), \tau^t \sharp \shif m^{T_n}_t) \right) \shif m^{T_n}_t(x)\, dx dt.
\end{aligned}
\]
Note that, by Theorem \ref{thm.CVU}, the function \( D_x U(t - T_n, x, m) \) converges to $D_x\chi(x,m)$ uniformly with respect to $(t,x,m)\in[0,t_1]\times\T^d\times \mathcal{P}(\T^d)$ for any $t_1>0$. Since \( D_x \chi \) is Lipschitz continuous in \( (x, m) \), as shown in Theorem \ref{lem.existence_weak_solution_cell_Master_Equation}, we deduce that uniformly for $(t,x)$,
\[
D_x U(t - T_n, \tau_t(x), \tau_t \sharp \shif m_t^{T_n}) \xrightarrow[n \to \infty]{} D_x \chi(\tau_t(x), \tau_t \sharp \shif m_t^\infty).
\]
Moreover, the uniform convergence of \( \shif m^{T_n} \) to \( \shif m^\infty \) in the Wasserstein-1 distance implies, in particular, the weak-* convergence of \( \shif m^{T_n} \) to \( \shif m^\infty \) when viewed as measures on \( [0, \infty) \times \T^d \). Combining these convergences, we may pass to the limit, obtaining:
\[
\begin{aligned}
- \int_0^{\infty} \int_{\T^d} & \shif m_t^\infty(x) \, \partial_t \varphi(t,x)\, dx dt 
- \int_{\T^d} \varphi(0,x)\, m_0(dx) \\
&= \int_0^{\infty} \int_{\T^d} \left( \Delta \varphi(t,x) - D \varphi(t,x) \cdot D_x \chi(\tau_t(x), \tau_t \sharp \shif m_t^\infty) \right) \shif m_t^\infty(x)\, dx dt,
\end{aligned}
\]
which shows that \( \shif m^\infty \) is a weak solution of the McKean-Vlasov equation \eqref{tmp_MV_equation}. Thus a.s. and as $T\to\infty$, \( \{ \shif m^T \}_{T \geq 0} \) converges locally uniformly  in time to the solution  \( \shif m^\infty \colon [0, \infty)\to \mathcal{P}(\T^d) \) of the McKean-Vlasov equation \eqref{tmp_MV_equation}.

On the other hand, equation \eqref{tmp_MV_equation} admits a unique solution \( \bar{\shif m}_t := (\operatorname{Id}-\sqrt{2\cwn}W_t)\sharp \mS \) by It\^o-Wenzell formula. Hence, for the fixed $\omega$, we have $\shif m^\infty = \bar{\shif m}$. Since the above reasoning holds almost surely true,  we deduce that \( \mT \) converges almost surely in \( C^0([0,t_1]; \mathcal{P}(\T^d)) \) to \( \mS \), for any \( t_1 \ge 0 \). Then, applying Theorem \ref{thm.CVU} once again, we obtain a.s. that:
\begin{align*}
\lim_{T \to +\infty} \Big[ \uT(x) -\bar \lambda (T-t) \Big] -& \Big[ \chi( x, \bar m_t) + c \Big] \\
&= \lim_{T \to +\infty} \Big[ U(t - T, x, \mT) + \bar{\lambda}(t - T) \Big] - \Big[ \chi( x, \bar m_t) + c \Big]
= 0,
\end{align*}
uniformly for \( (t,x) \in [0, t_1] \times \T^d  \) for any $t_1>0$. Finally, fixing \( \theta \in (0,1) \), we have by Theorem \ref{lem.keyuniqueCT} combined with to \cite[Theorem 6.15]{V08} that \( {\mathbf d}_1 ( \bar{m}_{\theta T}, \bar{\mu}_{\theta T} ) \) converges to 0 in \( L^2_{\omega} \). Hence, combining this with the argument above, we deduce that
\[
\Big[ u^T_{\theta T}(x) - (1-\theta)\, \bar{\lambda}\, T \Big] - \Big[ \chi\left( x, \bar{\mu}_{\theta T} \right) + c \Big] \xrightarrow[T \to \infty]{} 0 \qquad \text{in } L^2_\omega,
\]
uniformly for \( x \in \T^d \).
\end{proof}

%%%%%%%%%%%%%%%%%%%
%%%%%%%%%%%%%%%%%%%
%%%%%%%%%%%%%%%%%%%
\appendix
\section{Estimates for some linear systems}\label{sec:AppendEsti}

This part is largely borrowed from \cite{CLLP2,CP19} where similar estimates on the long-time behavior of some linear equations are stated. We work here in the probabilistic setting explained in Section \ref{Sec:NotHyp}. In all the results, the filtration can be either $(\mathcal F_{t_0,t})_{t\in [t_0, T]}$, or $(\mathcal F_t)_{t\in \R}$. The first statement is about  linear equations in divergence form and is borrowed from \cite[Lemma 7.1, Lemma 7.6]{CLLP2}, see also  \cite{Por24}), except that the statement there is deterministic. Solving the equation $\omega$ by $\omega$ and then checking the progressive measurability, one derives the following lemma.

\begin{lem} \label{lemCLLP2.1} Let $V$ be a bounded progressively measurable random vector field on $[t_0,T]\times \T^d$. We consider the following equation,
\be\label{eq.kolmo}
\begin{cases}
\partial_t \mu-\Delta \mu +\dive (\mu V)= \dive( B)& \qquad {\rm in} \; (t_0,T)\times \T^d, 
\\
\mu(t_0)=\mu_0 & \qquad {\rm in} \;   \T^d.
\end{cases}
\ee 
Then we have 
\begin{enumerate}
    \item If $B \equiv 0$ and $\mu_0\in L^2(\Omega;L^1(\T^d))$ is $\mathcal{F}_{t_0}$ measurable s.t. $\inte \mu_0=0 ~a.s.$, then there exists a unique solution $\mu\in L^2( \Omega ; C ( [t_0,T]; L^1(\T^d)))$ progressively measurable which solves the equation in the sense of distributions. Moreover, there exist constants $\lambda>0$ and $C>0$, depending only on $\esssup\|V\|_{L^\infty_{t,x}}$, such that 
    \be\label{es.a-priori_L2L1_FPDE}
    \|\mu(t)\|_{L^2_\omega(L^1_x)}\leq Ce^{-\lambda t} \|\mu_0\|_{L^2_\omega(L^1_x)};
    \ee
    \item If $B\in L^2([t_0,T] \times \T^d \times \Omega)$ is a progressively random field and $\mu_0\in L^2(\T^d\times\Omega)$ is $\mathcal{F}_{t_0}$ measurable s.t. $\inte \mu_0=0 ~a.s.$, then there exists a unique solution $\mu\in L^2(  \Omega ; C ( [t_0,T]; L^2(\T^d)))$ which is a progressively measurable solution of Equation \eqref{eq.kolmo} in the sense of distributions. Moreover, there exist constants $\lambda>0$ and $C>0$, depending only on $\esssup\|V\|_{L^\infty_{t,x}}$, such that 
    \be\label{es.a-priori_L2_FPDE}
    \normtwo{\mu(t)} \leq Ce^{-\lambda t} \normtwo{\mu_0} +C\left[\int_0^t \normtwo{B(s)}^2 ds\right]^{1/2}.
    \ee
%    Moreover, 
%    \be\label{es.a-priori_L2_sup_time_FPDE}
%    \EP{ \sup_{0 \leq t \leq T } \normtwox{\mu(t)}^2}^{\frac{1}{2}} \leq C \normtwo{\mu_0} +C\left[\int_0^t \normtwo{B(s)}^2 ds\right]^{1/2}.
%    \ee
\end{enumerate}
\end{lem}

The next lemma provides estimates on the solution of backward stochastic linear equations. It is the generalization to our random framework of \cite[Lemma 7.4, Lemma 7.5]{CLLP2}. 

\begin{lem} \label{lemCLLP2.2} Let $\delta \geq 0$ and $V\in L^\infty( [0,T] \times \Omega ; L^\infty(\T^d;\R^d ) )$ and $A\in L^2( [0,T] \times \Omega ; L^2(\T^d) )$ be progressively measurable processes. Assume that $(v,M)$ is a  classical solution to the backward equation
\be\label{eq.backward_stochastic}
d v_t=(-\Delta v_t+\delta v_t +V_t\cdot Dv_t+ A_t)dt + dM_t \qquad {\rm in} \; (0,T)\times \T^d. 
\ee
Set $\tilde v(t,x)= v(t,x)-\inte v(t)$. There exist constants $\lambda>0$ and $C>0$, depending only on $\esssup\|V\|_{L^\infty_{t,x}}$, such that, for any $0\leq t_0\leq t \leq T$,  
\be\label{zerldjkjnpimj}
\normtwo{\tilde v(t_0)}\leq Ce^{-\lambda (t-t_0)} \normtwo{\tilde v(t)} +C\int_{t_0}^t e^{-\lambda(s-t_0)}\normtwo{A(s)}ds. 
\ee
Moreover, 
\be\label{zerldjkjnpimj2}
(t-t_0) \normtwo{D  v(t_0)}   \leq C(t-t_0+1) \left( \normtwo{\tilde v(t)} + \int_{t_0}^t (\normtwo{\tilde v(s)}+ \normtwo{A(s)})ds \right) .
\ee
%{
%\wenbin
%Finally, if that $A\in L^2( [0,T] \times \Omega ; L^\infty(\R^d) )$, then we have,
%\be\label{es.a-priori_Linfty_BSPDE}
%\|\tilde v(t_0) \|_{L^2_\omega(L^\infty_x)} \leq C e^{-\lambda (t-t_0)} \|\tilde v(t) \|_{L^2_\omega(L^\infty_x)} +C\int_{t_0}^t e^{-\lambda(s-t_0)}\|A(s)\|_{L^2_\omega(L^\infty_x)} ds. 
%\ee
%}
\end{lem}

%Note that we in deed have that,
%\[
%d \mT = \Big( (1+\cwn) \Delta \mT + \div(\mT D U(t-T,\cdot,\mT)) \Big) dt - \sqrt{2\cwn} D \mT d W_t
%\]

By a classical solution, we mean the following: $v$ is an adapted process  with paths in $C^0([0,T], C^2(\T^d))$ and such that 
$$
\sup_{t\in [0,T]} \|v_t(\cdot)\|_{C^2} <\infty.
$$
The process $(M_t)$ is adapted process with paths in the space $C^0([0,T], C^1(\T^d))$, such that, for any
$x \in \T^d$, $t\to M_t(x)$ is a martingale. It is also required to satisfy
$$
\sup_{t\in [0,T]} \|M_t(\cdot)\|_{C^1} <\infty.
$$
The estimates above could be stated with weaker notions of solutions; however, as all the solutions we manipulate are classical, we won't need it and the fact that the solution is classical simplifies the proof. 

\begin{proof} Fix $t\in [0,T)$ and let $\xi \in L^2(\T^d\times \Omega)$ be $\mathcal F_t-$measurable. Let $\rho$ be a random solution to the forward equation 
$$
\begin{cases}
\partial_s \rho - \Delta \rho -{\rm div} (\rho \,V)=0 & \hbox{in $(t,T)\times \T^d$,}
\\
\rho(t)=\xi -\inte \xi \,. &
\end{cases}
$$
By approximation (regularizing $V$ and $\xi$), one easily checks that such a solution $\rho$ exists. It can be built in such a way that it is adapted and satisfies $\inte \rho(s)=0$ for any $s\in [t,T]$. By duality (which can also be checked by regularization) we have
$$
e^{-\delta T}\inte \rho(T)v(T)-e^{-\delta t}\inte (\xi -\inte \xi )v(t) = \int_t^T \inte e^{-\delta s}A\rho +\int_t^T \inte e^{-\delta s}\rho dM . 
$$
Recalling that $\rho(s)$ has zero average for any $s$, taking expectation and using the estimate \eqref{es.a-priori_L2_FPDE} for $\rho$ in Lemma \ref{lemCLLP2.1} we infer that 
\begin{align*}
& e^{-\delta t} \E\left[ \inte \xi (v(t)-\inte v(t))\right] = e^{-\delta t} \E\left[ \inte (\xi -\inte \xi )v(t)\right] \\ 
& \qquad  \leq \E\left[ \inte e^{-\delta T} \rho(T)(v(T)-\inte v(T))\right] -\E\left[\int_t^T \inte e^{-\delta s}A\rho \right]  \\ 
& \qquad\leq Ce^{-\lambda (T-t)-\delta T}\left(\E\left[\|\xi-\inte \xi \|^2_{L^2}\right]\right)^{1/2} \left(\E\left[\|v(T)-\inte v(T)\|^2_{L^2}\right]\right)^{1/2}\\
& \qquad\qquad + C  \int_t^T  e^{-\lambda (s-t)-\delta s} \left(\E\left[ \|A(s)\|_{L^2}^2\right]\right)^{1/2} ds \left(\E\left[\|\xi-\inte \xi \|^2_{L^2}\right]\right)^{1/2} ,
\end{align*}
where $C$ and $\lambda$ depend on $\|V\|_\infty$ only. 
Taking the supremum over $\xi\in L^2_{\omega,x}$ with $\|\xi\|_{L^2_{\omega,x}}\leq 1$ gives \eqref{zerldjkjnpimj}.
%{\wenbin
%Now, using the same technique, except taking the supremum over $\xi \in L^2_\omega(L^1_x)$ with $\|\xi\|_{L^2_\omega(L^1_x)} \leq 1$ and applying the estimate \eqref{es.a-priori_L2L1_FPDE} for $\rho$, we obtain \eqref{es.a-priori_Linfty_BSPDE}.
%}
\\

We now estimate $Dv$.  Let us set $\tilde v(t,x)= v(t,x)-\inte v(t)$. Note that $\tilde v$ solves 
$$
d\tilde v = (-\Delta \tilde v + \delta \tilde v+ V\cdot D\tilde v+A-h(t)) dt +d\tilde M
$$
where 
$$
h(t)= -\inte (\delta v+V\cdot D v+A) , \qquad \tilde M(t,x)= M(t,x)-\inte M(t).
$$
By It\^{o}'s formula, we have 
\begin{align*}
\frac12 \E\left[\inte (\tilde v(t))^2\right]  & = \frac12 
\E\left[\inte (\tilde v(t_0))^2\right] +\E\left[ \int_{t_0}^t  \inte ( |D \tilde v|^2 +\delta \tilde v^2+ (V\cdot D\tilde v+A-h(t))\tilde v\right] \\ 
& \qquad + \frac12 \E\left[ \inte \lg \tilde M\rg (t)-\inte \lg \tilde M\rg (t_0)\right]. 
\end{align*}
Therefore, as $V$ is bounded and by the definition of $h$,  
\begin{align*}
\E\left[ \int_{t_0}^t  \inte  |D  v|^2 \right] & \leq \E\left[\inte (\tilde v(t))^2\right]  + \E\left[ \int_{t_0}^t  \inte (|\tilde v| |V| |D \tilde v| + (|A|+|h|)|\tilde v|)\right] \\
& \leq  \E\left[\inte (\tilde v(t))^2\right]+ \frac12 \E\left[ \int_{t_0}^t  \inte  |D  v|^2 \right]  + C \E\left[ \int_{t_0}^t  \inte (|\tilde v|^2+ |A|^2 )\right].
\end{align*}
Hence
\be\label{azelirsjdfkn}
\E\left[ \int_{t_0}^t \inte   |D  v|^2 \right]  \leq C \E\left[\inte (\tilde v(t))^2\right]+ C \E\left[ \int_{t_0}^t  \inte (|\tilde v|^2+ |A|^2 )\right].
\ee
We now look at the equation satisfied by $ (t-s) (v_{x_i}(s,x))^2$. Note that $ v_{x_i}$ is a  solution to 
$$
d v_{x_i} = (-\Delta  v_{x_i} + \delta v_{x_i}+ (V\cdot D v)_{x_i}+A_{x_i}) dt +d M_{x_i}.
$$
By It\^{o}'s formula, 
\begin{align*}
0 & = 
(t-t_0) \E\left[\inte (v_{x_i}(t_0))^2\right] +\E\left[ \int_{t_0}^t  \inte ( -(v_{x_i})^2 + (t-s)( |D  v_{x_i}|^2 +\delta v_{x_i}^2- V\cdot D v v_{x_ix_i} - A v_{x_ix_i}))\right] \\ 
& \qquad + \E\left[ \int_{t_0}^t \inte (t-s)  \lg M_{x_i}\rg(ds) \right]. 
\end{align*}
Thus 
\begin{align*}
& (t-t_0) \E\left[\inte (v_{x_i}(t_0))^2\right] +\E\left[ \int_{t_0}^t  \inte  (t-s)|D  v_{x_i}|^2\right]  \\
& \qquad \leq  \E\left[ \int_{t_0}^t  \inte ( (v_{x_i})^2 + (t-s)( C |D v|^2 + C|A|^2 +\frac12 |v_{x_ix_i}|^2))\right] 
\end{align*}
By \eqref{azelirsjdfkn} we get 
\begin{align*}
 (t-t_0)\E\left[\inte (v_{x_i}(t))^2\right]  &   \leq C(t-t_0+1) \left(  \E\left[\inte (\tilde v(t))^2\right] 
%  \\
%& \qquad 
+ C \E\left[ \int_{t_0}^t \inte (|\tilde v|^2+ |A|^2)\right] \right), 
\end{align*}
which gives \eqref{zerldjkjnpimj2}. 
\end{proof}

\section{Exponential decays}\label{sec:AppendExpo}

Let   $\lambda, C_0>0$ be fixed constants and  $\mathcal E(T)=\mathcal E(T, \lambda, C_0)$ be the set of c\`{a}dl\`{a}g maps $\alpha, \beta, \gamma \in L^1([0,T], \R_+)$ satisfying the following four inequalities: for any $t_1,t_2$ with  $0\leq t_1\leq t_2\leq T$:
\be\label{hypexpo1}
\int_{t_1}^{t_2} \beta(t)dt \leq C_0 (\sqrt{\alpha(t_1)\gamma(t_1)}+\sqrt{\alpha(t_2)\gamma(t_2)}) ,
\ee
\be\label{hypexpo2}
\gamma(t_2) \leq C_0 e^{-\lambda(t_2-t_1)} \gamma(t_1)+C_0 \int_{t_1}^{t_2} \beta(t)dt, 
\ee
\be\label{hypexpo3}
\alpha(t_1) \leq C_0 e^{-\lambda(t_2-t_1)} \alpha(t_2)+C_0 \int_{t_1}^{t_2} \beta(t)dt+C_0\sup_{t_1\leq t\leq t_2} \gamma(t).
\ee
\be\label{hypexpo4}
\alpha(t_1) \leq C_0\beta(t_1). 
\ee 

\begin{lem}\label{lem.expodecay} There exists constant $C,\lambda'>0$, depending on $C_0$ and $\lambda$ only, such that, for any $T\geq 1$ and $(\alpha,\beta, \gamma)\in \mathcal E(T)$, and for any $t\in [0,T]$, 
$$
\alpha(t)+\gamma(t)\leq C(e^{-\lambda' t}+e^{-\lambda'(T-t)}) (\alpha(T)+\gamma(0)).
$$
\end{lem}

\begin{proof} 
Throughout the proof, $C$ denotes a constant which depends on $\lambda$ and $C_0$ only and might change from line to line. For $\tau \geq 1$, let 
$$
e(\tau)  = \sup\left\{\alpha(t)+\gamma(t), \begin{array}{l}\; (\alpha,\beta, \gamma)\in \mathcal E(T), \; T\geq 2\tau, \; t\in [\tau, T-\tau]\\
\alpha(0)+\gamma(0)+\alpha(T)+\gamma(T)\leq 1\end{array}\right\}. 
$$
Let us note that $\tau\to e(\tau)$ is nonincreasing. We now check that $e(\tau)$ is finite and 
$$
\lim_{\tau\to \infty} e(\tau)=0.
$$
Indeed, fix $\ep>0$ and $T\geq 1/(4\ep)$. Let $(\alpha, \beta, \gamma)\in \mathcal E(T)$ be such that $\alpha(0)+\gamma(0)+\alpha(T)+\gamma(T)\leq 1$. By \eqref{hypexpo1} and Cauchy-Schwarz inequality, 
\be\label{auekzhrjsdfng}
\int_0^T \beta(t)dt \leq C_0(\alpha(0)+\gamma(0)+\alpha(T)+\gamma(T))\leq C_0.
\ee
Thus \eqref{hypexpo2} and \eqref{hypexpo3} imply that $\gamma$ and $\alpha$ are bounded by some constant $C$ on $[0,T]$ and, as as a consequence, $e(\tau)$ is finite for any $\tau\geq 1$. On the other hand, \eqref{auekzhrjsdfng} implies the existence of $t_1\in (0,1/\ep)$ and $t_2\in (T-1/\ep)$ such that $\beta(t_1), \beta(t_2)\leq C_0\ep$.  Hence, by \eqref{hypexpo4}, $\alpha(t_1), \alpha(t_2)\leq C\ep$.  Coming back to \eqref{hypexpo1}  we obtain that 
$$
\int_{1/\ep}^{T-1/\ep} \beta(t)dt \leq \int_{t_1}^{t_2} \beta(t)dt \leq C_0 \|\gamma\|_\infty^{1/2} (\sqrt{\alpha(t_1)}+\sqrt{\alpha(t_2)}) \leq C\ep^{1/2}. 
$$
We  then infer from \eqref{hypexpo2} that, for any $t\in [2/\ep,T-1/\ep]$, 
$$
\gamma(t) \leq C_0 e^{-\lambda(t-1/\ep)}\gamma(1/\ep) +C_0 \int_{1/\ep}^t \beta(s)ds \leq C e^{-\lambda/\ep} + C\ep^{1/2}.
$$
Then, by \eqref{hypexpo3}, we obtain that, for any $t\in [2/\ep, T-2/\ep]$, 
\begin{align*}
\alpha(t) & \leq C_0 e^{-\lambda(T-1/\ep-t)} \alpha(T-1/\ep)+C_0 \int_{t}^{T-1/\ep} \beta(s)ds+C_0\sup_{t\leq s\leq T-1/\ep} \gamma(s)  \\
&\leq C e^{-\lambda/\ep} + C\ep^{1/2}+ C_0(C e^{-\lambda/\ep} + C\ep^{1/2}).
\end{align*}
Hence, for any $t\in [2/\ep, T-2/\ep]$, 
$$
\alpha(t)+\gamma(t) \leq C (e^{-\lambda/\ep} +\ep^{1/2}).
$$
As $(\alpha, \beta, \gamma)$ and $T$ are arbitrary, this proves that, for any $\tau \geq 2/\ep$, 
$$
e(\tau) \leq C (e^{-\lambda/\ep} +\ep^{1/2}).
$$
Thus $e(\tau)\to 0$ as $\tau\to \infty$. 

We now prove that, for any $\tau$ large enough,  
\be\label{azeioeufsdilmjok}
e(2\tau) \leq \frac12 e(\tau). 
\ee
Let $\tau$ large,  $T\geq 4\tau$ and $(\alpha, \beta, \gamma)\in \mathcal E(T)$ with $\alpha(0)+\gamma(0)+\alpha(T)+\gamma(T)\leq 1$. Then, by the definition of $e(\tau)$,  for any $t\in [\tau, T-\tau]$, 
$$
\alpha(t)+\gamma(t)\leq e(\tau).
$$
Note that, the map 
$$
(\tilde \alpha, \tilde \beta, \tilde \gamma)(t)= (\alpha, \beta, \gamma)(t+\tau)/(2e(\tau))\qquad \forall t\in [0,T-2\tau]
$$
belongs to $\mathcal E(T-2\tau)$ and satisfies $\tilde \alpha(0)+\tilde \gamma(0)+\tilde \alpha(T-2\tau)+\tilde \gamma(T-2\tau)\leq 1$: indeed, the constraints \eqref{hypexpo1}, \eqref{hypexpo2}, \eqref{hypexpo3} and \eqref{hypexpo4} are invariant by translation in time and are positively homogeneous. As $T-2\tau\geq 2\tau$, we infer from the definition of $e(\tau)$ that 
$$
\tilde \alpha(t)+\tilde \gamma(t) \leq e(\tau) \qquad \forall t\in [\tau, T-3\tau].
$$
Thus, for $\tau$ large enough,  
$$
\alpha(t)+\gamma(t) \leq 2(e(\tau))^2\leq \frac12 e(\tau) \qquad \forall t\in [2\tau, T-2\tau].
$$
This implies \eqref{azeioeufsdilmjok}. From \eqref{azeioeufsdilmjok} and the fact that $\tau\to e(\tau)$ is nonincreasing, we deduce that there exists $\lambda'>0$ and $C>0$ such that 
$$
e(\tau) \leq C e^{-\lambda' \tau} \qquad \forall \tau \geq 1. 
$$
Fix now $T\geq 3$ and let $(\alpha, \beta,\gamma)\in \mathcal E(T)$. Given $t\in [0,T]$, set $\tau = \min\{t, T-t\}$ and note that $2\tau  \leq T$. As $t\in [\tau, T-\tau]$, we then have, by the definition of $e(\tau)$, 
$$
\alpha(t)+\gamma(t) \leq e(\tau) \leq C e^{-\lambda' \tau} \leq C(e^{-\lambda' t}+e^{-\lambda' (T-t)}) .
$$
This proves that, for any $T\geq 1$ and $(\alpha,\beta, \gamma)\in \mathcal E(T)$, and for any $t\in [0,T]$,  
$$
\alpha(t)+\gamma(t)\leq C(e^{-\lambda' t}+e^{-\lambda'(T-t)}) (\alpha(0)+\gamma(0)+\alpha(T)+\gamma(T)).
$$
To complete the statement we show next that $\alpha(0)+\gamma(T)$ can be estimated by $\alpha(T)+\gamma(0)$. Indeed, using \eqref{hypexpo3},  \eqref{hypexpo2}  and \eqref{hypexpo1} respectively, we have
\begin{align*}
\alpha(0)+\gamma(T) & \leq Ce^{-\lambda T} \alpha(T)+ C\int_0^T \beta(s)ds + C\sup_{s} \gamma(s) + \gamma(T)  \\
& \leq Ce^{-\lambda T} \alpha(T)+ C\int_0^T \beta(s)ds + C \sup_s (e^{-\lambda s} \gamma(0) +\int_0^s \beta(u)du)\\ 
& \leq Ce^{-\lambda T} \alpha(T)+ C \gamma(0) + C\int_0^T \beta(s)ds\\ 
& \leq Ce^{-\lambda T} \alpha(T)+C\gamma(0)+ C(\sqrt{\alpha(0)\gamma(0)}+\sqrt{\alpha(T)\gamma(T)}) \\ 
& \leq C (\alpha(T)+\gamma(0)) + \frac12 (\alpha(0)+\gamma(T)). 
\end{align*}
Therefore 
\begin{align*}
\alpha(0)+\gamma(T) \leq C (\alpha(T)+\gamma(0)), 
\end{align*}
which completes the proof. 
\end{proof}

We now investigate the counterpart of the previous lemma in the discounted case. 
Let   $\lambda, C_0>0$ be fixed constants and, for $\delta>0$,   $\mathcal E(\delta)=\mathcal E(\delta, \lambda, C_0)$ be the set of c\`{a}dl\`{a}g maps $\alpha, \beta, \gamma \in L^\infty([0,T], \R_+)$ satisfying, for any $t_1,t_2$ with  $0\leq t_1\leq t_2$, \eqref{hypexpo2}, \eqref{hypexpo4}, and 
\be\label{hypexpo1DISC}
\int_{t_1}^{\infty} e^{-\delta t} \beta(t)dt \leq C_0 e^{-\delta t_1} \sqrt{\alpha(t_1)\gamma(t_1)} 
\ee
and
\be\label{hypexpo3DISC}
\alpha(t_1)\leq C_0 \int_{t_1}^{\infty} e^{-\lambda (t-t_1)}\beta(t) dt + C_0 \sup_{t\geq t_1} e^{-\lambda(t-t_1)}\gamma(t). 
\ee

\begin{lem}\label{lem.expodecayDISC} There exists constant $C,\lambda', \delta_0>0$, depending on $C_0$ and $\lambda$ only, such that, for any $\delta\in (0, \delta_0)$ and $(\alpha,\beta, \gamma)\in \mathcal E(\delta)$, 
$$
\alpha(t)+\gamma(t)\leq Ce^{-\lambda' t} \gamma(0).
$$
\end{lem}

\begin{proof} The proof is a variant of the proof of Lemma \ref{lem.expodecay}. 
Throughout the proof, $C$ denotes a constant which depends on $\lambda$ and $C_0$ only (and thus {\it not} on $\delta$) and might change from line to line. For  $\delta\in (0,1)$ and $\tau \geq 1$, let 
$$
e^\delta(\tau)  = \sup\left\{e^{-\delta t} (\alpha(t)+\gamma(t)), \begin{array}{l}\; (\alpha,\beta, \gamma)\in \mathcal E(\delta), \; t\geq \tau\\
\alpha(0)+\gamma(0)\leq 1\end{array}\right\}. 
$$
Let us note that $\tau\to e^\delta(\tau)$ is nonincreasing. We first check that $e^\delta(\tau)$ is finite and tends to $0$ uniformly with respect to $\delta$ as $\tau\to\infty$. 

Fix $\delta\in (0,\lambda]$ and $\ep>0$. Let $(\alpha, \beta, \gamma)\in \mathcal E(\delta)$ be such that $\alpha(0)+\gamma(0)\leq 1$. By \eqref{hypexpo1DISC} and Cauchy-Schwarz inequality, 
\be\label{auekzhrjsdfngDISC}
\int_0^\infty e^{-\delta t} \beta(t)dt \leq C_0(\alpha(0)+\gamma(0))\leq C_0.
\ee
Thus by \eqref{hypexpo2}, for any $t\geq 0$, 
$$
e^{-\delta t}\gamma(t) \leq C_0 e^{-\lambda t} \gamma(0)+C_0 \int_{0}^{t}e^{-\delta s} \beta(s)dt \leq C.
$$
In addition,  by \eqref{hypexpo3DISC} and since $\delta \leq \lambda$,
\begin{align*}
e^{-\delta t} \alpha(t) & \leq C_0 \int_{t}^{\infty} e^{-\lambda (s-t)-\delta t}\beta(s) ds + C_0 \sup_{s\geq t} e^{-\lambda(s-t)-\delta t}\gamma(s)\\
& \leq C_0 \int_{t}^{\infty} e^{-\delta s}\beta(s) ds + C_0 \sup_{s\geq t} e^{-\delta s}\gamma(s) \leq C. 
\end{align*}
Thus $e^{-\delta t}\gamma(t)$ and $e^{-\delta t}\alpha(t)$ are bounded by some constant $C$ independent of $\delta$ on $[0,\infty)$. As  a consequence, $e^\delta(\tau)$ is bounded by a constant independent of $\delta$ for any $\tau\geq 1$. On the other hand, \eqref{auekzhrjsdfngDISC} implies the existence of $t_1\in (0,1/\ep)$ (possibly depending on $\delta$) such that $e^{-\delta t_1}\beta(t_1)\leq C_0\ep$.  Hence, by \eqref{hypexpo4}, $e^{-\delta t_1}\alpha(t_1)\leq C\ep$.  Coming back to \eqref{hypexpo1DISC}  we obtain that
$$
\int_{1/\ep}^{\infty} e^{-\delta t}\beta(t)dt \leq \int_{t_1}^{\infty} e^{-\delta t}\beta(t)dt \leq C_0e^{-\delta t_1/2}\sqrt{\alpha(t_1)}  \sup_{t\geq 0} (e^{-\delta t}\gamma(t))^{1/2} \leq C\ep^{1/2}. 
$$
We  then infer from \eqref{hypexpo2} that, for any $t\geq 2/\ep$ and any $\delta\in (0,1)$, 
$$
e^{-\delta t}\gamma(t) \leq C_0 e^{-\lambda(t-1/\ep)}e^{-\delta/\ep}\gamma(1/\ep) +C_0 \int_{1/\ep}^t e^{-\delta s} \beta(s)ds \leq C e^{-\lambda/\ep} + C\ep^{1/2}.
$$
Then, by \eqref{hypexpo3DISC}, we obtain that, for any $t\geq 2/\ep$ and any $\delta\in (0,\lambda]$,  
\begin{align*}
e^{-\delta t} \alpha(t) & \leq C_0 \int_{t}^{\infty} e^{-\lambda (s-t)-\delta t} \beta(s)ds+C_0\sup_{s\geq t} e^{-\lambda(s-t)-\delta t} \gamma(s)  \\
& \leq C_0 \int_{t}^{\infty} e^{-\delta s} \beta(s)ds+C_0\sup_{s\geq t} e^{-\delta s} \gamma(s)  \leq C\ep^{1/2}+ C( e^{-\lambda/\ep} + \ep^{1/2}).  
\end{align*}
We infer that, for any $t\geq 2/\ep$, 
$$
e^{-\delta t}(\alpha(t)+\gamma(t)) \leq C(\ep^{1/2}+ e^{-\lambda/\ep}). 
$$
As $(\alpha, \beta, \gamma)$ is arbitrary, this proves that, for any $\tau \geq 2/\ep$ and any $\delta\in (0,\lambda]$, 
$$
e^\delta (\tau) \leq C(\ep^{1/2}+ e^{-\lambda/\ep}). 
$$
From now on we fix $\ep>0$ such that $e^\delta (2/\ep)\leq 1/4$ and choose $\bar \delta_0\in(0, \lambda]$ such that $e^{4\bar \delta_0/\ep}\leq 2$. 

We now prove that, for any $\delta \in (0, \bar \delta_0]$ and $\tau= 2/\ep$,  
\be\label{azeioeufsdilmjokDISC}
e^\delta(2\tau) \leq \frac12 e^\delta(\tau). 
\ee
Let $\tau\geq 2/\ep$ and $(\alpha, \beta, \gamma)\in \mathcal E(\delta)$ with $\alpha(0)+\gamma(0)\leq 1$. Then, by the definition of $e^\delta(\tau)$,  for any $t\geq \tau$, 
$$
e^{-\delta t} (\alpha(t)+\gamma(t))\leq e^\delta(\tau).
$$
Note that, the map 
$$
(\tilde \alpha, \tilde \beta, \tilde \gamma)(t)= (\alpha, \beta, \gamma)(t+\tau)/(e^{\delta \tau} e^\delta (\tau))\qquad \forall t\geq 0
$$
belongs to $\mathcal E(\delta)$ and satisfies $\tilde \alpha(0)+\tilde \gamma(0)\leq 1$. Hence 
$$
e^{-\delta t}(\tilde \alpha(t)+\tilde \gamma(t)) \leq e^\delta(\tau) \qquad \forall t\geq \tau,
$$
which means that, for any $s\geq 2\tau$, 
$$
e^{-\delta (s-\tau)}( \alpha(s)+ \gamma(s)) \leq (e^\delta(\tau) )^2 e^{\delta \tau}.
$$
By the choice of $\ep$ and $\bar \delta_0$, $e^{2\delta \tau} e^\delta(\tau)\leq e^{4\bar \delta_0/\ep} e^\delta (2/\ep)\leq 1/2$, we conclude that 
$$
e^{-\delta s}( \alpha(s)+ \gamma(s)) \leq \frac12 e^\delta(\tau) \qquad \forall s\geq 2\tau, \; \forall (\alpha, \beta, \gamma)\in \mathcal E(\delta).
$$
Thus \eqref{azeioeufsdilmjokDISC} holds. We conclude as in the proof of Lemma \ref{lem.expodecay} that there exists $C>0$ and $\bar \lambda'>0$ such that 
$$
e^{\delta}(t)\leq Ce^{-\bar \lambda' t}, \qquad \forall t \geq 0. 
$$ 
Choosing $\lambda'=\bar \lambda'/2$ and $\delta_0= \bar \delta_0\vee (\lambda'/2)$ gives inequality 
$$
\alpha(t)+\gamma(t)\leq Ce^{-\lambda' t} (\alpha(0)+\gamma(0)).
$$

We now check that $\alpha(0)\leq C\gamma(0)$. Indeed, we have 
\begin{align*} 
\alpha(0)  & \leq C_0  \int_{0}^{\infty} e^{-\lambda t}\beta(t) dt + C_0 \sup_{t\geq 0} e^{-\lambda(t-t_1)}\gamma(t) \\ 
& \leq C_0^2(\alpha(0)\gamma(0))^{1/2} + C_0\sup_{t\geq 0}  e^{-\lambda t} \gamma(t) , 
\end{align*}
where, by \eqref{hypexpo2}, 
\begin{align*}
e^{-\lambda t} \gamma(t) &  \leq C_0 e^{-(\lambda+\lambda') t} \gamma(0)+C_0 \int_{0}^{t} e^{-\lambda s}\beta(s)ds\\ 
& \leq C_0 \gamma(0)+ C_0 \int_{0}^{\infty} e^{-\delta s}\beta(s)ds \leq C_0\gamma(0)+ C_0 (\alpha(0)\gamma(0))^{1/2}.
\end{align*}
Hence 
\begin{align*} 
\alpha(0)  & \leq C (\gamma(0)+ (\alpha(0)\gamma(0))^{1/2})
\end{align*}
which implies that $\alpha(0)\leq C\gamma(0)$. 
\end{proof}

%%%%%%%%%%%%%%%%%%%%%%%%%%%%%%%
%%%%%%%%%%%%%%%%%%%%%%%%%%%%%%%
%%%%%%%%%%%%%%%%%%%%%%%%%%%%%%%

\end{document}